\documentclass[11pt]{article}
\usepackage[small]{titlesec}
\usepackage[cm]{fullpage}

\usepackage{color}

\usepackage{amsmath,amscd, float,rotating}
\usepackage{pb-diagram}

\usepackage{latexsym}
\usepackage{amsfonts}
\usepackage{amsmath}
\usepackage{amssymb}

\numberwithin{equation}{section}
\newcommand{\qed}{\hfill \ensuremath{\Box}}

\def\XXint#1#2#3{{\setbox0=\hbox{$#1{#2#3}{\int}$}
\vcenter{\hbox{$#2#3$}}\kern-.5\wd0}}

\newcommand{\tr}[2]{\textrm{tr}_{#1} \, {#2}}
\newcommand{\ve}{\varepsilon}

\newcommand{\kod}{\textnormal{kod}}

\newcommand{\ofs}{\omega_{\textrm{FS}}}

\newcommand{\oddt}{\frac{d}{dt}}

\newcommand{\dbar}{\overline{\partial}}

\newcommand{\ddt}[1]{\frac{\partial #1}{\partial t}}

\newcommand{\Ric}{\textrm{Ric}}

\newcommand{\ov}[1]{\overline{#1}}

\newcommand{\Tmax}{\ensuremath{T_{\textrm{max}}}}
\newcommand{\oke}{\omega_{\textrm{KE}}}
\newcommand{\Mab}{\textrm{Mab}}

\newcommand{\PSH}{\textrm{PSH}(M, \omega_0)}

\newcommand{\KES}{\omega_{S}}
\newcommand{\KEE}{\omega_{E}}

\newcommand{\oflat}{\omega_{\textrm{flat}}}
\newcommand{\gflat}{g_{\textrm{flat}}}

\newcommand{\toke}{\tilde{\omega}_{\textrm{KE}}}

\newcommand{\ddbar}{\frac{\sqrt{-1}}{2\pi} \partial\dbar}

\newcommand{\Xcan}{X_{\textrm{can}}}

\begin{document}
\newtheorem{theorem}{Theorem}[section]
\newtheorem{remark}[theorem]{Remark}
\newtheorem{proposition}[theorem]{Proposition}
\newtheorem{lemma}[theorem]{Lemma}
\newtheorem{example}[theorem]{Example}
~

\newtheorem{definition}[theorem]{Definition}
\newtheorem{conjecture}[theorem]{Conjecture}
\newtheorem{corollary}[theorem]{Corollary}
\newenvironment{proof}[1][Proof]{\begin{trivlist}
\item[\hskip \labelsep {\bfseries #1}]}{\end{trivlist}}
~


\centerline{\bf \Large Lecture notes on the K\"ahler-Ricci flow\footnote{The first-named author was supported in part by an NSF CAREER grant
  DMS-08-47524  and the second-named author by the NSF grants DMS-08-48193 and DMS-11-05373.  Both authors were also supported in part by Sloan Research Fellowships.  This work was carried out while the second-named author was a member of the mathematics department at UC San Diego.}}

\bigskip
\bigskip

\centerline{\large \bf Jian Song$^{*}$ and
Ben Weinkove$^\dagger$}

\bigskip


\tableofcontents

\pagebreak

\subsection*{Introduction}
\addcontentsline{toc}{subsection}{Introduction}

The Ricci flow, first introduced by Hamilton \cite{H1} three decades ago, is the equation
\begin{equation} \label{hamiltonrf}
\ddt{} g_{ij} = - 2 R_{ij},
\end{equation}
evolving a Riemannian metric by its Ricci curvature.    It now occupies a central position as one of the key tools of geometry.   It was used in \cite{H1,H2} to classify 3-manifolds with positive Ricci curvature and $4$-manifolds with positive curvature operator.  Hamilton later introduced the notion of \emph{Ricci flow with surgery} \cite{H3} and laid out an ambitious program to prove the Poincar\'e and Geometrization conjectures.    In a spectacular demonstration of the power of the Ricci flow, Perelman \cite{P1, P2, P3} developed new techniques which enabled him to complete  Hamilton's program and settle these celebrated conjectures (see also \cite{CZ, KL, MT1, MT2}).  More recently, the Ricci flow was used to prove the Brendle-Schoen Differentiable Sphere Theorem \cite{BS1} and other geometric classification results \cite{BW,NW}.

In addition to these successes has been the development of the K\"ahler-Ricci flow.  If the Ricci flow starts from a K\"ahler metric on a complex manifold, the evolving metrics will remain K\"ahler, and the resulting PDE is called the K\"ahler-Ricci flow.  Cao \cite{Cao} used this flow, together with parabolic versions of the estimates of Yau \cite{Y1} and Aubin \cite{A}, to reprove the existence of K\"ahler-Einstein metrics on manifolds with negative and zero first Chern class.  Since then, the study of the K\"ahler-Ricci flow has developed into a vast field in its own right.  There have been several different avenues of research involving this flow, including: existence of K\"ahler-Einstein metrics on manifolds with positive first Chern class and  notions of algebraic stability \cite{B2, C2, CW, Do,  MS, P4, PSS, PSSW1, PSSW3, PS0, PS, PS2, R, SeT, Sz,  T1, TZhu, To1,   Y3, Zhu};  the classification of K\"ahler manifolds with positive curvature in both the compact and non-compact cases \cite{B1, Cao2, CZ0, CT, CST, ChT, Gu, M,N, PSSW2}; and extensions of the flow to non-K\"ahler settings \cite{G,ST}.  (These lists of references are far from exhaustive).  In these notes we will not even manage to touch on these areas.

Our main goal is to give an introduction to the K\"ahler-Ricci flow.  In the last two sections of the notes, we will also discuss some results related to the  \emph{analytic minimal model program} of the first-named author and Tian \cite{SoT1, SoT2, SoT3, TICM, T2}.
The field has been developing at a fast pace in the last several years, and we mention briefly now some of the ideas.


Ultimately, the goal is to see whether the K\"ahler-Ricci flow will give a geometric classification of algebraic varieties.  In the case of real 3-manifolds, the work of Perelman and Hamilton shows that the Ricci flow with surgery, starting at any Riemannian metric, can be used to break up the manifold into pieces, each of which has a particular geometric structure.  We can ask the same question for the K\"ahler-Ricci flow on a projective algebraic variety:  starting with any K\"ahler metric, will the K\"ahler-Ricci flow `with surgery' break up the variety into simpler pieces, each equipped with some canonical geometric structure?

A process of `simplifying' algebraic varieties through surgeries already exists and is known as the Minimal Model Program.  In the case of complex dimension two, the idea is relatively simple.  Start with a variety and find `$(-1)$-curves' - these are special holomorphic spheres embedded in the variety - and remove them using an algebraic procedure known as `blowing down'.  It can be shown that after a finite number of these algebraic surgeries, the final variety either has a `ruled' structure, or has \emph{nef canonical bundle}, a condition that can be interpreted as being `nonpositively curved' in some appropriate sense.  This last type of variety is known as a `minimal model'.  In higher dimensions, a similar, though more complicated, process also exists.
It turns out that there are many different ways to arrive at the minimal model by algebraic procedures such as blow-downs.  However, in \cite{BCHM} Birkar-Cascini-Hacon-McKernan introduced the notion of the \emph{Minimal Model Program with scaling} (or MMP with scaling), which, ignoring some technical assumptions, takes a variety with a `polarization' and describes a particular sequence of algebraic operations which take it to a minimal model or a ruled surface (or its higher dimensional analogue).  This process seems to be closely related to the K\"ahler-Ricci flow, with the polarization corresponding to a choice of initial K\"ahler metric.

Starting in 2007, Song-Tian \cite{SoT1, SoT2, SoT3} and Tian \cite{T2} proposed the analytic MMP using the K\"ahler-Ricci flow with a series of conjectures, and showed \cite{SoT3} that, in a weak sense, the flow can be continued through singularities related to the MMP with scaling. In the case of  complex dimension two, it was shown by the authors  \cite{SW2} that the algebraic procedure of `blowing down' a holomorphic sphere corresponds to a geometric `canonical surgical contraction' for the K\"ahler-Ricci flow.

 Moreover, the minimal model is endowed with an analytic structure.  Eyssidieux-Guedj-Zeriahi \cite{EGZ2} generalized an estimate of Kolodziej \cite{Kol1} (see also the work of Zhang \cite{Zha1}) to  construct singular K\"ahler-Einstein metrics  on minimal models of general type.  In the case of smooth minimal models,   convergence of the K\"ahler-Ricci flow to this metric was already known by the work of Tsuji in the 1980's \cite{Ts1}, results which were clarified and extended by Tian-Zhang  \cite{ TZha}.  On Iitaka fibrations, the K\"ahler-Ricci flow was shown by Song-Tian to converge to a `generalized K\"ahler-Einstein metric'  \cite{SoT1, SoT2}.

 These are very recent developments in a field which we expect is only just beginning.
    In these lecture notes we have decided to focus on describing  the main tools and techniques which are now well-established, rather than give expositions of the most recent advances.  In particular, we do not  in any serious way address `surgery' for the K\"ahler-Ricci flow and we only give a sketchy outline of the Minimal Model Program and its relation to the K\"ahler-Ricci flow.   On the other hand, we have taken the opportunity to include two new results in these notes:  a detailed description of collapsing along the K\"ahler-Ricci flow in the case of a product elliptic surface (Section \ref{sectpes}) and a description of the K\"ahler-Ricci flow on K\"ahler surfaces (Section \ref{sectlast}), extending our previous work on algebraic surfaces \cite{SW2}.

     We have aimed these notes at the non-expert and have tried to make them as self-contained and complete as possible.  We do not expect the reader to be either a geometric analyst or an algebraic geometer.  We  assume only a basic knowledge of complex manifolds.
  We hope that these notes will provide enough background material  for the non-expert reader to go on to begin research in this area.

We give now a brief outline of the contents of these notes.  In Section \ref{sectpre}, we give some preliminaries and background material on K\"ahler manifolds and curvature, describe some analytic tools such as the maximum principle, and provide some definitions and results from algebraic geometry.  Readers may wish to skip this section at first and refer back to it if necessary.
In Section \ref{sectgen}, we describe a number of well-known basic analytic results for the K\"ahler-Ricci flow.  Many of these results have their origin in the work of Calabi, Yau, Cheng,  Aubin  and others \cite{A, C1, Cao, ChY, Y1, Y2}.   We include a more recent argument, due to Phong-\v{S}e\v{s}um-Sturm \cite{PSS}, for the `Calabi third-order' estimate in the setting of the K\"ahler-Ricci flow.

 In Section \ref{maximal} we prove one of the basic results for the K\"ahler-Ricci flow:  the flow admits a smooth solution as long as the class of the metric remains K\"ahler.  The result in this generality is due to Tian-Zhang \cite{TZha}, extending earlier results of Cao and Tsuji \cite{Cao, Ts1, Ts2}.   In Section \ref{sectn0}, we give an exposition of  Cao's work \cite{Cao} - the first paper on the K\"ahler-Ricci flow.    Namely, we describe the behavior of the flow on manifolds with negative or zero first Chern class.  We include in this section the crucial $C^0$ estimate of Yau \cite{Y1}.  We give a different proof of convergence in this case, following Phong-Sturm \cite{PS}. In Section \ref{secttsuji}, we consider the K\"ahler-Ricci flow on manifolds with nef and big canonical bundle.  This was first studied by Tsuji \cite{Ts1} and demonstrates how one can study the singular behavior of the K\"ahler-Ricci flow.

 In Section \ref{sectpes}, we address the case of \emph{collapsing} along the K\"ahler-Ricci flow with the example of a product of an elliptic curve and a curve of higher genus.  In Section \ref{sectfinite}, we describe some basic results in the case where a singularity for the flow occurs at a finite time, including the recent result of Zhang \cite{Zha2} on the behavior of the scalar curvature.  We also describe without proof some of the results of \cite{SSW, SW2}.

In Section \ref{sectlast}, we discuss the K\"ahler-Ricci flow and the Minimal Model Program.  We give a brief sketch of some of the ideas of the MMP and how the K\"ahler-Ricci flow relates to it.  We also describe some results from \cite{SW2} and extend them to the case of K\"ahler surfaces.

The authors would like to mention that these notes arose from lectures given at the conference \emph{Analytic aspects of complex algebraic geometry}, held at the Centre International de Rencontres Math\'ematiques in Luminy, in February 2011.  The authors would like to thank  S. Boucksom, P. Eyssidieux and V. Guedj for  organizing this wonderful conference.
Additional thanks go to V. Guedj for his encouragement to write these notes.

We would also like to express our gratitude to the following people for providing helpful suggestions and corrections: Huai-Dong Cao, John Lott, Gang Tian, Valentino Tosatti and Zhenlei Zhang.
Finally the authors thank their former PhD advisor D.H. Phong for his valuable advice, encouragement and support over the last several years.  In addition, his teaching and ideas have had a strong influence on the style and point of view of these notes.

\pagebreak

\section{Preliminaries} \label{sectpre}

In this section we describe some definitions and results which will be used throughout the text.

\subsection{K\"ahler manifolds} \label{sectkahler}

Let $M$ be a compact complex manifold of complex dimension $n$.   We will often work in a holomorphic coordinate chart $U$ with coordinates $(z^1, \ldots, z^n)$ and write a tensor in terms of its components in such a coordinate system.  We refer the reader to \cite{GH, KM}  for an introduction to complex manifolds etc.

  A \emph{Hermitian metric} on $M$ is a smooth tensor $g=g_{i\ov{j}}$ such that $(g_{i \ov{j}})$ is a positive definite Hermitian matrix at each point of $M$.  Associated to $g$ is a $(1,1)$-form $\omega$ given by
\begin{equation}
\omega = \frac{\sqrt{-1}}{2\pi} g_{i \ov{j}} dz^i \wedge d\ov{z^{j}},
\end{equation}
where here and henceforth we are summing over repeated indices from $1$ to $n$.  If $d \omega=0$ then we say that $g$ is a \emph{K\"ahler} metric and that $\omega$ is the \emph{K\"ahler form} associated to $g$.   Henceforth, whenever, for example, $g(t),\hat{g},g_0, \ldots$ are K\"ahler metrics we will use the obvious notation $\omega(t),\hat{\omega},\omega_0, \ldots$  for the associated K\"ahler forms, and vice versa.  Abusing terminology slightly, we will often refer to a K\"ahler form  $\omega$ as a K\"ahler metric.

The K\"ahler condition $d\omega=0$ is equivalent to:
\begin{equation} \label{kahlercondition}
\partial_k g_{i \ov{j}} = \partial_i g_{k \ov{j}}, \quad \textrm{for all } i,j,k,
\end{equation}
where we are writing $\partial_i = \partial/\partial z^i$.  The condition (\ref{kahlercondition}) is independent of choice of holomorphic coordinate system.

For examples of K\"ahler manifolds, consider complex projective space $\mathbb{P}^n = (\mathbb{C}^{n+1}\setminus \{ 0 \} )/ \sim$ where $(z_0, \ldots, z_n) \sim (z'_0, \ldots, z'_n)$ if there exists $\lambda \in \mathbb{C}^*$ with $z_i = \lambda z'_i$ for all $i$.  We denote by $[Z_0, \ldots, Z_n] \in \mathbb{P}^n$ the equivalence class of $(Z_0, \ldots, Z_n) \in \mathbb{C}^{n+1} \setminus \{ 0 \}$.  Define the \emph{Fubini-Study metric} $\ofs$ by
\begin{equation}
\ofs = \ddbar \log (|Z_0|^2 + \cdots + |Z_n|^2).
\end{equation}
Note that although $|Z_0|^2 + \cdots + |Z_n|^2$ is not a well-defined function on $\mathbb{P}^n$, $\ofs$ is a well-defined $(1,1)$-form.  We leave it as an exercise for the reader to check that $\ofs$ is K\"ahler.  Moreover, since the restriction of a K\"ahler metric to a complex submanifold is K\"ahler, we can produce a large class of K\"ahler manifolds by considering complex submanifolds of $\mathbb{P}^n$.  These are known as \emph{smooth projective varieties}.

Let $X=X^i \partial_i$ and $Y = Y^{\ov{i}} \partial_{\ov{i}}$ be $T^{1,0}$ and $T^{0,1}$ vector fields respectively, and let $a = a_i dz^i$ and $b=b_{\ov{i}}d\ov{z^i}$ be $(1,0)$ and $(0,1)$ forms respectively.  By definition this means that if $(\tilde{z}^1, \ldots, \tilde{z}^n)$ is another  holomorphic coordinate system then on their overlap,
\begin{equation}
\tilde{X}^j = X^i \frac{\partial \tilde{z}^j}{\partial z^i}, \ \  \tilde{Y}^{\ov{j}} = Y^{\ov{i}} \ov{ \frac{\partial \tilde{z}^j}{\partial z^i}}, \ \ \tilde{a}_j = a_i \frac{\partial z^i}{\partial \tilde{z}^j}, \  \ \tilde{b}_{\ov{j}} = b_{\ov{i}} \ov{\frac{\partial z^{i}}{\partial \tilde{z}^j}}.
\end{equation}

Associated to a K\"ahler metric $g$ are \emph{covariant derivatives} $\nabla_k$ and $\nabla_{\ov{k}}$ which act on the tensors $X, Y, a, b$ in the following way:
\begin{align}
\nabla_k X^i & =  \partial_k X^i + \Gamma^i_{jk} X^j, \ \nabla_{\ov{k}} X^i = \partial_{\ov{k}} X^i, \ \nabla_k Y^{\ov{i}} = \partial_k Y^{\ov{i}}, \ \nabla_{\ov{k}} Y^{\ov{i}} = \partial_{\ov{k}} Y^{\ov{i}} + \ov{\Gamma^i_{jk}} Y^{\ov{j}}, \\
\nabla_k a_i & = \partial_k a_i - \Gamma^j_{ik} a_j, \ \nabla_{\ov{k}} a_i = \partial_{\ov{k}} a_i, \ \nabla_k b_{\ov{i}} = \partial_{k} b_{\ov{i}}, \ \nabla_{\ov{k}} b_{\ov{i}} = \partial_{\ov{k}} b_{\ov{i}} - \ov{\Gamma^j_{ik}} b_{\ov{j}},
\end{align}
where $\Gamma^i_{jk}$ are the \emph{Christoffel symbols} given by
\begin{equation}
\Gamma^i_{jk} = g^{\ov{\ell} i} \partial_{j} g_{k \ov{\ell}},
\end{equation}
for $(g^{\ov{\ell} i})$ the inverse of the matrix $(g_{i \ov{\ell}})$.  Observe that $\Gamma^i_{jk} = \Gamma^i_{kj}$ from (\ref{kahlercondition}).  The Christoffel symbols are not the components of a tensor,
 but  $\nabla_k X^i,\nabla_{\ov{k}} X^i, \ldots $ do define tensors, as the reader can verify.  Also, if $g$ and $\hat{g}$ are K\"ahler metrics with Christoffel symbols $\Gamma^i_{jk}$ and $\hat{\Gamma}^i_{jk}$ then the difference $\Gamma^i_{jk} - \hat{\Gamma}^i_{jk}$ is a tensor.

We extend covariant derivatives to act naturally on  any type of tensor.  For example, if $W$ is a tensor with components $W_{k}^{i \ov{j}}$ then define
\begin{equation}
\nabla_m W_{k}^{i \ov{j}} = \partial_m W_k^{i \ov{j}} + \Gamma^i_{\ell m} W^{\ell \ov{j}}_k - \Gamma^{\ell}_{km} W^{i \ov{j}}_{\ell}, \ \ \nabla_{\ov{m}} W_k^{i \ov{j}} = \partial_{\ov{m}} W_k^{i \ov{j}} + \ov{\Gamma^j_{\ell m}} W_k^{i \ov{\ell}}.
\end{equation}
Note also that the Christoffel symbols are chosen so that  $\nabla_k g_{i \ov{j}}=0$.

The metric $g$ defines a pointwise norm $| \cdot |_g$ on any tensor.  For example, with $X, Y, a, b$ as above, we define
\begin{equation}
| X|_g^2 = g_{i \ov{j}} X^i \ov{X^{j}}, \quad |Y|_g^2 = g_{i \ov{j}} Y^{\ov{j}}  \, \ov{ Y^{\ov{i}}}, \quad |a|_g^2 = g^{\ov{j}i} a_i \ov{a_j}, \quad |b|^2_g = g^{\ov{j} i} b_{\ov{j}} \ov{b_{\ov{i}}}.
\end{equation}
This is  extended to any type of tensor.  For example, if $W$ is a tensor with components $W^{i\ov{j}}_{k}$ then define $|W|^2_g = g^{\ov{\ell} k} g_{i \ov{j}} g_{p \ov{q}} W^{i\ov{q}}_{k} \ov{W^{j \ov{p}}_{\ell}}$.

Finally, note that a K\"ahler metric $g$ defines a Riemannian metric $g_{\mathbb{R}}$.  In local coordinates, write $z^i = x^i + \sqrt{-1} y^i$, so that $\partial_{z^i} = \frac{1}{2} (\partial_{x^i} - \sqrt{-1} \partial_{y^i})$ and $\partial_{\ov{z^i}} = \frac{1}{2} (\partial_{x^i} + \sqrt{-1} \partial_{y^i})$.  Then
\begin{equation} \label{gR}
g_{\mathbb{R}} (\partial_{x^i}, \partial_{x^j})  = 2 \textrm{Re}( g_{i \ov{j}}) = g_{\mathbb{R}} (\partial_{y^i}, \partial_{y^j}), \quad g_{\mathbb{R}} (\partial_{x^i}, \partial_{y^j})  = 2 \textrm{Im}( g_{i \ov{j}}).
\end{equation}
We will typically write $g$ instead of $g_{\mathbb{R}}$.


\subsection{Normal coordinates}

The following proposition is very useful in computations.

\begin{proposition} \label{propnormal}  Let $g$ be a K\"ahler metric on $M$ and let $S= S_{i\ov{j}}$ be a tensor which is Hermitian (that is $\ov{S_{i \ov{j}}} = S_{\ov{j} i}$.)  Then at each point $p$ on $M$ there exists a holomorphic coordinate system centered at $p$ such that,
\begin{equation}
g_{i\ov{j}}(p) = \delta_{ij}, \quad S_{i \ov{j}}(p) = \lambda_i \delta_{i j}, \quad \partial_k g_{i \ov{j}}(p) =0, \quad \forall \, i,j,k=1, \ldots, n,
\end{equation}
for some $\lambda_1, \ldots, \lambda_n \in \mathbb{R}$, where $\delta_{ij}$ is the Kronecker delta.
\end{proposition}
\begin{proof}
It is an exercise in linear algebra to check that we can find a coordinate system $(z^1, \ldots, z^n)$ centered at $p$ (so that $p \mapsto 0$) satisfying the first two conditions: $g$ is the identity at $p$ and $S$ is diagonal at $p$.  Indeed this amounts to the fact that a Hermitian matrix can be diagonalized by a unitary transformation.

For the last condition we make a change of coordinates.  Define coordinates  $(\tilde{z}^1, \ldots, \tilde{z}^n)$ in a neighborhood of $p$ by
\begin{equation}
z^i = \tilde{z}^i - \frac{1}{2}\Gamma^i_{jk}(0) \tilde{z}^j \tilde{z}^k, \quad \textrm{for } i=1, \ldots, n.
\end{equation}
Writing $\tilde{g}_{i \ov{j}}$, $\tilde{S}_{i\ov{j}}$ for the components of $g$, $S$ with respect to $(\tilde{z}^1, \ldots, \tilde{z}^n)$ we see that $\tilde{g}_{i\ov{j}}(0) = g_{i \ov{j}}(0)$ and $\tilde{S}_{i \ov{j}} (0) = S_{i\ov{j}}(0)$ since $\partial z^i/\partial \tilde{z}^j (0) = \delta_{ij}$.  It remains to check that the first derivatives of $\tilde{g}_{i \ov{j}}$ vanish at $0$.
Compute at 0,
\begin{align} \nonumber
\frac{\partial}{\partial \tilde{z}^k} \tilde{g}_{i \ov{j}}  & = \frac{\partial}{\partial \tilde{z}^k} \left( \frac{\partial z^a}{\partial \tilde{z}^i} \ov{ \frac{\partial z^b}{\partial \tilde{z}^j}} g_{a \ov{b}} \right) \\ \nonumber
& =  \frac{\partial^2 z^a}{\partial \tilde{z}^k \partial \tilde{z}^i} \ov{ \frac{\partial z^b}{\partial \tilde{z}^j}} g_{a \ov{b}} + \frac{\partial z^a}{\partial \tilde{z}^i} \ov{ \frac{\partial z^b}{\partial \tilde{z}^j}} \frac{\partial z^m}{\partial \tilde{z}^k} \frac{\partial}{\partial z^m} g_{a \ov{b}}   \\
& = - \Gamma^j_{ik}  + \frac{\partial}{\partial z^k} g_{i \ov{j}} =0,
\end{align}
as required. \qed
\end{proof}

We call a holomorphic coordinate system centered at $p$  satisfying $g_{i \ov{j}}(p) = \delta_{ij}$ and $\partial_k g_{i \ov{j}}(p) =0$ a \emph{normal coordinate system} for $g$ centered at $p$.  It implies in particular that the Christoffel symbols of $g$ vanish at $p$.  Proposition \ref{propnormal} states that we can find a normal coordinate system for $g$ at any point $p$, and that moreover we can simultaneously diagonalize any other Hermitian tensor (such as another K\"ahler metric) at that point.

\subsection{Curvature} \label{sectcurv}

Define the \emph{curvature tensor} of the K\"ahler  metric $g$ to be the tensor
\begin{equation}
R^{\  m}_{i \ \, \, k \ov{\ell}} = - \partial_{\ov{\ell}} \Gamma^{m}_{ik}.
\end{equation}
The reader can verify that this does indeed define a tensor on $M$.  We often lower the second index using the metric $g$ and define
\begin{equation}
R_{i \ov{j} k \ov{\ell}} = g_{m \ov{j}} R^{\  m}_{i \ \, \, k \ov{\ell}},
\end{equation}
an object which we also refer to as the curvature tensor.   In addition, we can lower or raise any index of curvature using the metric $g$.  For example, $R_{i \ov{j}}^{\ \ \, \ov{q} p} := g^{\ov{q}k} g^{\ov{\ell} p} R_{i \ov{j}k \ov{\ell}}$.

Using the formula for the Christoffel symbols and (\ref{kahlercondition}), calculate
\begin{equation} \label{formulacurvature}
R_{i \ov{j} k \ov{\ell}} = - \partial_ i \partial_{\ov{j}} g_{k \ov{\ell}} + g^{\ov{q}p} (\partial_i g_{k \ov{q}})(\partial_{\ov{j}} g_{p \ov{\ell}}).
\end{equation}

The curvature tensor has a number of symmetries:

\begin{proposition} \label{Rsymmetry}
We have
\begin{enumerate}
\item[(i)] $\displaystyle{\ov{R_{i \ov{j} k \ov{\ell}}} = R_{j \ov{i} \ell \ov{k}}}$.
\item[(ii)] $\displaystyle{R_{i \ov{j} k \ov{\ell}} = R_{k \ov{j} i \ov{\ell}} = R_{i \ov{\ell} k \ov{j}}}$.
\item[(iii)] $\displaystyle{\nabla_m R_{i \ov{j} k \ov{\ell}} = \nabla_i R_{m \ov{j} k \ov{\ell}}}$.
\end{enumerate}
\end{proposition}
\begin{proof}
(i) and (ii) follow immediately from the formula (\ref{formulacurvature}) together with the K\"ahler condition (\ref{kahlercondition}).  For (iii) we compute at a point $p$ in normal coordinates for $g$,
\begin{equation}
\nabla_m R_{i \ov{j} k \ov{\ell}} = - \partial_m \partial_i \partial_{\ov{j}} g_{k \ov{\ell}} = - \partial_i \partial_m \partial_{\ov{j}} g_{k \ov{\ell}} = \nabla_i R_{m \ov{j} k \ov{\ell}},
\end{equation}
as required.  \qed
\end{proof}

Parts (ii) and (iii) of Proposition \ref{Rsymmetry} are often referred to as the \emph{first and second Bianchi identities},  respectively.  Define the \emph{Ricci curvature} of $g$ to be the tensor
\begin{equation}
R_{i \ov{j}} : = g^{\ov{\ell} k} R_{i \ov{j} k \ov{\ell}} = g^{\ov{\ell} k} R_{k \ov{\ell} i \ov{j}} = R^{\ \, k}_{k \ \,  i \ov{j}},
\end{equation}
and the \emph{scalar curvature} $R= g^{\ov{j} i} R_{i \ov{j}}$ to be the trace of the Ricci curvature.   For K\"ahler manifolds, the Ricci curvature takes on a simple form:

\begin{proposition} \label{Rcformula} We have
\begin{equation}
R_{i \ov{j}}  = - \partial_i \partial_{\ov{j}} \log \det g.
\end{equation}
\end{proposition}
\begin{proof}
First, recall the well-known formula for the derivative of the determinant of a Hermitian matrix.  Let $A = (A_{i \ov{j}})$ be an invertible Hermitian matrix with inverse $(A^{\ov{j} i})$.  If the entries of $A$ depend on a variable $s$ then an application of Cramer's rule shows that
\begin{equation}
\frac{d}{ds} \det A = A^{\ov{j} i}  \left( \frac{d}{ds} A_{i \ov{j}} \right) \det A.
\end{equation}
Using this, calculate
\begin{equation}
R_{i \ov{j}} = - \partial_{\ov{j}} \Gamma^k_{ki} = - \partial_{\ov{j}} (g^{\ov{q} k} \partial_i g_{k \ov{q}}) = - \partial_{\ov{j}} \partial_{i} \log \det g,
\end{equation}
which gives the desired formula.  \qed
\end{proof}

Associated to the tensor $R_{i\ov{j}}$ is a $(1,1)$ form $\Ric(\omega)$ given by
\begin{equation}
\Ric(\omega) = \frac{\sqrt{-1}}{2\pi} R_{i \ov{j}} dz^i \wedge d\ov{z^j}.
\end{equation}
Proposition \ref{Rcformula} implies that $\Ric(\omega)$ is closed.

We end this subsection by showing that the curvature tensor arises when commuting covariant derivatives $\nabla_k$ and $\nabla_{\ov{\ell}}$.  Indeed, the curvature tensor is often defined by this property.

\begin{proposition} \label{commformulae} Let $X= X^i \partial_i$, $Y=Y^{\ov{i}} \partial_{\ov{i}}$ be $T^{1,0}$ and $T^{0,1}$ vector fields respectively, and let $a=a_idz^i$ and $b=b_{\ov{i}} d\ov{z^i}$ be
$(1,0)$ and $(0,1)$ forms respectively.  Then
\begin{align}
[ \nabla_k, \nabla_{\ov{\ell}} ] X^m & = R^{\ m}_{i \ \, \, k \ov{\ell}} X^i \\
[ \nabla_k, \nabla_{\ov{\ell}} ] Y^{\ov{m}} & = - R^{\ov{m}}_{\ \, \, \ov{j} k \ov{\ell}} Y^{\ov{j}} \\
[ \nabla_k, \nabla_{\ov{\ell}} ] a_i & = - R^{\ m}_{i \ \, \, k \ov{\ell}} a_m \\
[ \nabla_k, \nabla_{\ov{\ell}} ] b_{\ov{j}} & = R^{\ov{m}}_{\ \, \ov{j} k \ov{\ell}} b_{\ov{m}},
\end{align}
where we are writing $[ \nabla_k, \nabla_{\ov{\ell}} ] = \nabla_k \nabla_{\ov{\ell}} - \nabla_{\ov{\ell}} \nabla_k$.
\end{proposition}
\begin{proof}
We prove the first and leave the other three as exercises.  Compute at a point $p$ in a normal coordinate system for $g$,
\begin{align}
[ \nabla_k, \nabla_{\ov{\ell}} ] X^m & = \partial_k \partial_{\ov{\ell}} X^m - \partial_{\ov{\ell}} (\partial_k X^m + \Gamma^m_{ki} X^i)
 =  - (\partial_{\ov{\ell}}  \Gamma^m_{ik}) X^i = R^{\ m}_{i \ \, \, k \ov{\ell}} X^i,
\end{align}
as required.  \qed
\end{proof}

Note that the commutation formulae of Proposition \ref{commformulae} can naturally be extended to tensors of any type.  Finally we remark that, when acting on any tensor, we have $[\nabla_i, \nabla_j] =  0 = [ \nabla_{\ov{i}}, \nabla_{\ov{j}}]$, as the reader can  verify.

\subsection{The maximum principle} \label{sectmax}

There are various notions of `maximum principle'.  In the setting of the Ricci flow, Hamilton introduced his \emph{maximum principle for tensors} \cite{CLN, H1, H3} which has been exploited in quite sophisticated ways to investigate the positivity of curvature tensors along the flow (see for example \cite{B1, BW, BS1,H2,M, NW}).    For our purposes however, we need only a simple version of the maximum principle.

 We begin with an elementary lemma.  As above, $(M, \omega)$ will be a compact K\"ahler manifold.

\begin{proposition} \label{sdt}
Let $f$ be a smooth real-valued function on $M$ which achieves its maximum (minimum) at a point $x_0$ in $M$.  Then at $x_0$,
\begin{equation} \label{sdte}
d f =0 \quad \textrm{and} \quad \sqrt{-1} \partial \overline{\partial} f  \le 0 \ (\ge 0).
\end{equation}
\end{proposition}

Here, if $\alpha= \sqrt{-1} a_{i \ov{j}} dz^i \wedge d\overline{z^j}$ is a real $(1,1)$-form, we write $\alpha \le 0 \ (\ge 0)$ to mean that the Hermitian matrix $(a_{i \ov{j}})$ is nonpositive (nonnegative). Proposition \ref{sdt} is a simple consequence of the fact  from calculus that a smooth function has nonpositive Hessian (and hence nonpositive complex Hessian) and zero first derivative at its maximum.

Next we introduce the Laplace operator $\Delta$ on functions.  Define
\begin{equation}
\Delta f = g^{\ov{j}i} \partial_i \partial_{\ov{j}} f
\end{equation}
for a function $f$.

 In these lecture notes, we will often make use of the trace notation `$\textrm{tr}$'.  If $\alpha= \frac{\sqrt{-1}}{2\pi} a_{i \ov{j}} dz^i \wedge d\overline{z^j}$ is a real $(1,1)$-form then we write
 \begin{equation}
 \tr{\omega}{\alpha} = g^{\ov{j} i} a_{i \ov{j}} = \frac{n \, \omega^{n-1} \wedge \alpha}{\omega^n}.
 \end{equation}
 In this notation, we can write $\Delta f = \tr{\omega}{\ddbar f}$.

It follows immediately from this definition that Proposition \ref{sdt} still holds if we replace $\sqrt{-1}\partial \ov{\partial} f \le 0 \ (\ge 0)$ in (\ref{sdte}) by $\Delta f  \le 0 \ (\ge 0)$.

For the \emph{parabolic maximum principle} (which we still call the \emph{maximum principle}) we
introduce a time parameter $t$.  The following proposition will be used many times in these lecture notes.

\begin{proposition} \label{pmp} Fix $T>0$.
Let $f=f(x,t)$ be a smooth function on $M \times [0,T]$.  If $f$ achieves its maximum (minimum) at $(x_0, t_0) \in M \times [0,T ]$ then either $t_0=0$ or at $(x_0, t_0)$,
\begin{equation} \label{pmpe}
\frac{\partial f}{\partial t} \ge 0 \ (\le 0) \quad \textrm{and} \quad d f =0 \quad \textrm{and} \quad \sqrt{-1} \partial \overline{\partial} f  \le 0 \ (\ge 0).
\end{equation}
\end{proposition}
\begin{proof}
Exercise for the reader.  \qed
\end{proof}

We remark that, in practice, one is usually given a function $f$ defined on a half-open time interval $[0,T)$ say, rather than a compact interval.  To apply
this proposition it may be necessary to fix an arbitrary $T_0 \in (0, T)$ and work on $[0,T_0]$.  Since we use this procedure many times in the notes, we will often omit to mention the fact that we are restricting to such a compact interval.
Note also that Propositions \ref{sdt} and \ref{pmp} still hold with $M$ replaced by an open set $U \subseteq M$ as long as the maximum (or minimum) of $f$ is achieved in the interior of the set $U$.

We end this section with a useful application of the maximum principle in the case where $f$ satisfies a heat-type differential inequality.

\begin{proposition} \label{propheat} Fix $T$ with $0 < T \le \infty$.
Suppose that $f=f(x,t)$ is a smooth function on $M \times [0,T)$ satisfying the differential inequality
\begin{equation} \label{diffineq}
\left( \ddt{} - \Delta \right) f \le 0.
\end{equation}
Then $\sup_{(x,t) \in M \times [0, T)} f(x,t) \le \sup_{x\in M} f(x,0).$
\end{proposition}
\begin{proof}
 Fix $T_0 \in (0,T)$.  For $\ve>0$, define $f_{\ve} = f-\ve t$.  Suppose that $f_{\ve}$ on $M \times [0,T_0]$ achieves its maximum at $(x_0, t_0)$.  If $t_0>0$ then by Proposition \ref{pmp},
\begin{equation}
0 \le \left( \ddt{} - \Delta \right) f_{\ve} \, (x_0, t_0) \le - \ve,
\end{equation}
a contradiction.  Hence the maximum of $f_{\ve}$ is achieved at $t_0=0$ and
\begin{equation}
\sup_{(x,t) \in M \times [0, T_0]} f(x,t) \le
\sup_{(x, t) \in M \times [0, T_0]} f_{\ve} (x,t) + \ve T_0 \le \sup_{x \in M} f(x, 0) + \ve T_0.
\end{equation}
Let $\ve \rightarrow 0$.  Since $T_0$ is arbitrary, this proves the result. \qed
\end{proof}

We remark that a similar result of course holds for the infimum of $f$ if we replace $\left( \ddt{} - \Delta \right) f \le 0$ by $\left( \ddt{} - \Delta \right) f \ge 0$.  Finally, note that  Proposition \ref{propheat} holds, with the same proof, if the Laplace operator $\Delta$ in (\ref{diffineq}) is defined with respect to a metric $g=g(t)$ that depends on $t$.

\subsection{Other analytic results and definitions}

In this subsection, we list a number of other results and definitions from analysis, besides the maximum principle, which we will need later.  For a good reference, see \cite{A2}.
Let $(M, \omega)$ be a compact K\"ahler manifold of complex dimension $n$.  In these lecture notes, we will be concerned only with smooth functions and tensors so for the rest of this section assume that all functions and tensors on $M$ are smooth.  The following is known as the \emph{Poincar\'e inequality}.

\begin{theorem} \label{poincare}
There exists a constant $C_P$ such that for any real-valued function $f$ on $M$ with $\int_M f \omega^n=0$, we have
\begin{equation}
\int_M f^2 \omega^n \le C_P \int_M | \partial f|^2 \omega^n,
\end{equation}
for $| \partial f|^2 = g^{\ov{j} i} \partial_i f \partial_{\ov{j}} f$.
\end{theorem}

We remark that the constant $C_P$ is (up to scaling by some universal factor) equal to $\lambda^{-1}$ where $\lambda$ is the first nonzero eigenvalue of the operator $-\Delta$ associated to $g$.

Next, we have the \emph{Sobolev inequality}.

\begin{theorem}  \label{sobolev} Assume $n>1$.
There exists a uniform constant $C_S$ such that for any real-valued function $f$ on $M$, we have
\begin{equation}
\left( \int_M |f|^{2\beta} \omega^n \right)^{1/\beta} \le C_S \left( \int_M | \partial f|^2 \omega^n + \int_M |f|^2 \omega^n \right),
\end{equation}
for $\beta = n/(n-1)>1$.
\end{theorem}

We give now some definitions for later use.     Given a function $f$, define the $C^0$  norm on $M$ to be $\| f \|_{C^0(M)} = \sup_M |f|$.  We give a similar definition for any subset $U \subset M$.
Given a (real) tensor $W$ and a Riemannian metric $g$, we define $|W|_g^2$ by contracting with $g$, in the obvious way (cf. Section \ref{sectkahler}).  Define $\| W \|_{C^0(M,g)}$ to be the $C^0(M)$ norm of $|W|_g$. If no confusion arises, we will often drop the $M$ and $g$ in denoting norms.

Given a function $f$ on $M$, we define for $p \ge 1$ the $L^p(M,\omega)$ norm with respect to a K\"ahler metric $\omega$ by
\begin{equation}
 \| f\|_{L^p(M, \omega)} = \left( \int_M |f|^p \omega^n \right)^{1/p}.
\end{equation}
Note that $\| f\|_{L^p(M,\omega)} \rightarrow \| f\|_{C^0(M)}$ as $p \rightarrow \infty$.

 We use $\nabla_{\mathbb{R}}$ to denote the (real) covariant derivative of $g$.  Given a function $f$, write $\nabla_{\mathbb{R}}^m f$ for the tensor with components (in real coordinates) $(\nabla_{\mathbb{R}})_{i_1} \cdots (\nabla_{\mathbb{R}})_{i_m} f$ and similarly for $\nabla$ acting on tensors.

For a function $f$ and a subset $U \subseteq M$, define
\begin{equation}
\| f \|_{C^k(U, g)} = \sum_{m=0}^k \| \nabla_{\mathbb{R}}^m f \|_{C^0(U,g)},
\end{equation}
and similarly for tensors.

We say that a tensor $T$ has \emph{uniform $C^{\infty}(M,g)$ bounds} if for each $k=0, 1,2, \ldots$ there exists a uniform constant $C_k$ such that $\| T \|_{C^k(M,g)} \le C_k$.  Given an open subset $U \subseteq M$ we say that $T$ has \emph{uniform $C^{\infty}_{\emph{loc}}(U,g)$ bounds} if for any compact subset $K \subseteq U$ there exist constants $C_{k,K}$ such that $\| T \|_{C^k(K,g)} \le C_{k,K}$.  We say that a family of tensors $T_t$ \emph{converges in} $C^{\infty}_{\textrm{loc}}(U,g)$ to a tensor $T_{\infty}$ if for every compact $K \subseteq U$, and each $k=0,1,2, \ldots$, the tensors $T_t$ converge to $T_{\infty}$ in $C^k(K,g)$.

Given $\beta \in (0,1)$, define the H\"older norm $C^{\beta}(M, g)$ of a function $f$ by
\begin{equation}
\| f\|_{C^{\beta}(M, g)} = \| f\|_{C^0(M)} + \sup_{p\neq q} \frac{ |f(p)-f(q)|}{d(p,q)^{\beta}},
\end{equation}
for $d$ the distance function of $g$.  The $C^{\beta}(M,g)$ norm for tensors $T$ is defined similarly, except that we must use parallel transport with respect to $g$ construct the difference $T(p)-T(q)$.  For a positive integer $k$, define $\| f \|_{C^{k+\beta}(M,g)} = \| f \|_{C^k(M,g)} + \| \nabla_{\mathbb{R}}^k f \|_{C^{\beta}(M,g)}$, and similarly for tensors.

Finally, we define what is meant by \emph{Gromov-Hausdorff convergence}.  This is a notion of convergence for metric spaces.  Given two subsets $A$ and $B$ of a metric space $(X,d)$, we define the \emph{Hausdorff distance} between $A$ and $B$ to be
\begin{equation}
d_{\textrm{H}}(A, B) = \inf \{ \ve>0 \ | \ A \subseteq B_{\ve} \ \textrm{and} \ B \subseteq A_{\ve} \}
\end{equation}
where $A_{\ve}= \cup_{a \in A} \{ x \in X \ | \ d(a,x) \le \ve \}.$  We then define the \emph{Gromov-Hausdorff distance} between two compact metric spaces $X$ and $Y$ to be
\begin{equation}
d_{\textrm{GH}}(X,Y) = \inf_{f,g}  d_{\textrm{H}} (f(X), g(Y)),
\end{equation}
where the infimum is taken over all isometric embeddings $f: X \rightarrow Z$, $g: Y \rightarrow Z$ into a metric space $Z$ (for all possible $Z$).  We then say that a family $X_t$ of compact metric spaces \emph{converges in the Gromov-Hausdorff sense} to a compact metric space $X_{\infty}$ if the $X_t$ converge to $X_{\infty}$ with respect to $d_{\textrm{GH}}$.

\subsection{Dolbeault cohomology, line bundles and divisors} \label{sectdol}

In this section we introduce cohomology classes, line bundles, divisors, Hermitian metrics etc.  Good references for this and the next subsection are \cite{GH, KM}.
Let $M$ be a compact complex manifold.  We say that a form $\alpha$ is \emph{$\ov{\partial}$-closed} if $\ov{\partial} \alpha =0$ and \emph{$\ov{\partial}$-exact} if $\alpha = \ov{\partial} \eta$ for some form $\eta$.
Define the \emph{Dolbeault cohomology group} $H_{\ov{\partial}}^{1,1}(M, \mathbb{R})$ by
\begin{equation}
H_{\ov{\partial}}^{1,1}(M, \mathbb{R}) = \frac{ \{ \ov{\partial}\textrm{-closed real (1,1)-forms} \}}{ \{ \ov{\partial}\textrm{-exact real (1,1)-forms} \}}.
\end{equation}
A K\"ahler metric $\omega$ on $M$ defines a nonzero element $[\omega]$ of $H_{\ov{\partial}}^{1,1}(M, \mathbb{R})$.    If a cohomology class $\alpha \in  H_{\ov{\partial}}^{1,1}(M, \mathbb{R})$ can be written $\alpha = [\omega]$ for some K\"ahler metric $\omega$ then we say that $\alpha$ is a \emph{K\"ahler class} and  write $\alpha>0$.

 A basic result of K\"ahler geometry is the $\partial\ov{\partial}$-Lemma.

\begin{theorem}
Let $(M, \omega)$ be a compact K\"ahler manifold.  Suppose that $0= [\alpha] \in H_{\ov{\partial}}^{1,1}(M, \mathbb{R})$ for a real smooth $\ov{\partial}$-closed $(1,1)$-form $\alpha$.  Then there exists a real-valued smooth function $\varphi$ with $\alpha = \ddbar \varphi$, which is uniquely determined up to the addition of a constant.
\end{theorem}

In other words, a real $(1,1)$-form $\alpha$ is $\ov{\partial}$-exact if and only if it is $\partial\ov{\partial}$-exact.  It is an immediate consequence of the  $\partial\ov{\partial}$-Lemma that if $\omega$ and $\omega'$ are K\"ahler metrics in the same K\"ahler class then $\omega' = \omega+ \ddbar \varphi$ for some smooth function $\varphi$, which is uniquely determined up to a constant, and sometimes referred to as a \emph{(K\"ahler) potential}.

A \emph{line bundle} $L$ over $M$ is given by an open cover $\{ U_{\alpha} \}$ of $M$ together with collection of \emph{transition functions} $\{ t_{\alpha \beta} \}$ which are holomorphic maps $t_{\alpha \beta}: U_{\alpha} \cap U_{\beta} \rightarrow \mathbb{C}^*$ satisfying
\begin{equation} \label{cocycle}
t_{\alpha \beta} t_{\beta \alpha}=1, \quad t_{\alpha \beta} t_{\beta \gamma} = t_{\alpha \gamma}.
\end{equation}
We identify two such collections of transition functions $\{ t_{\alpha\beta} \}$ and $\{ t'_{\alpha \beta} \}$ if we can find holomorphic functions $f_{\alpha} : U_{\alpha} \rightarrow \mathbb{C}^*$ with $t'_{\alpha\beta} = \frac{f_{\alpha}}{f_{\beta}} t_{\alpha \beta}$.  (In addition, we also need to identify $( \{ U_{\alpha} \}, \{ t_{\alpha \beta} \})$, $( \{U'_{\gamma} \}, \{ t'_{\gamma \delta} \})$ whenever $\{ U'_{\gamma} \}$ is a refinement of $\{ U_{\alpha} \}$ and the $t'_{\gamma\delta}$ are restrictions of the $t_{\alpha \beta}$.  We will not dwell on  technical details about refinements etc and instead refer the reader to \cite{GH} or \cite{KM}.)  Given line bundles $L, L'$ with transition functions $\{ t_{\alpha \beta} \}$, $\{ t'_{\alpha \beta} \}$ write $LL'$ for the new line bundle with transition functions $\{ t_{\alpha \beta} t'_{\alpha \beta}\}$.  Similarly, for any $m \in \mathbb{Z}$, we  define line bundles $L^m$ by $\{ t_{\alpha \beta}^m \}$.  We call $L^{-1}$ the inverse of $L$.     Sometimes we use the additive notation for line bundles, writing $L+L'$ for $LL'$ and $mL$ for $L^m$.

A \emph{holomorphic section} $s$ of $L$ is a collection $\{ s_{\alpha} \}$ of holomorphic maps $s_{\alpha} : U_{\alpha} \rightarrow \mathbb{C}$ satisfying the transformation rule $s_{\alpha} = t_{\alpha \beta} s_{\beta}$ on $U_{\alpha} \cap U_{\beta}$.  A \emph{Hermitian metric} $h$ on $L$ is a collection $\{ h_{\alpha} \}$ of smooth positive functions $h_{\alpha}: U_{\alpha} \rightarrow \mathbb{R}$ satisfying the transformation rule $h_{\alpha} = |t_{\beta \alpha}|^2 h_{\beta}$ on $U_{\alpha} \cap U_{\beta}$.   Given a holomorphic section $s$ and a Hermitian metric $h$, we can define the \emph{pointwise norm squared} of $s$ with respect to $h$ by  $|s|^2_h = h_{\alpha} s_{\alpha} \ov{s_{\alpha}}$ on $U_{\alpha}$.  The reader can check that $|s|^2_h$ is a well-defined function on $M$.

We define the \emph{curvature} $R_h$ of a Hermitian metric $h$ on $L$ to be the closed $(1,1)$ form on $M$ given by $R_h = - \ddbar \log h_{\alpha}$ on $U_{\alpha}$.  Again, we let the reader check that this is well-defined.
Define the \emph{first Chern class} $c_1(L)$ of $L$ to be the cohomology class $[R_h] \in H^{1,1}_{\ov{\partial}}(M, \mathbb{R})$.  Since any two Hermitian metrics $h, h'$ on $L$ are related by $h' = h e^{-\varphi}$ for some smooth function $\varphi$, we see that $R_{h'} = R_h + \ddbar \varphi$ and hence $c_1(L)$ is well-defined independent of choice of Hermitian metric.  Note that if $h$ is a Hermitian metric on $L$ then $h^m$ is a Hermitian metric on $L^m$ and $c_1(L^m) = m c_1(L)$.

Every complex manifold $M$ is equipped with a line bundle $K_M$, known as the \emph{canonical bundle}, whose transition functions are given by $t_{\alpha \beta} = \det \left( \partial z_{\beta}^i/{\partial z_{\alpha}^j} \right)$ on $U_{\alpha}\cap U_{\beta}$, where $U_{\alpha}$ are coordinate charts for $M$ with coordinates $z^1_{\alpha}, \ldots, z^n_{\alpha}$.  If $g$ is a K\"ahler metric (or more generally, a Hermitian metric) on $M$ then $h_{\alpha}=\det (g^{\alpha}_{i \ov{j}} )^{-1}$ on $U_{\alpha}$ defines a Hermitian metric on $K_M$.  The inverse $K_M^{-1}$ of $K_M$ is sometimes called the \emph{anti-canonical bundle}.  Its first Chern class $c_1(K_M^{-1})$ is called the \emph{first Chern class of $M$} and is often denoted by $c_1(M)$.  It follows from Proposition \ref{Rcformula} and the above definitions that $c_1(M) = [\textrm{Ric}(\omega)]$ for any K\"ahler metric $\omega$ on $M$.

We now discuss \emph{divisors} on $M$.  First, we say that a subset $V$ of $M$ is an \emph{analytic hypersurface} if $V$ is locally given as the zero set $\{ f=0 \}$ of a locally defined holomorphic function $f$.  In general, $V$ may not be a submanifold.  Denote by $V^{\textrm{reg}}$ the set of points $p \in V$ for which $V$ is a submanifold of $M$ near $p$.  We say that $V$ is \emph{irreducible} if $V^{\textrm{reg}}$ is connected.
  A \emph{divisor} $D$ on $M$ is a formal finite sum $\sum_i a_i V_i$ where $a_i \in \mathbb{Z}$ and each $V_i$ is an irreducible analytic hypersurface of $M$.    We say that $D$ is \emph{effective} if the $a_i$ are all nonnegative.  The \emph{support} of $D$ is the union of the $V_i$ for each $i$ with $a_i \neq 0$.

Given a divisor $D$ we define an \emph{associated line bundle} as follows.  Suppose that $D$ is given by local defining functions $f_{\alpha}$ (vanishing on $D$ to order 1) over an open cover $U_{\alpha}$.  Define transition functions $t_{\alpha \beta} = f_{\alpha}/f_{\beta}$ on $U_{\alpha} \cap U_{\beta}$.  These are holomorphic and nonvanishing in $U_{\alpha}\cap U_{\beta}$, and satisfy (\ref{cocycle}).  Write $[D]$ for the associated line bundle, which is well-defined independent of choice of local defining functions.  Note that the map $D \mapsto [D]$ is not injective.  Indeed  if $D\neq 0$ is defined by a meromorphic function $f$ on $M$ then $[D]$ is trivial.

As an example: associated to a hyperplane $ \{ Z_i =0 \}$ in $\mathbb{P}^n$ is the line bundle $H$, called the \emph{hyperplane bundle}.  Taking the open cover $U_{\alpha} = \{ Z_{\alpha} \neq 0 \}$, the hyperplane is given by $ Z_i/Z_{\alpha}=0$ in $U_{\alpha}$.  Thus we can define $H$ by the transition functions $t_{\alpha \beta} = Z_{\beta}/Z_{\alpha}$.    Define a Hermitian metric $h_{\textrm{FS}}$ on $H$ by
\begin{equation}
(h_{\textrm{FS}})_{\alpha} = \frac{|Z_{\alpha}|^2}{|Z_0|^2 + \cdots + |Z_n|^2} \quad \textrm{on} \quad U_{\alpha}.
\end{equation}
Notice that $R_{h_{\textrm{FS}}} = \omega_{\textrm{FS}}$.  The canonical bundle of $\mathbb{P}^n$ is given by $K_{\mathbb{P}^n} = -(n+1)H$ and  $c_1(\mathbb{P}^{n}) = (n+1) [\omega_{\textrm{FS}}]>0$.  The line bundle $H$ is sometimes written $\mathcal{O}(1)$.

\subsection{Notions of positivity of line bundles} \label{sectnotion}

Let $L$ be a line bundle over a compact K\"ahler manifold $(M, \omega)$.  We say that $L$ is \emph{positive}  if $c_1(L)>0$.  This is equivalent to saying that there exists a Hermitian metric $h$ on $L$ for which $R_h$ is a K\"ahler form.

The Kodaira Embedding Theorem relates the positivity of $L$ with embeddings of $M$ into projective space via sections of $L$.  More precisely, write $H^0(M, L)$ for the vector space of holomorphic sections of $L$.   This is finite dimensional if not empty.   We say that $L$ is \emph{very ample} if for any ordered basis
 $\underline{s} = (s_0, \ldots, s_N)$ of $H^0(M, L)$, the map $\iota_{\underline{s}}: M \rightarrow \mathbb{P}^N$ given by
 \begin{equation} \label{iota}
 \iota_{\underline{s}} (x) = [s_0(x), \ldots, s_N(x)],
 \end{equation}
is well-defined and an embedding.  Note that $s_0(x), \ldots, s_N(x)$ are not well-defined as elements of $\mathbb{C}$, but $[s_0(x), \ldots, s_N(x)]$ is a well-defined element of $\mathbb{P}^N$ as long as not all the $s_i(x)$ vanish.  We say that $L$ is \emph{ample} if there exists a positive integer $m_0$ such that $L^m$ is very ample for all $m \ge m_0$.  The Kodaira Embedding Theorem states:

\begin{theorem}
$L$ is ample if and only if $L$ is positive.
\end{theorem}

The hard part of this theorem is the `if' direction.   For the other direction, assume that $L^m$ is very ample, with $(s_0, \ldots, s_N)$ a basis of $H^0(M, L^m)$.
Since $M$ is a submanifold of $\mathbb{P}^n$, we see that $\iota_{\underline{s}}^* \ofs$ is a K\"ahler form on $M$ and  if $h$ is any Hermitian metric on $L^m$ then by definition of $\iota_{\underline{s}}$,
\begin{equation}
\iota_{\underline{s}}^* \ofs = - \ddbar \log h + \ddbar \log ( |s_0|^2_h + \cdots + |s_N|^2_h) = R_h + \ddbar f,
\end{equation}
for a globally defined function $f$. This implies that $\frac{1}{m} \iota_{\underline{s}}^*\ofs \in c_1(L)$ and hence $c_1(L)>0$.

We say that a line bundle $L$ is \emph{globally generated} if for each $x \in M$ there exists a holomorphic section $s$ of $L$ such that $s(x) \neq 0$.  If $L$ is globally generated then given an ordered basis $\underline{s} = (s_0, \ldots, s_N)$ of holomorphic sections of $L$, we have a well-defined holomorphic map $\iota_{\underline{s}}: M \rightarrow \mathbb{P}^N$ given by (\ref{iota})  (although it is not necessarily an embedding).  We say that a line bundle $L$ is
 \emph{semi-ample} if there exists a positive integer $m_0$ such that $L^{m_0}$ is globally generated.  Observe that if $L$ is semi-ample then, by considering again the pull-back of $\ofs$ to $M$ by an appropriate map $\iota_{\underline{s}}$, there exists a Hermitian metric $h$ on $L$ such that $R_h$ is a nonnegative (1,1)-form.  That is, $c_1(L)$ contains a nonnegative representative.

We next discuss the pairing of line bundles with curves in $M$.   By a \emph{curve} in $M$ we mean an analytic subvariety of dimension 1.  If $C$ is smooth, then we define
\begin{equation}
L \cdot C = \int_C R_h,
\end{equation}
where $h$ is any Hermitian metric on $L$.  By Stokes' Theorem, $L \cdot C$ is independent of choice of $h$.  If $C$ is not smooth then we integrate over $C^{\textrm{reg}}$, the smooth part of $C$ (Stokes' Theorem still holds - see for example  \cite{GH}, p.33).  We can also pair a divisor $D$ with a curve by setting $D \cdot C = [D] \cdot C$, and we may pair a general element $\alpha \in H^{1,1}(M, \mathbb{R})$ with a curve $C$ by setting $\alpha \cdot C = \int_C \eta$ for $\eta \in \alpha$.

We say that a line bundle $L$ is \emph{nef} if $L \cdot C \ge 0$ for all curves $C$ in $M$ (`nef' is an abbreviation of either `numerically eventually free' or `numerically effective', depending on whom you ask).  It follows immediately from the definitions that:
\begin{equation}
L \ \textrm{ample} \quad \Rightarrow \quad L \ \textrm{semi-ample} \quad \Rightarrow \quad L \ \textrm{nef}.
\end{equation}

  We may also pair a line bundle with itself $n$ times, where $n$ is the complex dimension of $M$.  Define
  \begin{equation}
c_1(L)^n := \int_M (R_h)^n.
  \end{equation}
Moreover, given any $\alpha \in H^{1,1}(M, \mathbb{R})$ we define $\alpha^n = \int_M \eta^n$ for $\eta \in \alpha$.

Assume now that $M$ is a smooth projective variety.  We say that a line bundle $L$ on $M$ is \emph{big}  if there exist constants $m_0$ and $c>0$ such that $\dim H^0(M, L^m) \ge c \, m^n$ for all $m \ge m_0$.     It follows from the Riemann-Roch Theorem (see \cite{Ha, L}, for example) that a nef line bundle is big if and only if $c_1(L)^n>0$.  It follows that an ample line bundle is both nef and big.  If $M$ has $K_M$ big then we say that $M$ is of \emph{general type}.   If $M$ has $K_M$ nef then we say that $M$ is a \emph{smooth minimal model}.

We define the \emph{Kodaira dimension} of $M$ to be the infimum of $\kappa \in [-\infty, \infty)$  such  that there exists a constant $C$ with  $\dim H^0(M, K_M^m) \le C m^{\kappa}$ for all positive $m$.  In the special case that all $H^0(M, K_M^m)$ are empty, we have $\kappa=-\infty$.  The largest possible value of $\kappa$ is $n$. We write $\kod(M)$ for the Kodaira dimension $\kappa$ of $M$. Thus if $M$ is of general type then $\kod(M)=n$.  If $M$ is \emph{Fano}, which means that $c_1(M)>0$ then $\kod (M)=-\infty$.

If $K_M$ is semi-ample then for $m$ sufficiently large, the map $\iota_{\underline{s}}: M \rightarrow \mathbb{P}^N$ given by sections of $K_M^m$ has image a subvariety $Y$, which is uniquely determined up to isomorphism.  $Y$ is called the \emph{canonical model} of $M$ and $\dim Y = \kod(M)$ \cite{L}.

We now quote some results from algebraic geometry:

\pagebreak[3]
\begin{theorem} \label{algebraic} Let $M$ be an projective algebraic manifold.
\begin{enumerate}
\item[(i)] Let $\alpha$ be a K\"ahler class and let $L$ be a nef line bundle.  Then $\alpha + s \, c_1(L)$ is K\"ahler for all $s>0$.
\item[(ii)]  (Kawamata's Base Point Free Theorem) If $L$ is nef  and $aL - K_X$ is nef and big for some $a>0$ then  $L$ is semi-ample.
\item[(iii)]  (Kodaira's Lemma) Let $L$ be a nef and big line bundle on $M$.
Then there exists an effective divisor $E$ and $\delta>0$ such that $c_1(L) - \ve c_1( [E])>0$   for all $\ve \in (0, \delta]$.
\end{enumerate}
\end{theorem}
\begin{proof}  For part (i), see for example Proposition 6.2 in \cite{De} or Corollary 1.4.10 in \cite{L}.  For part (ii), see \cite{KMM, Sh}.
 For part (iii), see for example p.43 of \cite{De}. \qed
\end{proof}

It will be useful to gather here some results from complex surfaces which we will make use of later.
First we have the \emph{Adjunction Formula} for surfaces.  See for example \cite{GH} or \cite{BHPV}.

\begin{theorem} \label{thmadj}
Let $M$ be a K\"ahler surface, with $C$ an irreducible smooth curve in $M$.  Then if $g(C)$ is the genus of $C$, we have
\begin{equation}\label{adj}
 1 + \frac{K_M \cdot C + C \cdot C}{2}=g(C).
\end{equation}
Moreover, if $C$ is an irreducible, possibly singular, curve in $M$, we have
\begin{equation}\label{adj2}
 1 + \frac{K_M \cdot C + C \cdot C}{2} \ge 0,
\end{equation}
with equality if and only if $C$ is smooth and isomorphic to $\mathbb{P}^1$.
\end{theorem}

Note that $C \cdot C$ is well-defined, since $M$ has complex dimension 2 and so $C$ is both a curve and a divisor.  We may write $C^2$ instead of $C\cdot C$.
Generalizing the intersection pairing, we have the \emph{cup product form} on $H^{1,1}(M, \mathbb{R})$ given by $\alpha \cdot \beta = \int_M \alpha \wedge \beta$.  Again, we write $\alpha^2$ instead of $\alpha \cdot \alpha$.  A divisor $D$  in $M$ defines an element of $H^{1,1}(M, \mathbb{R})$ by $D \mapsto [R_h] \in H^{1,1}(M, \mathbb{R})$ for $h$ a Hermitian metric on the line bundle $[D]$, and  this is consistent with our previous definitions.

We have the  \emph{Hodge Index Theorem} for K\"ahler surfaces (see for example Theorem IV.2.14 of \cite{BHPV} or p.470 of \cite{GH}).

\begin{theorem} \label{index}
The cup product form on $H^{1,1}(M, \mathbb{R})$ is non-degenerate of type $(1, k-1)$, where $k$ is the dimension of $H^{1,1}(M, \mathbb{R})$.  In particular, if $\alpha \in H^{1,1}(M, \mathbb{R})$ satisfies $\alpha^2>0$ then for any $\beta \in H^{1,1}(M, \mathbb{R})$,
\begin{equation}
\alpha \cdot \beta =0 \quad \Rightarrow \quad \beta^2<0 \ \ \textrm{or} \ \  \beta =0.
 \end{equation}
\end{theorem}

Finally, we state the Nakai-Moishezon criterion for K\"ahler surfaces, due to Buchdahl and Lamari \cite{Bu, La}.

\begin{theorem} \label{nakai}
Let $M$ be a K\"ahler surface and $\beta$ be a K\"ahler class on $M$. If $\alpha\in H^{1,1}(M, \mathbb{R})$ is a class satisfying $$\alpha^2>0, ~~~ \alpha\cdot \beta>0, ~~~\alpha\cdot C>0$$ for every irreducible curve  $C$ on $M$, then $\alpha$ is a K\"ahler class on $M$.
\end{theorem}

A generalization of this to K\"ahler manifolds of any dimension was established by Demailly-Paun \cite{Dpa}.

\pagebreak
\section{General estimates for the K\"ahler-Ricci flow} \label{sectgen}

In this section we introduce the K\"ahler-Ricci flow equation.  We derive a number of fundamental evolution equations and estimates for the flow which will be used extensively throughout these notes.  In addition,  we discuss higher order estimates for the flow.

\subsection{The K\"ahler-Ricci flow}

Let $(M, \omega_0)$ be a compact K\"ahler manifold of complex dimension $n$.
A solution of the \emph{K\"ahler-Ricci flow} on $M$ starting at $\omega_0$ is a family of K\"ahler metrics $\omega=\omega(t)$ solving
\begin{equation} \label{krf00}
\ddt{} \omega  = - \Ric(\omega), \qquad \omega|_{t=0} = \omega_0.
\end{equation}
Note that this differs from Hamilton's equation (\ref{hamiltonrf}) by a factor of 2:  see Remark \ref{factorof2}.

For later use it will be convenient to consider a more general equation than (\ref{krf00}), namely
\begin{equation} \label{nu}
\ddt{} \omega  = - \Ric(\omega) - \nu \omega, \qquad \omega|_{t=0} = \omega_0,
\end{equation}
where $\nu$ is a fixed real number which we take to be either $\nu=0$ or $\nu=1$. As we will discuss later in Section \ref{sectn0}, the case $\nu=1$ corresponds to a rescaling of (\ref{krf00}).  When $\nu=1$ we call (\ref{nu}) the \emph{normalized K\"ahler-Ricci flow}.

We have the following existence and uniqueness result.

\begin{theorem} \label{hamilton} There exists a unique solution $\omega=\omega(t)$ to (\ref{nu}) on some maximal time interval $[0, T)$ for some $T$ with $0< T \le \infty$.
\end{theorem}

Since the case $\nu=1$ is a rescaling of (\ref{krf00}), it suffices to consider (\ref{krf00}).
We will provide a proof of this in Section \ref{maximal}, and   show in addition that $T$ can be prescribed in terms of the cohomology class of $[\omega_0]$ and the manifold $M$.   Theorem \ref{hamilton} also follows from the well-known results of Hamilton.  Indeed
we can use the
short time existence result  of Hamilton \cite{H1} (see also \cite{DeT}) to obtain a maximal solution to the Ricci flow $\ddt{} g_{ij} = - R_{ij}$ on $[0, T)$ starting at $g_0$ for some $T>0$.  Since the Ricci flow preserves the K\"ahler condition (see e.g. \cite{H3}), $g(t)$ solves (\ref{krf00}) on $[0,T)$.   Note  that this argument does not  explicitly give us the value of $T$.

A remark about notation.  When we write tensorial objects such as curvature tensors $R_{i \ov{j} k \ov{\ell}}$, covariant derivatives $\nabla_i$, Laplace operators $\Delta$, we refer to the objects corresponding  to  the evolving metric $\omega=\omega(t)$, unless otherwise indicated.

\subsection{Evolution of scalar curvature}

Let $\omega=\omega(t)$ be a solution to the K\"ahler-Ricci flow (\ref{nu}) on $[0,T)$ for $T$ with $0 < T \le \infty$.
We compute the well-known evolution of the scalar curvature.
\begin{theorem} \label{scalar}
The scalar curvature $R$ of $\omega=\omega(t)$ evolves by
\begin{equation}
\ddt{R} = \Delta R + | \emph{Ric}(\omega) |^2 +  \nu  R,
\end{equation}
where $| \emph{Ric}(\omega)|^2 = g^{\ov{\ell}i} g^{ \ov{j}k} R_{i \ov{j}} R_{k \ov{\ell}}$.  Hence the scalar curvature has a lower bound
\begin{equation} \label{lbR}
R(t)  \ge - \nu n -C_0e^{- \nu t},
\end{equation}
for $C_0 = - \inf_M R(0) - \nu n$.
\end{theorem}
\begin{proof}  Taking the trace of the evolution equation (\ref{nu}) gives
\begin{equation}
g^{\ov{\ell}k} \ddt{} g_{k \ov{\ell}} = - R - \nu n.
\end{equation}
Since $R = - g^{ \ov{j}i} \partial_i \partial_{\ov{j}} \log \det g$ we have
\begin{align}
\ddt{R} & =  - g^{\ov{j}i} \partial_i \partial_{\ov{j}} \left( g^{ \ov{\ell}k} \ddt{} g_{k \ov{\ell}}\right)- \left( \ddt{} g^{\ov{j}i} \right) \partial_i \partial_{\ov{j}} \log \det g \\ &= \Delta R + g^{\ov{\ell}i} g^{ \ov{j}k}  R_{k \ov{\ell}} R_{i \ov{j}} + \nu R,
\end{align}
as required.   For (\ref{lbR}), we use the elementary fact that $n | \Ric (\omega)|^2 \ge R^2$ to obtain
\begin{equation}
\left( \ddt{}- \Delta \right) R \ge  \frac{1}{n} R(R + \nu n) = \frac{1}{n} (R+\nu n)^2 - \nu (R+\nu n).
\end{equation}
Hence
\begin{equation}
\left( \ddt{} - \Delta \right) (e^{\nu t} (R+\nu n)) \ge 0.
\end{equation}
By the maximum principle (see Proposition \ref{propheat} and the remark following it), the quantity $e^{\nu t} (R+\nu n)$ is bounded below by $\inf_M R(0) + \nu n$, its value at time $t=0$.  \qed
\end{proof}

We remark that although we used the K\"ahler condition to prove Theorem \ref{scalar}, in fact it holds in full generality for the Riemannian Ricci flow \cite{H1} (see also \cite{CK}).

Theorem \ref{scalar} implies a bound on the volume form of the metric.

\begin{corollary} \label{volform} Let $\omega=\omega(t)$ be a solution of (\ref{nu}) on $[0,T)$ and $C_0$ as in Theorem \ref{scalar}.
\begin{enumerate}
\item[(i)] If $\nu=0$ then
\begin{equation} \label{vf}
\omega^n(t) \le e^{C_0 t} \omega^n(0).
\end{equation}
In particular, if $T$ is finite then the volume form $\omega^n(t)$ is uniformly bounded from above for $t \in [0,T)$.
\item[(ii)] If $\nu=1$ there exists a uniform constant $C$ such that
\begin{equation} \label{vf2}
\omega^n(t) \le e^{C_0(1- e^{-t})} \omega^n(0).
\end{equation}
In particular,  the volume form $\omega^n(t)$ is uniformly bounded from above for $ t \in [0,T)$.
\end{enumerate}
\end{corollary}
\begin{proof}
We have
\begin{equation}
\ddt{} \log \frac{\omega^n(t)}{\omega^n(0)} = g^{\ov{j}i} \ddt{} g_{i \ov{j}} = -R - \nu n \le C_0 e^{-\nu t}.
\end{equation}
Integrating in time, we obtain (\ref{vf}) and (\ref{vf2}).  \qed
\end{proof}

\subsection{Evolution of the trace of the metric}

We now prove an estimate for the trace of the metric along the K\"ahler-Ricci flow.
This is originally due to Cao \cite{Cao} and is the parabolic version of an estimate for the complex Monge-Amp\`ere equation due to Yau and Aubin \cite{A, Y2}.  We give the estimate in the form of an evolution inequality.  We begin by computing the evolution of $\tr{\hat{\omega}}{\omega}$, the trace of $\omega$ with respect to a fixed metric $\hat{\omega}$, using the notation of Section \ref{sectmax}.

\begin{proposition} \label{propkeyeqn}
Let $\hat{\omega}$ be a fixed K\"ahler metric on $M$, and let $\omega=\omega(t)$ be a solution to the K\"ahler-Ricci flow (\ref{nu}).  Then
\begin{equation} \label{keyeqn}
\left( \ddt{} - \Delta \right) \emph{tr}_{\hat{\omega}}{\, \omega} = - \nu \, \emph{tr}_{\hat{\omega}}{\, \omega}  - g^{ \ov{\ell}k} \hat{R}_{k \ov{\ell}}^{\ \ \, \ov{j} i} g_{i \ov{j}} - \hat{g}^{\ov{j}i} g^{ \ov{q}p} g^{ \ov{\ell} k} \hat{\nabla}_i g_{p \ov{\ell}} \hat{\nabla}_{\ov{j}} g_{k \ov{q}},
\end{equation}
where $ \hat{R}_{k \ov{\ell}}^{\ \ \, \ov{j} i}$, $\hat{\nabla}$ denote the curvature and covariant derivative with respect to $\hat{g}$.
\end{proposition}
\begin{proof}
Compute using normal coordinates for $\hat{g}$ and the formula (\ref{formulacurvature}),
\begin{align} \nonumber
\Delta \tr{\hat{\omega}}{\omega} & = g^{ \ov{\ell} k} \partial_k \partial_{\ov{\ell}} (\hat{g}^{\ov{j}i} g_{i \ov{j}}) \\ \nonumber
& = g^{ \ov{\ell}k} (\partial_k \partial_{\ov{\ell}}\, \hat{g}^{ \ov{j}i }) g_{i \ov{j}} + g^{ \ov{\ell} k} \hat{g}^{\ov{j} i} \partial_k \partial_{\ov{\ell}} g_{i \ov{j}} \\
& = g^{\ov{\ell} k} \hat{R}_{k \ov{\ell}}^{\ \ \, \ov{j} i} g_{i \ov{j}} - \hat{g}^{ \ov{j} i} R_{i \ov{j}} + \hat{g}^{\ov{j} i} g^{ \ov{q} p} g^{ \ov{\ell} k} \partial_i g_{p \ov{\ell}} \partial_{\ov{j}} g_{k \ov{q}},
\end{align}
and
\begin{equation}
\ddt{} \tr{\hat{\omega}}{\omega} = - \hat{g}^{ \ov{j} i} R_{i \ov{j}} - \nu \, \tr{\hat{\omega}}{\omega},
\end{equation}
and combining these gives (\ref{keyeqn}). \qed
\end{proof}

We use Proposition \ref{propkeyeqn} to prove the following estimate, which will be used frequently in the sequel:

\begin{proposition} \label{propChat}
Let $\hat{\omega}$ be a fixed K\"ahler metric on $M$, and let $\omega=\omega(t)$ be a solution to  (\ref{nu}).  Then there exists a constant $\hat{C}$ depending only on the lower bound of the bisectional curvature for $\hat{g}$ such that
\begin{equation} \label{eqnlog}
\left( \ddt{} - \Delta \right) \log \emph{tr}_{\hat{\omega}}{\, \omega} \le \hat{C} \emph{tr}_{\omega}{\, \hat{\omega}}- \nu.
\end{equation}
\end{proposition}
\begin{proof}
First observe  that for a positive function $f$,
\begin{equation} \label{logf}
\Delta \log f = \frac{\Delta f}{f} - \frac{| \partial f|^2_{g}}{f^2}.
\end{equation}
It follows immediately from Proposition \ref{propkeyeqn} that
\begin{eqnarray}  \nonumber
\lefteqn{\left( \ddt{} - \Delta \right) \log \tr{\hat{\omega}}{\omega} }\\
 && = \frac{1}{\tr{\hat{\omega}}{\omega}} \left( - \nu\tr{\hat{\omega}}{\omega} - g^{ \ov{\ell} k} \hat{R}_{k \ov{\ell}}^{\ \ \, \ov{j} i} g_{i \ov{j}}  + \frac{ | \partial \tr{\hat{\omega}}{\omega}|^2_g}{\tr{\hat{\omega}}{\omega}} -  \hat{g}^{ \ov{j}i} g^{ \ov{q}p} g^{ \ov{\ell}k} \hat{\nabla}_i g_{p \ov{\ell}} \hat{\nabla}_{\ov{j}} g_{k \ov{q}} \right). \label{eqnlog1}
\end{eqnarray}
We claim that
\begin{equation} \label{claimcs}
\frac{ | \partial \tr{\hat{\omega}}{\omega}|^2_g}{\tr{\hat{\omega}}{\omega}} -  \hat{g}^{\ov{j}i} g^{ \ov{q}p} g^{\ov{\ell}k} \hat{\nabla}_i g_{p \ov{\ell}} \hat{\nabla}_{\ov{j}} g_{k \ov{q}}  \le 0.
\end{equation}
To prove this we choose normal coordinates for $\hat{g}$ for which $g$ is diagonal.  Compute using the Cauchy-Schwarz inequality
\begin{align} \nonumber
| \partial \tr{\hat{\omega}}{\omega}|^2_{g} & = \sum_i g^{ \ov{i}i} \partial_i \left( \sum_j g_{j \ov{j}} \right) \partial_{\ov{i}} \left( \sum_k g_{k \ov{k}} \right) \\ \nonumber
& = \sum_{j,k} \sum_i g^{\ov{i}i}(\partial_i g_{j \ov{j}}) (\partial_{\ov{i}} g_{k \ov{k}}) \\ \nonumber
& \le \sum_{j,k} \left( \sum_i g^{\ov{i}i} |\partial_i g_{j \ov{j}}|^2 \right)^{1/2} \left( \sum_i g^{\ov{i}i} |\partial_i g_{k \ov{k}}|^2 \right)^{1/2} \\ \nonumber
& = \left( \sum_j \left( \sum_i g^{ \ov{i}i} |\partial_i g_{j \ov{j}}|^2 \right)^{1/2} \right)^2 \\ \nonumber
& = \left( \sum_j \sqrt{g_{j \ov{j}}} \left( \sum_i g^{\ov{i}i} g^{ \ov{j} j} |\partial_i g_{j \ov{j}}|^2 \right)^{1/2} \right)^2 \\ \nonumber
& \le \sum_\ell g_{\ell \ov{\ell}} \sum_{i,j} g^{\ov{i}i} g^{ \ov{j} j} | \partial_j g_{i \ov{j}}|^2 \\ \label{long}
& \le (\tr{\hat{\omega}}{\omega}) \sum_{i,j,k} g^{ \ov{i} i} g^{ \ov{j} j} \partial_k g_{i \ov{j}} \partial_{\ov{k}} g_{j \ov{i}},
 \end{align}
where in the second-to-last line we used the K\"ahler condition to give $\partial_i g_{j \ov{j}} = \partial_j g_{i \ov{j}}$.  The inequality (\ref{long}) gives exactly (\ref{claimcs})

We can now complete the proof of the proposition.  Define a constant $\hat{C}$ by
\begin{equation}
\hat{C} = - \inf_{x \in M}  \{ \hat{R}_{i \ov{i} j \ov{j}}(x) \ | \ \{ \partial_{z^1}, \ldots, \partial_{z^n} \} \textrm{ is orthonormal w.r.t. } \hat{g} \textrm{ at } x, \ i,j=1, \ldots, n \},
\end{equation}
which is finite since we are taking the infimum of a continuous function over a compact set.

Then computing at a point using normal coordinates for $\hat{g}$ for which the metric $g$ is diagonal  we have
\begin{equation} \label{eqnlb1}
g^{\ov{\ell} k} \hat{R}_{k \ov{\ell}}^{\ \ \, \ov{j} i} g_{i \ov{j}} = \sum_{k, i} g^{ \ov{k} k} \hat{R}_{k \ov{k} i \ov{i}} g_{i \ov{i}} \ge - \hat{C} \sum_{k} g^{ \ov{k} k} \sum_i g_{i \ov{i}} = -\hat{C} (\tr{\hat{\omega}}{\omega}) (\tr{\omega}{\hat{\omega}}).
\end{equation}
Combining (\ref{eqnlog1}), (\ref{claimcs}) and (\ref{eqnlb1}) yields (\ref{eqnlog}). \qed
\end{proof}

\subsection{The parabolic Schwarz Lemma}

In this section we prove the parabolic Schwarz lemma of \cite{SoT1}.  This is a parabolic version of Yau's Schwarz lemma \cite{Y1}.  We state it here in the form of an evolution inequality.

\begin{theorem} \label{psl}
Let $f: M \rightarrow N$  be a holomorphic map between compact complex manifolds $M$ and $N$ of complex dimension $n$ and $\kappa$ respectively.  Let $\omega_0$ and $\omega_N$ be K\"ahler metrics on $M$ and $N$ respectively and let $\omega=\omega(t)$ be a solution of (\ref{nu}) on $M \times [0,T)$, namely
\begin{equation} \label{krfmu}
\ddt{} \omega  = - \emph{Ric}(\omega) - \nu \omega, \qquad \omega|_{t=0} = \omega_0,
\end{equation}
for $t \in [0,T)$,  with either $\nu=0$ or $\nu=1$.  Then for all points of $M \times [0,T)$ with $\emph{tr}_{\omega}{(f^* \omega_N)}$ positive we have
\begin{equation} \label{psi}
\left( \ddt{} - \Delta \right) \log \emph{tr}_{\omega}{(f^* \omega_N)} \le C_N \emph{tr}_{\omega} (f^* \omega_N) + \nu,
\end{equation}
where $C_N$ is an upper bound for the bisectional curvature of $\omega_N$.
\end{theorem}

Observe that a simple maximum principle argument immediately  gives the following consequence which the reader will recognize as similar to the conclusion of Yau's Schwarz lemma.

\begin{corollary}  If the bisectional curvature of $\omega_N$ has a negative upper bound $C_N<0$ on $N$ then there exists a constant $C>0$ depending only on $C_N$, $\omega_0$, $\omega_N$ and $\nu$ such that $\emph{tr}_{\omega} (f^* \omega_N) \le C$ on $M \times [0,T)$ and hence
\begin{equation}
 \omega \ge \frac{1}{C} f^* \omega_N, \quad \textrm{on } M \times [0,T).
\end{equation}
\end{corollary}

In practice, we will find the inequality (\ref{psi}) more useful than this corollary, since the assumption of negative bisectional curvature is rather strong.
For the proof of Theorem \ref{psl}, we will follow quite closely the notation and calculations given in \cite{SoT1}.

\begin{proof}[Proof of Theorem \ref{psl}] Fix $x$ in $M$ with $f(x)=y \in N$, and choose  normal coordinate systems $(z^i)_{i=1, \ldots, n}$ for $g$ centered at $x$ and $(w^{\alpha})_{\alpha=1, \ldots, \kappa}$ for $g_N$ centered at $y$.  The map $f$ is given locally as $(f^1, \ldots, f^{\kappa})$ for holomorphic functions $f^{\alpha} = f^{\alpha}(z^1, \ldots, z^n)$.  Write $f^{\alpha}_i$ for $\frac{\partial}{\partial z^i} f^{\alpha}$.  To simplify notation we write the components of $g_N$ as $h_{\alpha \ov{\beta}}$ instead of $(g_N)_{\alpha \ov{\beta}}$.  The components of the tensor $f^* g_N$ are then $f^{\alpha}_i \ov{f^{\beta}_{j}} h_{\alpha \ov{\beta}}$   and hence $ \tr{\omega}{(f^* \omega_N)}= g^{ \ov{j}i } f^{\alpha}_i \ov{f^{\beta}_{j}} h_{\alpha \ov{\beta}}$.  Writing $u = \tr{\omega}{(f^* \omega_N)}>0$, we compute at $x$,
\begin{align} \nonumber
\Delta u & = g^{\ov{\ell}k} \partial_k \partial_{\ov{\ell}} \left( g^{\ov{j} i} f_i^{\alpha} \ov{f^{\beta}_{j}} h_{\alpha \ov{\beta}} \right) \\
& = R^{ \ov{j} i} f_i^{\alpha} \ov{f^{\beta}_{j}} h_{\alpha \ov{\beta}} + g^{\ov{\ell} k} g^{ \ov{j} i} ( \partial_k f^{\alpha}_i )(\ov{ \partial_{\ell} f^{\beta}_j} )h_{\alpha \ov{\beta}} - g^{ \ov{\ell} k} g^{ \ov{j} i} S_{\alpha \ov{\beta} \gamma \ov{\delta}} f^\alpha_i \ov{f^{\beta}_j} f^{\gamma}_k \ov{f^{\delta}_{\ell}}, \label{Deltau}
\end{align}
for $S_{\alpha \ov{\beta} \gamma \ov{\delta}}$ the curvature tensor of $g_N$ on $N$.  Next,
\begin{align}
\ddt{u}  = - g^{\ov{\ell} i} g^{ \ov{j} k} \left( \ddt{} g_{k \ov{\ell}} \right) f^{\alpha}_i \ov{f^{\beta}_j} h_{\alpha \ov{\beta}} = R^{ \ov{j}i} f^{\alpha}_i \ov{f^{\beta}_j} h_{\alpha \ov{\beta}} + \nu u. \label{ddtu}
\end{align}
Combining (\ref{Deltau}) and (\ref{ddtu}) with  (\ref{logf}), we obtain
\begin{align} \nonumber
\left( \ddt{} - \Delta \right) \log u & =  \frac{1}{u}  g^{ \ov{\ell} k} g^{ \ov{j} i} S_{\alpha \ov{\beta} \gamma \ov{\delta}} f^\alpha_i \ov{f^{\beta}_j} f^{\gamma}_k \ov{f^{\delta}_{\ell}} \\
&  + \frac{1}{u} \left( \frac{|\partial u|_{g}^2}{u} - g^{ \ov{\ell} k} g^{ \ov{j} i} ( \partial_k f^{\alpha}_i )(\ov{ \partial_{\ell} f^{\beta}_j} )h_{\alpha \ov{\beta}} \right) + \nu.
\end{align}
If $C_N$ is an upper bound for the bisectional curvature of $g_N$ we see that
\begin{equation}
g^{\ov{\ell} k} g^{\ov{j}i} S_{\alpha \ov{\beta} \gamma \ov{\delta}} f^\alpha_i \ov{f^{\beta}_j} f^{\gamma}_k \ov{f^{\delta}_{\ell}} \le C_N u^2,
\end{equation}
and hence (\ref{psi}) will follow from the inequality
\begin{equation} \label{claimcs2}
\frac{|\partial u|_{g}^2}{u} -  g^{ \ov{\ell} k} g^{ \ov{j} i} ( \partial_k f^{\alpha}_i )(\ov{ \partial_{\ell} f^{\beta}_j} )h_{\alpha \ov{\beta}} \le 0.
\end{equation}
The inequality (\ref{claimcs2}) is analogous to (\ref{claimcs}) and the proof is almost identical.  Indeed, at the point $x$,
\begin{align} \nonumber
| \partial u|^2_g & = \sum_{i,j,k, \alpha, \beta} \ov{f_i^{\alpha}} f_j^{\beta} \partial_k f^{\alpha}_i \ov{ \partial_k f^{\beta}_j} \\ \nonumber
& \le  \sum_{i,j, \alpha, \beta} | f_i^{\alpha}| |f_j^{\beta}| \left( \sum_k | \partial_k f^{\alpha}_i |^2 \right)^{1/2} \left( \sum_{\ell} | \partial_{\ell} f^{\beta}_j |^2 \right)^{1/2} \\ \nonumber
& = \left( \sum_{i, \alpha} | f_i^{\alpha} | \left( \sum_k |\partial_k f^{\alpha}_i |^2 \right)^{1/2} \right)^2 \\
& \le \left( \sum_{j, \beta} |f^{\beta}_j|^2 \right) \left( \sum_{i,k, \alpha} |\partial_k f^{\alpha}_{i}|^2 \right) = u \, g^{ \ov{\ell} k} g^{ \ov{j} i} ( \partial_k f^{\alpha}_i )(\ov{ \partial_{\ell} f^{\beta}_j} )h_{\alpha \ov{\beta}},
\end{align}
which gives (\ref{claimcs2}).  \qed
\end{proof}

\subsection{The 3rd order estimate } \label{sect3rd}

In this section we prove the so-called `3rd order' estimate for the K\"ahler-Ricci flow assuming that the metric is uniformly bounded.   By 3rd order estimate we mean an estimate on the first derivative of the K\"ahler metric, which is  of order 3 in terms of the potential function.  Since the work of Yau \cite{Y2} on the elliptic Monge-Amp\`ere equation, such estimates have  often been referred to as  \emph{Calabi estimates} in reference to a well-known calculation of Calabi \cite{C1}.  There are now many generalizations of the Calabi estimate \cite{Che, ShW, To2, TWY, ZZ}.   A  parabolic Calabi estimate was applied to the K\"ahler-Ricci flow in \cite{Cao}.  Phong-\v{S}e\v{s}um-Sturm \cite{PSS} later gave a succinct and explicit formula, which we will describe here.

 Let $\omega=\omega(t)$ be a solution of the normalized K\"ahler-Ricci flow (\ref{nu}) on $[0,T)$ for $0<T\le \infty$ and let $\hat{\omega}$ be a fixed K\"ahler metric on $M$.
 We wish to estimate the quantity $S = | \hat{\nabla} g |^2$ where $\hat{\nabla}$ is the covariant derivative with respect to $\hat{g}$ and the norm $| \, \cdot \,  |$ is taken with respect to the evolving metric $g$.  Namely
 \begin{equation}
 S = g^{\ov{j} i } g^{\ov{\ell} k} g^{\ov{q} p} \hat{\nabla}_i g_{k\ov{q}}\ov{\hat{\nabla}_j  g_{\ell \ov{p}}}.
 \end{equation}
 Define a tensor $\Psi^{k}_{ij}$ by
 \begin{equation}
\Psi^k_{ij} :=  \Gamma^k_{ij} - \hat{\Gamma}^k_{ij} = g^{\ov{\ell} k} \hat{\nabla}_i g_{j \ov{\ell}}.
 \end{equation}
We may rewrite $S$ as
\begin{equation} \label{rewriteS}
S  = | \Psi|^2=  g^{\ov{j} i} g^{\ov{q} p} g_{k \ov{\ell}} \Psi_{ip}^{k}  \ov{\Psi_{jq}^{\ell} }.
\end{equation}
We have the following key equality of Phong-\v{S}e\v{s}um-Sturm \cite{PSS}.

\begin{proposition} \label{propPSS1} With the notation above, $S$ evolves by
\begin{align} \label{PSS1}
\left( \ddt{} - \Delta \right) S & = - | \ov{\nabla} \Psi |^2 - | \nabla \Psi |^2 + \nu | \Psi |^2
 - 2 \emph{Re} \left( \, g^{\ov{j} i} g^{\ov{q} p} g_{k \ov{\ell}} \nabla^{\ov{b}} \hat{R}_{i \ov{b} p}^{\ \ \ k} \ov{\Psi_{j q}^{\ell }} \right),
\end{align}
where $\nabla^{\ov{b}} = g^{\ov{b} a} \nabla_a$ and  $\hat{R}_{i \ov{b} p}^{\ \ \ k} := \hat{g}^{\ov{m} k} \hat{R}_{i \ov{b} p \ov{m}}$.
\end{proposition}
\begin{proof}   Compute
\begin{align} \label{pss2}
\Delta S & = g^{\ov{j} i} g^{\ov{q} p} g_{k \ov{\ell}}  \left( (\Delta \Psi^k_{ip})  \ov{  \Psi_{jq}^{\ell} }   +\Psi_{ip}^k  (\ov{ \ov{\Delta} \Psi_{jq}^{\ell}} )  \right)  +  | \ov{\nabla} \Psi|^2 + | \nabla \Psi |^2,
\end{align}
where we are writing $\Delta = g^{\ov{b} a} \nabla_a \nabla_{\ov{b}}$ for the `rough' Laplacian and $\ov{\Delta} = g^{\ov{b} a} \nabla_{\ov{b}} \nabla_a$ for its conjugate.
While  $\Delta$ and $\ov{\Delta}$ agree when acting on functions, they differ in general when acting on tensors.  In particular, using the commutation formulae (see Section \ref{sectcurv}),
\begin{align}  \label{pss3}
 \ov{\Delta} \Psi_{jq}^{\ell} & =  \Delta \Psi_{jq}^{\ell}   + R^{\ b}_{j} \Psi_{bq}^{\ell}  + R^{\ b}_{q} \Psi_{jb}^{\ell}   - R^{\ \ell}_{b} \Psi_{jq}^b.
\end{align}
Combining (\ref{pss2}) and (\ref{pss3}),
\begin{align} \nonumber
\Delta S & = 2 \textrm{Re} \left( g^{\ov{j} i} g^{\ov{q} p} g_{k \ov{\ell}}  (\Delta \Psi_{ip}^k ) \ov{ \Psi_{jq}^{\ell}} \right)
 +  | \ov{\nabla} \Psi |^2 + | \nabla \Psi |^2 \\ \label{pss4}
 & + R^{\ov{j} i} g^{\ov{q}p} g_{k \ov{\ell}} \Psi^k_{ip} \ov{\Psi^{\ell}_{jq}} + g^{\ov{j}i} R^{\ov{q} p} g_{k \ov{\ell}} \Psi^k_{ip} \ov{\Psi^{\ell}_{jq}} - g^{\ov{j}i} g^{\ov{q}p} R_{k \ov{\ell}} \Psi^k_{ip} \ov{\Psi^{\ell}_{jq}}
 \end{align}
We now compute the time derivative of $S$ given by (\ref{rewriteS}).
We claim that
\begin{equation} \label{pss5}
\ddt{} \Psi_{ip}^k = \Delta \Psi_{ip}^k  - \nabla^{\ov{b}} \hat{R}_{i \ov{b} p}^{\ \ \ k}.
\end{equation}
Given this, together with
\begin{equation}
\ddt{} g^{\ov{j} i} = R^{\ov{j} i} + \nu g^{\ov{j} i}, \ \ddt{} g_{k \ov{\ell}} = - R_{k \ov{\ell}} - \nu g_{k \ov{\ell}},
\end{equation}
we obtain
\begin{align} \nonumber
\ddt{} S & = 2 \textrm{Re} \left( g^{\ov{j} i} g^{\ov{q} p} g_{k \ov{\ell}}  \left(\Delta \Psi_{ip}^k - \nabla^{\ov{b}} \hat{R}_{i \ov{b} p}^{\ \ \ k} \right) \ov{ \Psi_{jq}^{\ell}} \right)    \label{pss55} + R^{\ov{j} i} g^{\ov{q}p} g_{k \ov{\ell}} \Psi^k_{ip} \ov{\Psi^{\ell}_{jq}}  \\ &
 + g^{\ov{j}i} R^{\ov{q} p} g_{k \ov{\ell}} \Psi^k_{ip} \ov{\Psi^{\ell}_{jq}} - g^{\ov{j}i} g^{\ov{q}p} R_{k \ov{\ell}} \Psi^k_{ip} \ov{\Psi^{\ell}_{jq}} + \nu |\Psi|^2.
\end{align}
Then (\ref{PSS1}) follows from
 (\ref{pss4}) and (\ref{pss55}).

To establish (\ref{pss5}), compute
\begin{equation} \label{pss6}
\ddt{} \Psi_{ip}^k = \ddt{} \Gamma^k_{ip} = - \nabla_i R_{p}^{\ k}.
\end{equation}
On the other hand,
\begin{equation} \label{rhatr}
\nabla_{\ov{b}} \Psi_{ip}^k = \partial_{\ov{b}} (\Gamma_{ip}^k - \hat{\Gamma}_{ip}^k) = \hat{R}_{i \ov{b} p}^{\ \ \ k} - R_{i \ov{b} p}^{\ \ \ k},
\end{equation}
and hence
\begin{equation} \label{pss7}
\Delta \Psi_{ip}^k =  g^{\ov{b} a} \nabla_a \nabla_{\ov{b}} \Psi_{ip}^k = \nabla^{\ov{b}} \hat{R}_{i \ov{b} p}^{\ \ \ k} - \nabla_i R_p^{\ k}.
\end{equation}
where for the last equality we have used the second Bianchi identity (part (iii) of Proposition \ref{Rsymmetry}).   Then (\ref{pss5}) follows from (\ref{pss6}) and (\ref{pss7}). \qed
\end{proof}

Using this evolution equation together with Proposition \ref{propkeyeqn}, we obtain a third order estimate assuming a metric bound.

\begin{theorem} \label{theoremSbound}
Let $\omega=\omega(t)$ solve  (\ref{nu}) and assume that there exists a constant $C_0>0$ such that
\begin{equation} \label{assumeomega}
\frac{1}{C_0} \omega_0 \le \omega \le C_0 \omega_0.
\end{equation}
Then there exists a constant $C$ depending only on $C_0$ and $\omega_0$ such that
\begin{equation} \label{SC}
S:= | \nabla_{g_0} g|^2  \le C.
\end{equation}
In addition, there exists a constant $C'$ depending only on $C_0$ and $\omega_0$ such that
\begin{equation} \label{SRm}
\left( \ddt{} - \Delta \right) S \le -\frac{1}{2} | \emph{Rm} |^2 +C',
\end{equation}
where $| \emph{Rm}|^2$ denotes the norm squared of the curvature tensor $R_{i \ov{j} k \ov{\ell}}$.
\end{theorem}
\begin{proof}
We apply (\ref{PSS1}).  First note that
\begin{equation}
\nabla^{\ov{b}} \hat{R}_{i \ov{b} p}^{\ \ \ \, k} = g^{\ov{b} r} \hat{\nabla}_r \hat{R}_{i \ov{b} p}^{\ \ \ \, k} - g^{\ov{b}r} \Psi^a_{i r} \hat{R}_{a\ov{b}p}^{\ \ \ \, k} - g^{\ov{b}r} \Psi^a_{pr} \hat{R}_{i \ov{b} a}^{\ \ \ \, k} + g^{\ov{b} r} \Psi^k_{a r} \hat{R}_{i \ov{b} p}^{\ \ \ \, a}.
\end{equation}
Then with $\hat{g}=g_0$, we have, using (\ref{assumeomega}),
\begin{equation}
\left| 2 \textrm{Re} \left( \, g^{\ov{j} i} g^{\ov{q} p} g_{k \ov{\ell}} \nabla^{\ov{b}} \hat{R}_{i \ov{b} p}^{\ \ \ k} \ov{\Psi_{j q}^{\ell }} \right) \right| \le C_1( S + \sqrt{S}) \le 2C_1(S+1),
\end{equation}
for some uniform constant $C_1$.
Hence for a uniform $C_2$,
\begin{equation} \label{evolveS}
\left( \ddt{} - \Delta \right) S \le - | \ov{\nabla} \Psi |^2 - | \nabla \Psi |^2 + C_2 S + C_2.
\end{equation}
On the other hand, from Proposition \ref{propkeyeqn} and the assumption (\ref{assumeomega}) again,
\begin{equation}
\left( \ddt{} - \Delta \right) \tr{\hat{\omega}}{\omega} \le C_3 - \frac{1}{C_3} S,
\end{equation}
for a uniform $C_3>0$.
Define $Q = S + C_3(1+ C_2) \tr{\hat{\omega}}{\omega}$ and compute
\begin{equation}
\left( \ddt{} - \Delta \right) Q \le -S +C_4,
\end{equation}
for a uniform constant $C_4$.  It follows that $S$ is bounded from above at a point at which $Q$ achieves a maximum, and (\ref{SC}) follows.

For (\ref{SRm}), observe from (\ref{rhatr}) that
\begin{equation} \label{Rmbd}
| \ov{\nabla} \Psi |^2 = | \hat{R}_{i \ov{b} p}^{\ \ \ k} - R_{i \ov{b} p}^{\ \ \ k}|^2 \ge \frac{1}{2} | \textrm{Rm}|^2 -C_5.
\end{equation}
Then (\ref{SRm}) follows from (\ref{evolveS}), (\ref{Rmbd}) and (\ref{SC}).  \qed
\end{proof}

\subsection{Curvature and higher derivative bounds} \label{secthoe}

In this section we assume that we have a solution $\omega=\omega(t)$ of (\ref{nu}) on $[0,T)$ with $0< T\le \infty$ which satisfies the estimates
\begin{equation} \label{assumeomega2}
\frac{1}{C_0} \omega_0 \le \omega \le C_0 \omega_0,
\end{equation}
for some uniform constant $C_0$.  We show that the curvature and all derivatives of the curvature of $\omega$ are uniformly bounded, and that we have uniform $C^{\infty}$ estimates of $g$ with respect to the fixed metric $\omega_0$.  We first compute the evolution of the curvature tensor.

\begin{lemma} \label{lemmaevolveRm}
Along the flow (\ref{nu}), the curvature tensor evolves by
\begin{align} \nonumber
\ddt{} R_{i \ov{j} k \ov{\ell}} & = \frac{1}{2} \Delta_{\mathbb{R}} R_{i \ov{j} k \ov{\ell}} - \nu R_{i \ov{j} k \ov{\ell}}+ R_{i \ov{j} a \ov{b}} R^{\ov{b}a}_{\ \ \, k \ov{\ell}} + R_{i \ov{b} a \ov{\ell}} R_{\  \ov{j} k}^{\ov{b}\ \  \, a} - R_{i \ov{a} k \ov{b}} R^{\ov{a} \ \, \ov{b}}_{\ \, \ov{j} \ \, \ov{\ell}} \\
& -\frac{1}{2} \left( R_i^{\ \, a} R_{a \ov{j} k \ov{\ell}} + R^{\ov{a}}_{\ \, \ov{j}} R_{i \ov{a} k \ov{\ell}} + R_k^{\ \, a} R_{i \ov{j} a \ov{\ell}} + R^{\ov{a}}_{\ \, \ov{\ell}} R_{i\ov{j}k \ov{a}} \right) \label{evolveRm}
\end{align}
where we write $\Delta_{\mathbb{R}} =\Delta + \ov{\Delta}$ and $\Delta = g^{\ov{q} p} \nabla_p \nabla_{\ov{q}}$.
\end{lemma}
\begin{proof}
Using the formula $\ddt{} \Gamma_{ik}^p = - \nabla_i R_k^{\ \, p}$ and the Bianchi identity, compute
\begin{equation}
\ddt{} R_{i \ov{j} k \ov{\ell}}= - \left( \ddt{} g_{p \ov{j}}\right) \partial_{\ov{\ell}} \Gamma^p_{ik} - g_{p \ov{j}} \partial_{\ov{\ell}}\left( \ddt{} \Gamma_{ik}^p \right) = - R^{\ov{a}}_{\ \, \ov{j}} R_{i \ov{a} k\ov{\ell}  } - \nu R_{i \ov{j} k \ov{\ell}} + \nabla_{\ov{\ell}} \nabla_k R_{i \ov{j}}. \label{ddtR}
\end{equation}
Using the Bianchi identity again and the commutation formulae, we obtain
\begin{align} \nonumber
\Delta R_{i \ov{j} k \ov{\ell}} & = g^{\ov{b} a} \nabla_{a} \nabla_{\ov{\ell}} R_{i \ov{j} k \ov{b}} \\ \nonumber
& = g^{\ov{b} a} \nabla_{\ov{\ell}}  \nabla_{a} R_{i \ov{j} k \ov{b}} + g^{\ov{b}{a}}[ \nabla_{a}, \nabla_{\ov{\ell}} ] R_{i \ov{j} k \ov{b}} \\ \label{R1}
&  = \nabla_{\ov{\ell}} \nabla_k R_{i \ov{j}} - R^{\ov{b} \ \ \, a}_{\  \ov{\ell} k } R_{a \ov{b} i \ov{j}} + R^{\ov{b}}_{\ \, \ov{\ell}} R_{ k \ov{b} i \ov{j}} - R^{\ov{b} \ \ \, a}_{ \ \, \ov{\ell} i} R_{k \ov{b} a \ov{j}} + R^{\ov{b} \ \, \ov{a}}_{ \  \ov{\ell} \ \, \, \ov{j}} R_{k \ov{b}  i \ov{a}}.
\end{align}
And
\begin{align} \nonumber
\ov{\Delta} R_{i \ov{j} k \ov{\ell}} & = g^{\ov{b} a} \nabla_{\ov{b}} \nabla_k R_{i \ov{j} a \ov{\ell}} \\ \nonumber
& = g^{\ov{b} a} \nabla_k \nabla_{\ov{b}} R_{i \ov{j} a \ov{\ell}} + g^{\ov{b} a}[\nabla_{\ov{b}}, \nabla_k ] R_{i \ov{j} a \ov{\ell}} \\ \nonumber
& = \nabla_{\ov{\ell}} \nabla_k R_{i \ov{j}}  + [ \nabla_k, \nabla_{\ov{\ell}}] R_{i \ov{j}} + g^{\ov{b} a}[\nabla_{\ov{b}}, \nabla_k ] R_{i \ov{j} a \ov{\ell}} \\ \nonumber
& = \nabla_{\ov{\ell}} \nabla_k R_{i \ov{j}} - R_{k \ov{\ell} i}^{\ \ \ \, a} R_{a \ov{j}} + R_{k \ov{\ell} \ \, \ov{j}}^{\ \ \, \, \ov{b}} R_{i \ov{b}}
 \\ \label{R2} & \mbox{} + R_{k \ \, i}^{\ \, a \ \, b} R_{b \ov{j} a \ov{\ell}} - R_{k \ \ \, \ov{j}}^{\ \, a \ov{b}} R_{i \ov{b} a \ov{\ell}} + R_{k}^{\ \, b} R_{i \ov{j} b \ov{\ell}} - R_{k \ \ \,  \ov{\ell}}^{\ \,  a\ov{b}} R_{i \ov{j} a \ov{b}}.
\end{align}
Combining (\ref{ddtR}), (\ref{R1}) and (\ref{R2}) gives (\ref{evolveRm}) \qed
\end{proof}

In fact we do not need the precise formula (\ref{evolveRm}) in what follows, but merely the fact that it has the general form
\begin{equation} \label{generalform}
\ddt{} \textrm{Rm} = \frac{1}{2}\Delta_{\mathbb{R}} \textrm{Rm} - \nu \textrm{Rm} + \textrm{Rm} * \textrm{Rm} + \textrm{Rc} * \text{Rm}.
\end{equation}
To clarify notation:
 if $A$ and $B$ are tensors, we write $A*B$ for any linear combination of products of the tensors $A$ and $B$ formed by contractions on $A_{i_1 \cdots i_k}$ and $B_{j_1 \cdots j_{\ell}}$ using the metric $g$.   We are writing  $\textrm{Rc}$ for the Ricci tensor.

\begin{remark} \label{factorof2}  \emph{A word about notation.  The operator $\Delta_{\mathbb{R}}$ is  the usual `rough' Laplace operator associated to the Riemannian metric $g_{\mathbb{R}}$ defined in (\ref{gR}).  Hamilton defined his Ricci flow as $\ddt{} g_{ij} = -2R_{ij}$ precisely to  remove the factor of $\frac{1}{2}$ appearing in evolution equations such as (\ref{generalform}).  In real coordinates, the K\"ahler-Ricci flow we consider in these notes is $\ddt{} g_{ij} =-R_{ij}$.}
\end{remark}

\begin{lemma} \label{lemmaRi}
There exists a universal constant $C$ such that
\begin{equation} \label{Ri1}
\left( \ddt{} - \Delta \right) | \emph{Rm} |^2 \le - | \nabla \emph{Rm}|^2 - | \ov{\nabla} \emph{Rm}|^2 + C | \emph{Rm}|^3 - \nu | \emph{Rm}|^2,
\end{equation}
and, for all points of $M \times [0,T)$ where $| \emph{Rm} |$ is not zero,
\begin{equation} \label{Ri2}
\left( \ddt{} - \Delta \right) | \emph{Rm} | \le \frac{C}{2} | \emph{Rm}|^2 - \frac{\nu}{2} | \emph{Rm}|.
\end{equation}
\end{lemma}
\begin{proof}
The inequality (\ref{Ri1}) follows from (\ref{generalform}).  Next, note that
\begin{equation}
\left( \ddt{} - \Delta \right) | \textrm{Rm} | = \frac{1}{2 | \textrm{Rm}|} \left( \ddt{} - \Delta \right) |\textrm{Rm}|^2 + \frac{1}{4 | \textrm{Rm}|^3} g^{\ov{j} i} \nabla_i | \textrm{Rm}|^2 \nabla_{\ov{j}} | \textrm{Rm}|^2,
\end{equation}
and
\begin{equation} \label{eli}
g^{\ov{j} i} \nabla_i | \textrm{Rm}|^2 \nabla_{\ov{j}} | \textrm{Rm}|^2 \le  2 | \textrm{Rm}|^2 (| \nabla \textrm{Rm}|^2 + | \ov{\nabla} \textrm{Rm}|^2).
\end{equation}
Then (\ref{Ri2}) follows from (\ref{Ri1}) and (\ref{eli}). \qed
\end{proof}

We combine this result with the third order estimate from Section \ref{sect3rd} to obtain:

\begin{theorem} \label{theoremcurv1}
Let $\omega=\omega(t)$ solve  (\ref{nu}) and assume that there exists a constant $C_0>0$ such that
\begin{equation} \label{assumeomega3}
\frac{1}{C_0} \omega_0 \le \omega \le C_0 \omega_0.
\end{equation}
Then there exists a constant $C$ depending only on $C_0$ and $\omega_0$ such that
\begin{equation} \label{RC}
|\emph{Rm}|^2  \le C.
\end{equation}
In addition, there exists a constant $C'$ depending only on $C_0$ and $\omega_0$ such that
\begin{equation} \label{Rmgood}
\left( \ddt{} - \Delta \right) | \emph{Rm}|^2 \le -| \nabla \emph{Rm} |^2 - | \ov{\nabla} \emph{Rm}|^2 +C',
\end{equation}
\end{theorem}
\begin{proof}
From Theorem \ref{theoremSbound}, the quantity $S= | \nabla_{g_0} g|^2$ is uniformly bounded from above.
We compute the evolution of $Q = | \textrm{Rm}| + A S$ for a constant $A$.  From (\ref{SRm}) and (\ref{Ri2}), if $A$ is chosen to be sufficiently large, we obtain
\begin{equation}
\left( \ddt{} - \Delta \right) Q \le - | \textrm{Rm}|^2 +C',
\end{equation}
for a uniform constant $C'$.  Then the upper bound of $|\textrm{Rm}|^2$ follows from the maximum principle.  Finally, (\ref{Rmgood}) follows from (\ref{Ri1}).  \qed
\end{proof}

Moreover, once we have bounded curvature, it is a result of Hamilton \cite{H1} that bounds on all derivatives of curvature follow.   For convenience we change to a real coordinate system.  Writing $\nabla_{\mathbb{R}}$ for the covariant derivative with respect to $g$ as a Riemannian metric, we have:

\begin{theorem} \label{theoremcurv2}
Let $\omega=\omega(t)$ solve  (\ref{nu}) on $[0,T)$ with $0<T \le \infty$ and assume that there exists a constant $C>0$ such that
\begin{equation}
|\emph{Rm}|^2  \le C.
\end{equation}
Then there exist uniform constants $C_m$ for $m=1, 2, \ldots$ such that
\begin{equation}
| \nabla_{\mathbb{R}}^m \emph{Rm} |^2 \le C_m.
\end{equation}
\end{theorem}
\begin{proof}
We give a sketch of the proof and leave the details as an exercise to the reader.  We use a maximum principle argument due to Shi \cite{Shi} (see  \cite{CLN} for a good exposition).   In fact we do not need the full force of Shi's results, which are local, since  we are assuming a global curvature bound.

From Lemma \ref{lemmaevolveRm} and an induction argument (see Theorem 13.2 of \cite{H1})
\begin{equation} \label{corH}
\left( \ddt{} - \frac{1}{2}\Delta_{\mathbb{R}} \right) \nabla_{\mathbb{R}}^m \textrm{Rm} = \sum_{p+q=m} \nabla_{\mathbb{R}}^p \textrm{Rm}* \nabla_{\mathbb{R}}^q \textrm{Rm}.
\end{equation}
It follows that
\begin{equation}
\left( \ddt{} - \frac{1}{2} \Delta_{\mathbb{R}} \right) |  \nabla_{\mathbb{R}}^m \textrm{Rm} |^2 = -  | \nabla_{\mathbb{R}}^{m+1} \textrm{Rm} |^2 + \sum_{p+q=m} \nabla_{\mathbb{R}}^p \textrm{Rm}* \nabla_{\mathbb{R}}^q \textrm{Rm} * \nabla_{\mathbb{R}}^m \textrm{Rm}.
\end{equation}
Moreover, since $|\textrm{Rm}|^2$ is bounded we have from Lemma \ref{lemmaRi} that
\begin{equation}
\left( \ddt{} - \frac{1}{2} \Delta_{\mathbb{R}} \right) | \textrm{Rm}|^2 \le - | \nabla_{\mathbb{R}} \textrm{Rm}|^2 + C',
\end{equation}
for some uniform constant $C'$.  For the case $m=1$, if we set $Q = | \nabla_{\mathbb{R}} \textrm{Rm}|^2 + A | \textrm{Rm}|^2$ for $A>0$ sufficiently large then from (\ref{corH}),
\begin{equation}
\left( \ddt{} -  \frac{1}{2} \Delta_{\mathbb{R}} \right) Q \le - | \nabla_{\mathbb{R}} \textrm{Rm}|^2 + C'',
\end{equation}
and it follows from the maximum principle that $| \nabla_{\mathbb{R}} \textrm{Rm}|^2$ is uniformly bounded from above.  In addition,
\begin{equation}
\left( \ddt{} -  \frac{1}{2} \Delta_{\mathbb{R}} \right) | \nabla_{\mathbb{R}} \textrm{Rm}|^2 \le - | \nabla_{\mathbb{R}}^{2} \textrm{Rm} |^2 + C''',
\end{equation}
and an induction completes the proof.  \qed
\end{proof}

Next, we show that once we have a uniform bound on a metric evolving by the K\"ahler-Ricci flow, together with bounds on derivatives of curvature, then we have $C^{\infty}$ bounds for the metric.  Moreover, this result is local:

\begin{theorem} \label{interior}
Let $\omega = \omega(t)$ solve (\ref{nu}) on $U \times [0,T)$ with $0 \le T \le \infty$, where $U$ is an open subset of $M$.  Assume that there there exist constants $C_m$ for $m=0, 1, 2 \ldots$ such that
\begin{equation} \label{assumeomega4}
\frac{1}{C_0} \omega_0 \le \omega \le C_0 \omega_0, \quad S \le C_0 \quad \textrm{and} \quad  | \nabla_{\mathbb{R}}^m \emph{Rm} |^2 \le C_m.
\end{equation}
Then for any compact subset $K \subset U$ and for $m=1, 2, \ldots,$ there exist constants $C'_m$ depending only on $\omega_0$, $K$, $U$ and $C_m$ such that
\begin{equation}
\| \omega(t) \|_{C^m(K, g_0)} \le C'_m.
\end{equation}
\end{theorem}
\begin{proof}   This is a well-known result.  See \cite{CK}, for example, or the discussion in \cite{PSSW3}.  We give just a sketch of the proof following quite closely the arguments in \cite{ShW, SW2}.  It suffices to prove the result on the  ball $B$ say, in a fixed holomorphic coordinate chart.  We will obtain the $C^{\infty}$ estimates for $\omega(t)$ on a slightly smaller ball.
Fix a time $t \in (0, T]$.  Consider the equations
\begin{equation} \label{poisson}
\Delta_{\textrm{E}} g_{i \bar{j}} = - \sum_k R_{k \bar{k} i \bar{j}} + \sum_{k,p,q} g^{q \bar{p}} \partial_k g_{i \bar{q}} \partial_{\bar{k}} g_{p \bar{j}}=: Q_{i \bar{j}}.
\end{equation}
where $\Delta_{\textrm{E}} = \sum_k \partial_k \partial_{\bar{k}}$.  For each fixed $i, j$, we can regard (\ref{poisson}) as Poisson's equation $\Delta_{\textrm{E}} g_{i \bar{j}} = Q_{i \bar{j}}$.


Fix $p>2n$.  From our assumptions, each  $\| Q_{i \bar{j}} \|_{L^p(B)}$ is uniformly bounded.  Applying the standard elliptic estimates  (see  Theorem 9.11 of \cite{GT} for example) to (\ref{poisson}) we see that the Sobolev norm $\| g_{i \bar{j}} \|_{L^p_2}$ is uniformly bounded on a slightly smaller ball.  From now on, the estimates that we state will always be modulo shrinking the ball slightly.
  Morrey's embedding theorem (Theorem 7.17 of \cite{GT}) gives that  $\| g_{i \bar{j}} \|_{C^{1+ \beta}}$ is uniformly bounded for some $0<\beta<1$.

The key observation we now need is as follows:     the $m$th derivative of $Q_{i \bar{j}}$  can be written in the form
$A * B$
where each $A$ or $B$ represents either a  covariant derivative of $\textrm{Rm}$  or a quantity involving derivatives of $g$ up to order at most $m+1$.  Hence if $g$ is uniformly bounded in $C^{m+1+ \beta}$  then each $Q_{i \bar{j}}$ is uniformly bounded in $C^{m+ \beta}$.

Applying this observation with $m=0$ we see that each $\| Q_{i\bar{j}} \|_{C^{\beta}}$ is uniformly bounded.
The standard Schauder estimates  (see Theorem 4.8 of \cite{GT}) give that $\| g_{i \bar{j}} \|_{C^{2+\beta}}$ is uniformly bounded.

We can now apply a bootstrapping argument.  Applying the observation with $m=1$ we see that $Q_{i\bar{j}}$  is uniformly bounded in $C^{1+\beta}$, and so on.  This completes the proof. \qed
\end{proof}

Combining Theorems \ref{theoremcurv1}, \ref{theoremcurv2} and \ref{interior}, we obtain:

\begin{corollary}  \label{choe}
Let $\omega = \omega(t)$ solve (\ref{nu}) on $M \times [0,T)$ with $0 \le T \le \infty$.  Assume that there  exists a constant $C_0$  such that
\begin{equation} \label{assumeomega5}
\frac{1}{C_0} \omega_0 \le \omega \le C_0 \omega_0.
\end{equation}
Then for $m=1, 2, \ldots$, there exist uniform constants $C_m$  such that
\begin{equation}
\| \omega(t) \|_{C^m(g_0)} \le C_m.
\end{equation}
\end{corollary}

In fact, there is a local version of Corollary \ref{choe}.  Although we will not actual make use of it in these lecture notes, we state here the result:

\begin{theorem}
Let $\omega = \omega(t)$ solve (\ref{nu}) on $U \times [0,T)$ with $0 \le T \le \infty$, where $U$ is an open subset of $M$.  Assume that there there exists a constant $C_0$ for such that
\begin{equation} \label{assumeomega6}
\frac{1}{C_0} \omega_0 \le \omega \le C_0 \omega_0.
\end{equation}
Then for any compact subset $K \subset U$ and for $m=1, 2, \ldots,$ there exist constants $C'_m$ depending only on $\omega_0$, $K$ and $U$  such that
\begin{equation}
\| \omega(t) \|_{C^m(K, g_0)} \le C'_m.
\end{equation}
\end{theorem}
\begin{proof}
This can either be proved using the Schauder estimates of Evans-Krylov \cite{E, K} (see also \cite{CLN, G}) or using local maximum principle arguments \cite{ShW}.  We omit the proof.  \qed
\end{proof}

\pagebreak
\section{Maximal existence time for the K\"ahler-Ricci flow} \label{maximal}

In this section we identify the maximal existence time for a smooth solution of the K\"ahler-Ricci flow.  To do this, we rewrite the K\"ahler-Ricci flow as a parabolic complex Monge-Amp\`ere equation.

\subsection{The parabolic Monge-Amp\`ere equation} \label{sectpma}

Let
 $\omega=\omega(t)$ be a solution of the K\"ahler-Ricci flow
 \begin{equation} \label{krf}
\ddt{} \omega  = - \Ric(\omega), \qquad \omega|_{t=0} = \omega_0.
\end{equation}
As long as the solution exists, the cohomology class $[\omega(t)]$ evolves by
\begin{equation}
\frac{d}{dt} [\omega(t)] = - c_1(M), \qquad [\omega(0)] = [\omega_0],
\end{equation}
and solving this ordinary differential equation gives $[\omega(t)] = [\omega_0] - t c_1(M)$.  Immediately we see that a necessary condition for the K\"ahler-Ricci flow to exist for $t \in [0,t')$ is that $[\omega_0] - t c_1(M)>0$ for $t \in [0,t')$.  This necessary condition is in fact sufficient.  If we define
\begin{equation} \label{T}
T = \sup \{ t>0 \ | \ [\omega_0] - t c_1(M) >0 \},
\end{equation}
then we have:

\begin{theorem} \label{longtime}
There exists a unique maximal solution $g(t)$ of the K\"ahler-Ricci flow (\ref{krf}) for $t \in [0,T)$.
\end{theorem}

This theorem was proved by Cao \cite{Cao} in the special case when $c_1(M)$ is zero or definite.    In this generality, the result is due to Tian-Zhang \cite{TZha}.  Weaker versions appeared earlier in the work of Tsuji (see \cite{Ts1} and Theorem 8 of \cite{Ts2}).

We now begin the proof of Theorem \ref{longtime}.  Fix $T'<T$.  We will show that there exists a solution to (\ref{krf}) on $[0,T')$.
First we observe that  (\ref{krf}) can be rewritten as a parabolic complex Monge-Amp\`ere equation.

To do this, we need to choose reference metrics $\hat{\omega}_t$ in the cohomology classes $[\omega_0] - tc_1(M)$.  Since $[\omega_0] - T' c_1(M)$ is a K\"ahler class, there exists a K\"ahler form $\eta$ in $[\omega_0] - T'c_1(M)$.  We choose our family of reference metrics $\hat{\omega}_t$ to be the linear path of metrics between $\omega_0$ and $\eta$.    Namely, define
\begin{equation} \label{chi}
\chi = \frac{1}{T'}(\eta - \omega_0) \in - c_1(M),
\end{equation}
and
\begin{equation}
\hat{\omega}_t = \omega_0 + t \chi = \frac{1}{T'} (( T'-t) \omega_0 + t\eta) \in [\omega_0] - t c_1(M).
\end{equation}
Fix a volume form $\Omega$ on $M$ with
\begin{equation}
\ddbar \log \Omega = \chi = \ddt{} \hat{\omega}_t \in - c_1(M),
\end{equation}
which exists by the discussion in Section \ref{sectdol}.
Notice that here we are abusing notation somewhat  by writing $\ddbar \log \Omega$.  To clarify, we mean that if the volume form $\Omega$ is written in local coordinates $z^i$ as $$\Omega = a(z^1, \ldots, z^n) (\sqrt{-1})^n dz^1 \wedge d\ov{z^1} \wedge \cdots \wedge dz^n \wedge d\ov{z^n},$$
for a locally defined smooth positive function $a$ then we define
$
\ddbar \log \Omega = \ddbar \log a.
$
Although the function $a$ depends on the choice of holomorphic coordinates, the $(1,1)$-form $\ddbar \log a$ does not, as the reader can easily verify.

We now consider the \emph{parabolic complex Monge-Amp\`ere equation}, for $\varphi= \varphi(t)$ a real-valued function on $M$,
\begin{equation} \label{pcma}
\ddt{ \varphi} = \log \frac{ (\hat{\omega}_t + \ddbar \varphi)^n}{\Omega}, \qquad \hat{\omega}_t+ \ddbar \varphi>0 , \qquad \varphi|_{t=0} =0.
\end{equation}
This equation is equivalent to the K\"ahler-Ricci flow (\ref{krf}).  Indeed,
given a smooth solution $\varphi$ of (\ref{pcma}) on $[0,T')$, we can obtain a solution $\omega= \omega(t)$ of (\ref{krf}) on $[0,T')$ as follows.  Define  $\omega(t) = \hat{\omega}_t  + \ddbar \varphi$
and observe that $\omega(0) = \hat{\omega}_0=\omega_0$ and
\begin{equation}
\ddt{} \omega = \ddt{} \hat{\omega}_t + \ddbar \left( \ddt{\varphi}  \right) = - \Ric(\omega),
\end{equation}
as required.  Conversely, suppose that $\omega= \omega(t)$ solves (\ref{krf}) on $[0,T')$.  Then since $\hat{\omega}_t \in [\omega(t)]$, we can apply the $\partial \ov{\partial}$-Lemma to find a  family of potential functions $\tilde{\varphi}(t)$ such that $\omega(t) = \hat{\omega}_t + \ddbar \tilde{\varphi}(t)$ and $\int_M \tilde{\varphi}(t) \omega_0^n=0$.  By standard elliptic regularity theory the family $\tilde{\varphi}(t)$ is smooth on $M \times [0,T')$.  Then
\begin{equation}
\ddbar \log \omega^n = \ddt{} \omega = \ddbar \log \Omega  + \ddbar \left( \ddt{ \tilde{\varphi}} \right),
\end{equation}
and since the only pluriharmonic functions on $M$ are the constants, we see that
$$\ddt{\tilde{\varphi}} = \log \frac{ \omega^n}{\Omega} + c(t),$$
for some smooth function $c : [0,T') \rightarrow \mathbb{R}$.  Now set $\varphi(t) = \tilde{\varphi}(t) - \int_0^t c(s)ds - \tilde{\varphi}(0)$, noting that
since $\omega(0) = \omega_0$ the function $\tilde{\varphi}(0)$ is  constant.  It follows that $\varphi= \varphi(t)$ solves the parabolic complex Monge-Amp\`ere equation (\ref{pcma}).

To prove Theorem \ref{longtime} then, it suffices to study (\ref{pcma}).  Since the linearization of the right hand side of (\ref{pcma}) is the Laplace operator $\Delta_{g(t)}$, which is  elliptic, it follows that (\ref{pcma}) is a strictly parabolic (nonlinear) partial differential equation for $\varphi$.  The standard parabolic theory \cite{Lie} gives a unique maximal solution of (\ref{pcma}) for some time interval $[0,\Tmax)$ with $0 < \Tmax \le \infty$.
We may assume without loss of generality that $\Tmax < T'$.  We will then obtain a contradiction by showing that a solution of (\ref{pcma}) exists beyond $\Tmax$.  This will be done in the next two subsections.

\subsection{Estimates for the potential and the volume form}

We assume now that we have a solution $\varphi=\varphi(t)$ to the parabolic complex Monge-Amp\`ere equation (\ref{pcma}) on $[0,\Tmax)$, for $0 < \Tmax < T' < T$.  Our goal is to establish uniform estimates for $\varphi$ on $[0, \Tmax)$.    In this subsection we will prove a $C^0$ estimate for $\varphi$ and a lower bound for the volume form.

Note that $\hat{\omega}_t$ is a family of smooth K\"ahler forms on the closed interval $[0,\Tmax]$.  Hence by compactness we have uniform bounds on $\hat{\omega}_t$ from above and below (away from zero).

\begin{lemma} \label{lemmaphibound} There exists a uniform $C$ such that for all $t \in [0,\Tmax)$,
\begin{equation}
 \| \varphi(t) \|_{C^0(M)} \le C.
\end{equation}
\end{lemma}
\begin{proof}
For the upper bound of $\varphi$, we will apply the maximum principle to $\theta := \varphi - At$ for $A>0$ a uniform constant to be determined later.  From (\ref{pcma}) we have
\begin{equation}
\ddt{ \theta} = \log \frac{ (\hat{\omega}_t + \ddbar \theta)^n}{\Omega} - A.
\end{equation}
Fix $t' \in (0, \Tmax)$.  Since $M \times [0,t']$ is compact, $\theta$ attains a maximum at some point $(x_0, t_0) \in M \times [0,t']$.  We claim that if $A$ is  sufficiently large we have $t_0=0$.

Otherwise $t_0>0$.  Then by Proposition \ref{pmp},
at $(x_0, t_0)$,
\begin{equation}
0 \le \ddt{\theta} = \log \frac{(\hat{\omega}_{t_0} + \ddbar \theta)^n}{\Omega} - A \le  \log \frac{\hat{\omega}_{t_0}^n}{\Omega} - A \le -1,
\end{equation}
a contradiction, where we have chosen $A \ge 1 + \sup_{M \times [0, \Tmax]} \log (\hat{\omega}_t^n/ \Omega)$.   Hence we have proved the claim that $t_0=0$, giving  $\sup_{M \times [0,t']} \theta \le \sup_M \theta|_{t=0} = 0$ and thus
\begin{equation}
\varphi(x,t) \le At \le A \Tmax, \qquad \textrm{for } (x,t) \in M \times [0,t'].
\end{equation}
Since $t' \in (0, \Tmax)$ was arbitrary, this gives a uniform upper bound for $\varphi$ on $[0,\Tmax)$.

We apply a similar argument to $\psi = \varphi + Bt$ for $B$ a positive constant with $B\ge 1 - \inf_{M \times [0,\Tmax]} \log(\hat{\omega}_t^n/\Omega)$ and obtain
\begin{equation}
\varphi(x,t) \ge - B \Tmax, \qquad \textrm{for } (x,t) \in M \times [0,t'],
\end{equation}
giving the lower bound.  \qed
\end{proof}

Next we prove a lower bound for the volume form along the flow, or equivalently a lower bound for $\dot{\varphi}=\partial \varphi/\partial t$.  This argument is due to Tian-Zhang \cite{TZha}.

\begin{lemma} \label{lemmaphidot}
There exists a uniform $C>0$ such that on $M \times  [0, T_{\emph{max}})$,
\begin{equation}
\frac{1}{C} \Omega \le \omega^n(t) \le C \Omega,
\end{equation}
or equivalently, $\| \dot{\varphi} \|_{C^0}$ is uniformly bounded.
\end{lemma}
\begin{proof}
The upper bound of $\omega^n$ follows from part (i) of Corollary \ref{volform}.  Note that since this is equivalent to an upper bound of $\dot{\varphi}$, we have given an alternative proof of the upper bound part of Lemma \ref{lemmaphibound}.

For the lower bound of $\omega^n$, differentiate (\ref{pcma}):
\begin{equation} \label{evolvephit}
\ddt{ \dot{\varphi}} = \Delta \dot{\varphi} + \tr{\omega}{\chi},
\end{equation}
where we recall that $\chi = \partial \hat{\omega}_t/\partial t$ is defined by (\ref{chi}).  Define a quantity $Q= (T' - t) \dot{\varphi} + \varphi +nt$ and compute using (\ref{evolvephit}),
\begin{equation} \label{evolveQ1}
\left( \ddt{} - \Delta \right) Q = (T'-t) \tr{\omega}{\chi} + n - \Delta \varphi = \tr{\omega}{ (\hat{\omega}_t + (T'-t) \chi)} = \tr{\omega}{\hat{\omega}_{T'}} >0,
\end{equation}
where we have used the fact that
\begin{equation}
\Delta \varphi = \tr{\omega}{(\omega - \hat{\omega}_t)} = n - \tr{\omega}{\hat{\omega}_t}.
\end{equation}
Then by the maximum principle (Proposition \ref{propheat}), $Q$ is uniformly bounded from below on $M \times [0,\Tmax)$ by its infimum at the initial time.  Thus
\begin{equation}
(T'-t) \dot{\varphi} + \varphi +nt \ge T' \inf_M  \, \log \frac{\omega_0^n}{\Omega}, \qquad \textrm{on } M \times [0,\Tmax),
\end{equation}
and since $\varphi$ is uniformly bounded from Lemma \ref{lemmaphibound} and $T'-t \ge T' - \Tmax>0$, this gives the desired lower bound of $\dot{\varphi}$. \qed
\end{proof}

\pagebreak[3]
\subsection{A uniform bound for the evolving metric}

Again we assume that we have a solution $\varphi=\varphi(t)$ to  (\ref{pcma}) on $[0,\Tmax)$, for $0 < \Tmax < T' < T$.  From Lemma \ref{lemmaphibound}, we have a uniform bound for $\| \varphi\|_{C^0(M)}$ and we will use this together with Proposition \ref{propChat} to obtain an upper bound for the quantity $\tr{\omega_0}{\omega}$ on $[0,\Tmax)$.  This argument is similar to those in \cite{A, Y2} (see also \cite{Cao} and Lemmas \ref{lemmametricbdn} and \ref{lemmayauC2} below).  We will then  complete the proof of Theorem \ref{longtime}.

\begin{lemma} \label{lemmatr}
There exists a uniform $C$ such that on $M \times  [0,T_{\emph{max}})$,
\begin{equation} \label{trub}
\emph{tr}_{\omega_0}{\, \omega} \le C.
\end{equation}
\end{lemma}

\begin{proof} We consider the quantity
\begin{equation}
Q = \log \tr{\omega_0}{\omega} - A \varphi,
\end{equation}
for $A>0$ a uniform constant to be determined later.  For a fixed $t' \in (0, \Tmax)$, assume that $Q$ on $M \times [0,t']$ attains a maximum at a point $(x_0, t_0)$.  Without loss of generality, we may suppose that $t_0 >0$.  Then at $(x_0, t_0)$, applying Proposition \ref{propChat} with $\hat{\omega} = \omega_0$,
\begin{align} \nonumber
0 \le \left( \ddt{} - \Delta \right) Q & \le C_0 \tr{\omega}{\omega_0} - A \dot{\varphi} + A \Delta \varphi \\
& = \tr{\omega}{(C_0\omega_0 - A \hat{\omega}_{t_0}}) - A \log \frac{\omega^n}{\Omega}  + An,
\end{align}
for $C_0$ depending only on the lower bound of the bisectional curvature of $g_0$.  Choose $A$ sufficiently large so that
$A \hat{\omega}_{t_0} - (C_0 +1) \omega_0$ is K\"ahler on $M$.  Then
\begin{equation}
\tr{\omega}{(C_0\omega_0 - A\hat{\omega}_{t_0})} \le - \tr{\omega}{\omega_0},
\end{equation}
and
so at $(x_0, t_0)$,
\begin{equation}
\tr{\omega}{\omega_0} + A \log \frac{\omega^n}{\Omega} \le An,
\end{equation}
and hence
\begin{equation} \label{conc}
\tr{\omega}{\omega_0} + A \log \frac{\omega^n}{\omega_0^n} \le C,
\end{equation}
for some uniform constant $C$.  At $(x_0, t_0)$, choose coordinates so that
\begin{equation} \label{coords}
(g_0)_{i \ov{j}} = \delta_{i j} \quad \textrm{and} \quad g_{i\ov{j}} = \lambda_i \delta_{ij}, \quad \textrm{for } i, j=1, \ldots, n,
\end{equation}
for positive $\lambda_1, \ldots, \lambda_n.$
Then (\ref{conc}) is precisely
\begin{equation}
\sum_{i=1}^n \left( \frac{1}{\lambda_i} + A \log \lambda_i \right) \le C.
\end{equation}
Since the function $x \mapsto \frac{1}{x} + A \log x$ for $x>0$ is uniformly bounded from below, we have (for a different $C$),
\begin{equation}
 \left( \frac{1}{\lambda_i} + A \log \lambda_i \right) \le C, \quad \textrm{for } i=1, \ldots, n.
\end{equation}
Then $A \log \lambda_i \le C$, giving a uniform upper bound for $\lambda_i$ and hence $(\tr{\omega_0}{\omega})(x_0, t_0)$.  Since $\varphi$ is uniformly bounded on $M \times [0, \Tmax)$ we see that $Q(x_0, t_0)$ is uniformly bounded from above.  Hence $Q$ is bounded from above on $M \times [0,t']$ for any $t' < \Tmax$.  Using again that $\varphi$ is uniformly bounded we obtain the required estimate (\ref{trub}).  \qed
\end{proof}

Note that we did not make use of the bound on $\dot{\varphi}$ in the above argument.  By doing so we could have simplified the proof slightly.  However, it turns out that the argument of Lemma \ref{lemmatr} will be useful later (see Lemma \ref{trick2} and Section \ref{sectfinite} below) where we do not have a uniform lower bound of $\dot{\varphi}$.

As a consequence of Lemma \ref{lemmatr}, we have:

\begin{corollary} \label{cmetricbound}
There exists a uniform $C>0$ such that on $M \times  [0,\Tmax)$,
\begin{equation}
\frac{1}{C} \omega_0 \le \omega \le C \omega_0.
\end{equation}
\end{corollary}
\begin{proof}
The upper bound follows from Lemma  \ref{lemmatr}.   For the lower bound,
\begin{equation} \label{gam}
\tr{\omega}{\omega_0} \le \frac{1}{(n-1)!} (\tr{\omega_0}{\omega})^{n-1} \frac{\omega_0^n}{\omega^n} \le C,
\end{equation}
using Lemma \ref{lemmaphidot}.  To verify the first inequality of (\ref{gam}), choose coordinates as in (\ref{coords}) and observe that
\begin{equation}
\frac{1}{\lambda_1} +\cdots + \frac{1}{\lambda_n} \le \frac{1}{(n-1)!} \frac{ (\lambda_1 + \cdots + \lambda_n)^{n-1}}{\lambda_1 \cdots \lambda_n},
\end{equation}
for positive $\lambda_i$.  \qed
\end{proof}

We can now finish the proof of Theorem \ref{longtime}.

\begin{proof}[Proof of Theorem \ref{longtime}]
  Combining Corollary \ref{cmetricbound} with Corollary \ref{choe}, we obtain uniform $C^{\infty}$ estimates for $g(t)$ on $[0, \Tmax)$.  Hence as $t \rightarrow \Tmax$, the metrics $g(t)$ converge in $C^{\infty}$ to a smooth K\"ahler metric $g(\Tmax)$ and thus we obtain a smooth solution to the K\"ahler-Ricci flow on $[0, \Tmax]$.  But we have already seen from Theorem \ref{hamilton}  (or by the discussion at the end of Section \ref{sectpma}) that we can always find a smooth solution of the K\"ahler-Ricci flow on some, possibly short, time interval with  any initial K\"ahler metric.  Applying this to $g(\Tmax)$, we obtain a solution of the K\"ahler-Ricci flow $g(t)$ on $[0, \Tmax + \ve)$ for $\ve>0$.  But this contradicts the definition of $\Tmax$, and completes the proof of Theorem \ref{longtime}.
  \qed \end{proof}

\pagebreak

\section{Convergence of the flow in the cases $c_1(M)<0$ and $c_1(M)=0$} \label{sectn0}

In this section we show that  the K\"ahler-Ricci flow converges, after appropriate normalization, to a K\"ahler-Einstein metric in the cases $c_1(M)<0$ and $c_1(M)=0$.  This was originally proved by Cao \cite{Cao} and makes use of parabolic versions of estimates due to Yau and Aubin  \cite{A, Y2} and also Yau's well-known $C^0$ estimate for the complex Monge-Amp\`ere equation \cite{Y2}.

\subsection{The normalized K\"ahler-Ricci flow when $c_1(M)<0$} \label{sectnkrf}

We first consider the case of a manifold $M$ with $c_1(M)<0$.     We restrict to the case when $[\omega_0]= -c_1(M)$.  By Theorem \ref{longtime} we have a solution to the K\"ahler-Ricci flow (\ref{krf}) for $t \in [0,\infty)$.  The K\"ahler class $[\omega(t)]$ is given by $(1+t)[\omega_0]$ which diverges as $t \rightarrow \infty$.  To avoid this we consider instead the \emph{normalized K\"ahler-Ricci flow}
\begin{equation} \label{nkrf}
\ddt{} \omega  = - \Ric(\omega) - \omega, \qquad \omega|_{t=0} = \omega_0.
\end{equation}
This is just a rescaling of (\ref{krf}) and we have a solution $\omega(t)$ to (\ref{nkrf})  for all time.   Indeed if $\tilde{\omega}(s)$ solves $\frac{\partial}{\partial s} \tilde{\omega} (s) = - \Ric (\tilde{\omega}(s))$ for $s \in [0,\infty)$ then $\omega(t) = \tilde{\omega}(s)/(s+1)$ with $t=\log (s+1)$ solves (\ref{nkrf}).  Conversely, given a solution to (\ref{nkrf}) we can rescale to obtain a solution to (\ref{krf}).

Since we have chosen $[\omega_0] = - c_1(M)$, we immediately see that $[\omega(t)] = [\omega_0]$ for all $t$.
 The following result is due to Cao \cite{Cao}.

\begin{theorem} \label{thmc1n}  The solution $\omega=\omega(t)$ to (\ref{nkrf}) converges in $C^{\infty}$ to the unique K\"ahler-Einstein metric $\omega_{\emph{KE}} \in -c_1(M)$.
\end{theorem}

We recall that a \emph{K\"ahler-Einstein metric} is a K\"ahler metric $\oke$ with $\Ric(\oke) = \mu \oke$ for some constant $\mu \in \mathbb{R}$.  If $\oke \in -c_1(M)$ then we necessarily have $\mu=-1$.  The existence of a K\"ahler-Einstein metric on $M$ with $c_1(M)<0$ is due to Yau \cite{Y2} and Aubin \cite{A} independently.

The uniqueness of $\oke \in -c_1(M)$ is due to Calabi \cite{C0} and follows from the maximum principle.  Indeed, suppose $\oke', \oke \in -c_1(M)$ are both K\"ahler-Einstein.  Writing  $\oke' = \oke+ \ddbar \varphi$, we have $\Ric(\oke') = -\oke' =   \Ric(\oke) - \ddbar \varphi$  and hence
\begin{equation}
\log \frac{ (\oke + \ddbar \varphi)^n}{\oke^n} = \varphi + C,
\end{equation}
for some constant $C$.  By considering the maximum and minimum values of $\varphi + C$ on $M$ we see that $\varphi + C=0$ and hence $\oke=\oke'$.

To prove Theorem \ref{thmc1n},
 we reduce (\ref{nkrf}) to a parabolic complex Monge-Amp\`ere equation as in the previous section.  Let $\Omega$ be a volume form on $M$ satisfying
\begin{equation}
\ddbar \log \Omega = \omega_0 \in -c_1(M), \qquad \int_M \Omega = \int_M \omega_0^n.
\end{equation}
Then we consider the \emph{normalized parabolic complex Monge-Amp\`ere equation},
\begin{equation} \label{npcma}
\ddt{ \varphi} = \log \frac{ (\omega_0 + \ddbar \varphi)^n}{\Omega} - \varphi, \qquad \omega_0+ \ddbar \varphi>0 , \qquad \varphi|_{t=0} =0.
\end{equation}
Given a solution $\varphi=\varphi(t)$ of (\ref{npcma}), the metrics $\omega = \omega_0 + \ddbar \varphi$ solve (\ref{nkrf}).  Conversely, as in Section \ref{sectpma},
 given a solution $\omega=\omega(t)$ of (\ref{nkrf}) we can obtain via the $\partial\ov{\partial}$-Lemma a solution $\varphi=\varphi(t)$ of (\ref{npcma}).

We wish to obtain estimates for $\varphi$ solving (\ref{npcma}).  First:

\begin{lemma} \label{lemmaex}  We have
\begin{enumerate}
\item[(i)] There exists a uniform constant $C$ such that for $t$ in  $[0,\infty)$,
\begin{equation}
\| \dot{\varphi} (t) \|_{C^0(M)} \le C e^{-t}.
\end{equation}
\item[(ii)]  There exists a continuous real-valued function $\varphi_{\infty}$ on $M$ such that for $t$ in $[0, \infty)$,
\begin{equation} \label{eb}
\| \varphi(t) - \varphi_{\infty} \|_{C^0(M)} \le C e^{-t}
\end{equation}
\item[(iii)] $\| \varphi(t) \|_{C^0(M)}$ is uniformly bounded for $t \in [0, \infty)$.
\item[(iv)] There exists a uniform constant $C'$ such that on $M \times [0,\infty)$, the volume form of $\omega=\omega(t)$ satisfies
\begin{equation}
\frac{1}{C'} \omega_0^n \le \omega^n  \le C' \omega_0^n.
\end{equation}
\end{enumerate}
\end{lemma}
\begin{proof}  Compute
\begin{equation}
\ddt{\dot{\varphi}} = \Delta \dot{\varphi} - \dot{\varphi},
\end{equation}
and hence
\begin{equation}
\ddt{} (e^t \dot{\varphi}) = \Delta (e^t \dot{\varphi}).
\end{equation}
Then (i) follows from the maximum principle (Proposition \ref{propheat}).   For (ii), let $s, t \ge 0$ and $x$ be in $M$.  Then
\begin{equation} \label{ts}
| \varphi(x,s) - \varphi(x,t)| = \left| \int_t^s \dot{\varphi}(x,u) du \right| \le \int_t^s |\dot{\varphi}(x,u)| du \le \int_t^s Ce^{-u} du = C(e^{-t} - e^{-s}),
\end{equation}
which shows that $\varphi(t)$ converges uniformly to some continuous function $\varphi_{\infty}$ on $M$.  Taking the limit in (\ref{ts}) as $s \rightarrow \infty$ gives (ii).  (iii) follows immediately from (ii).  (iv) follows from (\ref{npcma}) together with (i) and (iii). \qed
\end{proof}

We use the $C^0$ bound on $\varphi$ to obtain an upper bound on the evolving metric.

\begin{lemma} \label{lemmametricbdn}
There exists a uniform constant $C$ such that on $M \times [0,\infty)$, $\omega=\omega(t)$ satisfies
\begin{equation}
\frac{1}{C} \omega_0 \le \omega \le C \omega_0.
\end{equation}
\end{lemma}
\begin{proof}
By part (iv) of Lemma \ref{lemmaex} and the argument of Corollary \ref{cmetricbound}, it suffices to obtain a uniform upper bound for $\tr{\omega_0}{\omega}$.


Applying Proposition \ref{propChat},
\begin{equation} \label{mod}
\left( \ddt{} - \Delta \right) \log \tr{\omega_0}{\omega} \le C_0 \tr{\omega}{\omega_0} - 1,
\end{equation}
for $C_0$ depending only on $g_0$.   We apply the maximum principle to the quantity $Q = \log \tr{\omega_0}{\omega} - A \varphi$ as in the proof of Lemma \ref{lemmatr}, where $A$ is to be chosen later.   We have
\begin{equation}
\left( \ddt{} - \Delta \right) Q \le C_0 \tr{\omega}{\omega_0} - 1 - A \dot{\varphi} + An - A \tr{\omega}{\omega_0}.
\end{equation}
Assume that $Q$ achieves a maximum at a point $(x_0, t_0)$ with $t_0>0$.
Choosing $A=C_0+1$ and using the fact that $\dot{\varphi}$ is uniformly bounded, we see that $\tr{\omega}{\omega_0}$ is uniformly bounded at $(x_0, t_0)$.  Arguing as in  (\ref{gam}), we have,
\begin{equation} \label{trtr}
(\tr{\omega_0}{\omega}) (x_0, t_0) \le \frac{1}{(n-1)!} \left( \tr{\omega}{\omega_0} \right)^{n-1}(x_0, t_0) \frac{\omega^n}{\omega_0^n} (x_0, t_0) \le C,
\end{equation}
using part (iv) of Lemma \ref{lemmaex}.   Since $\varphi$ is uniformly bounded, this shows that $Q$ is bounded from above at $(x_0, t_0)$.  Hence $\tr{\omega_0}{\omega}$ is uniformly bounded from above.  \qed
 \end{proof}

We can now complete the proof of Theorem \ref{thmc1n}.  By Corollary \ref{choe} we have uniform $C^{\infty}$ estimates on $\omega(t)$.  Since $\varphi(t)$ is  bounded in $C^0$ it follows that we have uniform $C^{\infty}$ estimates on $\varphi(t)$.  Recall that $\varphi(t)$ converges uniformly to a continuous function $\varphi_{\infty}$ on $M$ as $t \rightarrow \infty$.    By the Arzela-Ascoli Theorem and the uniqueness of limits, it follows immediately that there exist times $t_k \rightarrow \infty$ such that the sequence of functions $\varphi(t_k)$ converges in $C^{\infty}$ to $\varphi_{\infty}$, which is smooth.   In fact we have this convergence without passing to a subsequence.  Indeed, suppose not.  Then  there exists an integer $k$, an $\ve>0$ and a sequence of times $t_i \rightarrow \infty$ such that
\begin{equation}
 \| \varphi(t_i) - \varphi_{\infty} \|_{C^k(M)} \ge \ve, \quad \textrm{for all } i.
\end{equation}
But since  $\varphi(t_i)$ is a sequence of functions with uniform $C^{k+1}$ bounds we apply the Arzela-Ascoli Theorem to obtain a subsequence $\varphi(t_{i_j})$ which converges in $C^k$ to $\varphi_{\infty}'$, say, with
\begin{equation}
\| \varphi'_{\infty} - \varphi_{\infty} \|_{C^k(M)} \ge \ve,
\end{equation}
so that $\varphi'_{\infty} \neq \varphi_{\infty}$.  But $\varphi({t_{i_j}})$ converges uniformly to $\varphi_{\infty}$, a contradiction.  Hence $\varphi(t)$ converges to $\varphi_{\infty}$ in $C^{\infty}$ as $t\rightarrow \infty$.

It remains to show that the limit metric $\omega_{\infty} = \omega_0 +\ddbar \varphi_{\infty}$ is K\"ahler-Einstein.  Since from Lemma \ref{lemmaex}, $\dot{\varphi}(t) \rightarrow 0$
 as $t \rightarrow \infty$, we can take a limit as $t \rightarrow \infty$ of (\ref{npcma}) to obtain
 \begin{equation}
 \log \frac{\omega_{\infty}^n}{\Omega} - \varphi_{\infty} =0,
 \end{equation}
 and applying $\ddbar$ to both sides of this equation gives that $\Ric(\omega_{\infty}) = -\omega_{\infty}$ as required.  This completes the proof of Theorem \ref{thmc1n}.

\subsection{The case of $c_1(M)=0$: Yau's zeroth order estimate}

In this section we discuss the case of the K\"ahler-Ricci flow on a K\"ahler manifold $(M, g_0)$ with vanishing first Chern class.  Unlike the case of $c_1(M)<0$ dealt with above, there will be no restriction on the K\"ahler class $[\omega_0]$.

By Theorem \ref{longtime}, there is a solution $\omega(t)$ of the K\"ahler-Ricci flow (\ref{krf}) for $t \in [0,\infty)$ and we have $[\omega(t)] = [\omega_0]$.  The following result is due to Cao \cite{Cao} and makes use of Yau's celebrated zeroth order estimate, which we will describe in this subsection.

\begin{theorem} \label{thmc0}
The solution $\omega(t)$ to (\ref{krf}) converges in $C^{\infty}$ to the unique K\"ahler-Einstein metric $\oke \in [\omega_0]$.
\end{theorem}

Since $c_1(M)=0$,  the K\"ahler-Einstein metric $\oke$ must be K\"ahler-Ricci flat (if $\Ric(\oke) = \mu \oke$ then $c_1(M) =  [\mu \oke] = 0$ implies $\mu=0$).  Note that, as Theorem \ref{thmc0} implies,  there is a unique K\"ahler-Einstein metric in \emph{every} K\"ahler class on $M$.

The uniqueness part of the argument is due to Calabi  \cite{C0}.   Suppose $\oke' = \oke + \ddbar \varphi$ is another K\"ahler-Einstein metric in the same cohomology class.  Then the equation $\Ric(\oke') = \Ric(\oke)$ gives
\begin{equation}
\log \frac{\oke'^n}{\oke^n} = C,
\end{equation}
for some constant $C$.  Exponentiating and then integrating gives $C=1$ and hence $\oke'^n = \oke'^n$.   Then compute, using integration by parts,
\begin{align} \nonumber
0 = \int_M \varphi (\oke^n - \oke'^n)  &= - \int_M \varphi \ddbar \varphi \wedge (\sum_{i=0}^{n-1} \oke^i \wedge \oke'^{n-1-i}) \\ \nonumber
& = \int_M \frac{\sqrt{-1}}{2\pi} \partial \varphi \wedge \ov{\partial} \varphi \wedge (\sum_{i=0}^{n-1} \oke^i \wedge \oke'^{n-1-i}) \\ \label{eunique}
& \ge \frac{1}{n} \int_M | \partial \varphi |^2_{\oke} \oke^n,
\end{align}
which implies that $\varphi$ is constant and hence $\oke=\oke'$.

As usual, we reduce (\ref{krf}) to a parabolic complex Monge-Amp\`ere equation.  Since $c_1(M)=0$ there exists a unique volume form $\Omega$ satisfying
\begin{equation}
\ddbar \log \Omega = 0, \qquad \int_M \Omega = \int_M \omega_0^n.
\end{equation}
Then solving (\ref{krf}) is equivalent to solving the parabolic complex Monge-Amp\`ere equation
\begin{equation} \label{pcma0}
\ddt{ \varphi} = \log \frac{ (\omega_0 + \ddbar \varphi)^n}{\Omega}, \qquad \omega_0+ \ddbar \varphi>0 , \qquad \varphi|_{t=0} =0.
\end{equation}

We first observe:

\pagebreak[3]
\begin{lemma} \label{c0ddt} We have
\begin{enumerate}
\item[(i)]  There exists a uniform constant $C$ such that for $t \in [0,\infty)$
\begin{equation}
\| \dot{\varphi}(t) \|_{C^0(M)} \le C.
\end{equation}
\item[(ii)]  There exists a uniform constant $C'$ such that on $M \times [0,\infty)$ the volume form of $\omega=\omega(t)$ satisfies
\begin{equation}
\frac{1}{C'} \omega_0^n \le \omega^n \le C' \omega_0^n.
\end{equation}
\end{enumerate}
\end{lemma}
\begin{proof}
Differentiating (\ref{pcma0}) with respect to $t$ we obtain
\begin{equation} \label{ddtdotphicM0}
\ddt{ \dot{\varphi}} = \Delta \dot{\varphi},
\end{equation}
and (i) follows immediately from the maximum principle.  Part (ii) follows from (i).  \qed
\end{proof}

We will obtain a bound on the oscillation of   $\varphi(t)$ using Yau's zeroth order estimate for the elliptic complex Monge-Amp\`ere equation.  Note that Yau's estimate holds for any K\"ahler manifold (not just those with $c_1(M)=0$):

\begin{theorem} \label{yau}
Let $(M, \omega_0)$ be a compact K\"ahler manifold and let $\varphi$ be a smooth function on $M$ satisfying
 \begin{equation}
(\omega_0 + \ddbar \varphi)^n = e^F \omega_0^n, \quad \omega_0+ \ddbar \varphi >0
\end{equation}
for some smooth function $F$.
Then there exists a uniform $C$ depending only on $\sup_M F$ and $\omega_0$ such that
\begin{equation} \label{osc}
\emph{osc}_M  \varphi:= \sup_M \varphi - \inf_M \varphi \le C.
\end{equation}
\end{theorem}
\begin{proof}  We will follow quite closely the exposition of Siu \cite{Si1}.  We assume without loss of generality that $\int_M \varphi\, \omega_0^n =0$.   We also assume $n>1$ (the case $n=1$ is easier, and we leave it as an exercise for the reader).

   Write $\omega = \omega_0+\ddbar \varphi$.  Then
\begin{align} \nonumber
C \int_M | \varphi| \omega_0^n & \ge \int_M \varphi (\omega_0^n - \omega^n) \\ \nonumber
& = - \int_M \varphi \ddbar \varphi \wedge \sum_{i=0}^{n-1} \omega_0^i \wedge \omega^{n-1-i} \\ \nonumber
& = \int_M \frac{\sqrt{-1}}{2\pi} \partial \varphi \wedge \ov{\partial} \varphi \wedge \sum_{i=0}^{n-1} \omega_0^i \wedge \omega^{n-1-i}  \\
& \ge \frac{1}{n} \int_M | \partial \varphi |_{\omega_0}^2 \omega_0^n.
\end{align}
By the Poincar\'e (Theorem \ref{poincare}) and Cauchy-Schwarz inequalities we have
\begin{equation}
\int_M |\varphi|^2 \omega_0^n \le C \int_M | \partial \varphi |^2_{\omega_0} \omega_0^n \le C' \int_M |\varphi| \omega_0^n \le C'' \left( \int_M | \varphi|^2 \omega_0^n \right)^{1/2},
\end{equation}
and hence $\| \varphi \|_{L^2(\omega_0)} \le C$.  We now repeat this argument with $\varphi$ replaced by $\varphi | \varphi|^{\alpha}$ for $\alpha\ge 0$.  Observe that the map of real numbers $x \mapsto x |x|^{\alpha}$ is differentiable with derivative $(\alpha +1)|x|^{\alpha}$.  Then
\begin{align} \nonumber
C \int_M | \varphi|^{\alpha+1} \omega_0^n & \ge \int_M \varphi | \varphi|^{\alpha} ( \omega_0^n - \omega^n) \\ \nonumber
& = - \int_M \varphi | \varphi |^{\alpha} \ddbar \varphi \wedge \sum_{i=0}^{n-1} \omega_0^i \wedge \omega^{n-1-i} \\ \nonumber
& = (\alpha+1)\int_M   | \varphi|^{\alpha} \sqrt{-1} \partial \varphi \wedge \ov{\partial} \varphi \wedge \sum_{i=0}^{n-1} \omega_0^i \wedge \omega^{n-1-i}\\
& = \frac{(\alpha +1)}{\left( \frac{\alpha}{2} + 1 \right)^2} \int_M \sqrt{-1} \partial \left( \varphi | \varphi|^{\alpha/2} \right) \wedge \ov{\partial} \left( \varphi | \varphi|^{\alpha/2} \right) \wedge \sum_{i=0}^{n-1} \omega_0^i \wedge \omega^{n-1-i}.
\end{align}
It then follows that for some uniform $C>0$,
\begin{equation}
\int_M \left| \partial \left( \varphi | \varphi |^{\alpha/2} \right) \right|^2_{\omega_0} \omega_0^n \le C (\alpha +1) \int_M | \varphi|^{\alpha+1} \omega_0^n.
\end{equation}
Now apply the Sobolev inequality (Theorem \ref{sobolev}) to $f = \varphi | \varphi|^{\alpha/2}$.  Then for $\beta = n/(n-1)$ we have
\begin{equation}
\left( \int_M | \varphi |^{(\alpha+2) \beta} \omega_0^n \right)^{1/\beta} \le C \left( (\alpha+1) \int_M | \varphi |^{\alpha+1}\omega_0^n + \int_M | \varphi |^{\alpha+2} \omega_0^n \right). \label{raise}
\end{equation}
By H\"older's inequality we have for a uniform constant $C$,
\begin{equation}
\int_M | \varphi|^{\alpha+1} \omega_0^n \le 1 + C \int_M | \varphi |^{\alpha+2} \omega_0^n.
\end{equation}
Now substituting $p=\alpha+2$ we have from (\ref{raise}),
\begin{equation}
\| \varphi \|^{p}_{L^{p\beta}(\omega_0)} \le C p \max \left( 1,  \| \varphi \|^p_{L^p(\omega_0)} \right).
\end{equation}
Raising to the power $1/p$ we have for all $p\ge 2$,
\begin{equation} \label{it}
\max ( 1 , \| \varphi \|_{L^{p\beta}(\omega_0)} ) \le C^{1/p} p^{1/p} \max (1, \| \varphi \|_{L^p(\omega_0)}).
\end{equation}
Fix an integer $k>0$.  Replace $p$ in (\ref{it})  by $p\beta^k$ and then $p\beta^{k-1}$ and so on, to obtain
\begin{align} \nonumber
\max ( 1, \| \varphi \|_{L^{p\beta^{k+1}}(\omega_0)}) & \le C^{\frac{1}{p\beta^k}} (p\beta^k)^{\frac{1}{p\beta^k}} \max (1, \| \varphi \|_{L^{p\beta^k}(\omega_0)}) \le \cdots   \\ \nonumber
& \le  C^{\frac{1}{p\beta^k} + \frac{1}{p\beta^{k-1}} + \cdots + \frac{1}{p}} (p\beta^k)^{\frac{1}{p\beta^k}} (p\beta^{k-1})^{\frac{1}{p\beta^{k-1}}} \cdots p^{\frac{1}{p}} \max (1, \| \varphi \|_{L^{p}(\omega_0)}) \\ \label{k}
& = C_k \max (1, \| \varphi \|_{L^{p}(\omega_0)})
\end{align}
for
\begin{equation}
C_k =
 C^{\frac{1}{p} \left( \frac{1}{\beta^k} + \frac{1}{\beta^{k-1}}+  \cdots + 1 \right)} p^{\frac{1}{p}\left( \frac{1}{\beta^k} +  \frac{1}{\beta^{k-1}}+\cdots +1 \right)} \beta^{\frac{1}{p} \left( \frac{k}{\beta^k} + \frac{k-1}{\beta^{k-1}} + \cdots + \frac{1}{\beta} \right)}.
\end{equation}
Since  the infinite sums $\sum \frac{1}{\beta^i}$ and $\sum \frac{i}{\beta^i}$ converge for $\beta = n/(n-1)>1$ we see that for any fixed $p$, the constants $C_k$ are uniformly bounded from above, independent of $k$.

Setting $p=2$ and letting $k \rightarrow \infty$ in (\ref{k}) we finally obtain
\begin{equation}
\max ( 1, \| \varphi \|_{C^0} ) \le C \max (1, \| \varphi \|_{L^2(\omega_0)}) \le C',
\end{equation}
and hence (\ref{osc}).  \qed
\end{proof}

Now the oscillation bound for $\varphi=\varphi(t)$ along the K\"ahler-Ricci flow (\ref{pcma0}) follows immediately:

\begin{lemma} \label{losc}
There exists a uniform constant $C$ such that for $t \in [0,\infty)$,
\begin{equation}
\emph{osc}_M \varphi \le C.
\end{equation}
\end{lemma}
\begin{proof}
From Lemma \ref{c0ddt} we have uniform bounds for $ \dot{\varphi}$.   Rewrite the parabolic complex Monge-Amp\`ere equation (\ref{pcma0}) as
\begin{equation}
(\omega_0 + \ddbar \varphi(t))^n = e^{F(t)} \omega_0^n \quad \textrm{with} \quad F(t) = \log \frac{\Omega}{\omega_0^n} + \dot{\varphi}(t)
\end{equation}
and apply Theorem \ref{yau}. \qed
\end{proof}

\pagebreak[3]
\subsection{The case of $c_1(M)=0$:  higher order estimates and convergence}

In this subsection we complete the proof of Theorem \ref{thmc0}.  The proof for the higher order estimates follows along similar lines as in the case for $c_1(M)<0$.  As above, let $\varphi(t)$ solve the parabolic complex Monge-Amp\`ere equation (\ref{pcma0}) on $M$ with $c_1(M)=0$ and write $\omega= \omega_0 + \ddbar \varphi$.

\begin{lemma} \label{lemmayauC2}  There exists a uniform constant $C$ such that on $M \times [0,\infty)$, $\omega=\omega(t)$ satisfies
\begin{equation}
\frac{1}{C} \omega_0 \le \omega \le C \omega_0.
\end{equation}
\end{lemma}
\begin{proof}  By Lemma \ref{c0ddt} and the argument of Corollary \ref{cmetricbound}, it suffices to obtain a uniform upper bound for $\tr{\omega_0}{\omega}$.
As in the case of Lemma \ref{lemmametricbdn}, define $Q = \log \tr{\omega_0}{\omega} - A\varphi$ for $A$ a constant to be determined later.  Compute using Proposition \ref{propChat},
\begin{equation} \label{yauc2eqn}
\left( \ddt{} - \Delta \right) Q \le C_0 \tr{\omega}{\omega_0} - A \dot{\varphi} + An - A \tr{\omega}{\omega_0},
\end{equation}
for $C_0$ depending only on $g_0$.  Choosing $A= C_0 +1$ we have, since $\dot{\varphi}$ is uniformly bounded,
\begin{equation} \label{yauc2eqn2}
\left( \ddt{} - \Delta \right) Q \le - \tr{\omega}{\omega_0} +C.
\end{equation}
We claim that for any $(x,t) \in M \times [0,\infty)$,
\begin{equation} \label{C2v2}
\left( \tr{\omega_0}{\omega}\right) (x, t)  \le C e^{A (\varphi(x,t)- \inf_{M \times [0,t]} \varphi)}.
\end{equation}
To see this, suppose that $Q$ achieves a maximum on $M \times [0,t]$ at the point $(x_0, t_0)$.  We assume without loss of generality that $t_0>0$.
  Applying the maximum principle to (\ref{yauc2eqn2}) we see that $(\tr{\omega}{\omega_0})(x_0,t_0) \le C$ and, by
   the argument of Lemma \ref{lemmametricbdn}, $(\tr{\omega_0}{\omega})(x_0, t_0) \le C'$.  Then for any $x \in M$,
\begin{equation}
\left( \log \tr{\omega_0}{\omega}\right)(x,t) - A \varphi(x,t) = Q(x, t) \le Q(x_0, t_0) \le \log C' - A \varphi(x_0, t_0).
\end{equation}
Exponentiating gives (\ref{C2v2}).

Define
\begin{equation}
\tilde{\varphi} := \varphi - \frac{1}{V} \int_M \varphi\,  \Omega, \quad \textrm{where } \quad V:= \int_M \Omega = \int_M \omega^n.
\end{equation}
From Lemma \ref{losc}, $ \| \tilde{\varphi} \|_{C^0(M)} \le C$.   The estimate (\ref{C2v2}) can be rewritten as:
\begin{align} \nonumber
\left( \tr{\omega_0}{\omega}\right) (x, t) &  \le C e^{A \left(\tilde{\varphi}(x,t) + \frac{1}{V} \int_M \varphi (t)\, \Omega- \inf_{M \times [0,t]} \tilde{\varphi} - \inf_{[0,t]} \frac{1}{V}\int_M \varphi\, \Omega \right)} \\ \label{C2v3}
& \le C e^{C'+\frac{A}{V} \left( \int_M \varphi (t) \Omega - \inf_{[0,t]} \int_M \varphi \, \Omega \right) }.
\end{align}
Using Jensen's inequality,
\begin{align}
\oddt \left( \frac{1}{V} \int_M \varphi \, \Omega \right) & = \frac{1}{V} \int_M \dot{\varphi}\, \Omega
 = \frac{1}{V} \int_M \log \left( \frac{\omega^n}{\Omega} \right) \, \Omega
 \le \log \left( \frac{1}{V} \int_M \omega^n \right) =0,
\end{align}
and hence $\inf_{[0,t]} \int_M \varphi \, \Omega = \int_M \varphi(t) \Omega$.  The required upper bound of $\tr{\omega_0}{\omega}$ follows then from (\ref{C2v3}).  \qed
\end{proof}

It follows from Corollary \ref{choe} that we have uniform $C^{\infty}$ estimates on $g(t)$ and $\varphi(t)$.   It remains to prove the $C^{\infty}$ convergence part of Theorem \ref{thmc0}.
We follow the method of Phong-Sturm \cite{PS} (see also \cite{MS, PSSW1}) and use a functional known as the Mabuchi energy \cite{Mab}.   It is noted in \cite{Cao4, DT} that
 the monotonicity of the  Mabuchi energy along the  K\"ahler-Ricci flow was established in unpublished work of H.-D. Cao in 1991.

We fix a metric $\omega_0$ as above.  The Mabuchi energy is a functional $\Mab_{\omega_0}$ on the space
\begin{equation}
\PSH = \{ \varphi \in C^{\infty}(M) \ | \ \omega_0 + \ddbar \varphi >0 \}
\end{equation}
with the property that  if $\varphi_t$ is any smooth path  in $\PSH$ then
\begin{equation} \label{Mab}
\oddt \Mab_{\omega_0} (\varphi_t) = - \int_M \dot{\varphi}_t R_{\varphi_t}  \, \omega_{\varphi_t}^n,
\end{equation}
where $\omega_{\varphi_t} = \omega_0 + \ddbar \varphi_t$, and $R_{\varphi_t}$ is the scalar curvature of $\omega_{\varphi_t}$.  Observe that if $\varphi_{\infty}$ is a critical point of $\Mab_{\omega_0}$ then $\omega_{\infty} = \omega_0 + \ddbar \varphi_{\infty}$  has zero scalar curvature and hence is Ricci flat  (for that last statement:  since $c_1(M)=0$, then $\Ric (\omega_{\infty}) = \ddbar h_{\infty}$ for some function $h_{\infty}$ and taking the trace  gives $\Delta_{\omega_\infty} h_{\infty}=0$ which implies $h_{\infty}$ is constant and $\Ric(\omega_{\infty})=0$).

  Typically, the Mabuchi energy is defined in terms of its derivative using the formula (\ref{Mab}) but instead we will use the explicit formula as derived in \cite{Tbook}.
Define
\begin{equation} \label{Mab2}
\Mab_{\omega_0} (\varphi) = \int_M \log \frac{\omega_{\varphi}^n}{\omega_0^n}\, \omega_{\varphi}^n - \int_M h_0 (\omega_{\varphi}^n - \omega_0^n),
\end{equation}
where $\omega_{\varphi} = \omega_0 + \ddbar \varphi$ and  $h_0$ is the \emph{Ricci potential} for $\omega_0$ given by
\begin{equation}
\Ric(\omega_0) = \ddbar h_0, \qquad \int_M e^{h_0} \omega_0^n = \int_M \omega_0^n.
\end{equation}
Observe that $\Mab_{\omega_0}$ depends only on the metric $\omega_{\varphi}$  and so can be regarded as a functional on the space of K\"ahler metrics cohomologous to $\omega_0$.   We now need to check that $\Mab_{\omega_0}$ defined by (\ref{Mab2}) satisfies (\ref{Mab}).  Let $\varphi_t$ be any smooth path in $\PSH$.  Using integration by parts, we compute
\begin{align} \nonumber
\oddt \Mab_{\omega_0} (\varphi_t) & = \int_M \Delta \dot{\varphi}_t \, \omega_{\varphi_t}^n + \int_M \log \frac{\omega_{\varphi_t}^n}{\omega_0^n} \, \Delta \dot{\varphi}_t \, \omega_{\varphi_t}^n - \int_M h_0 \Delta \dot{\varphi}_t \, \omega_{\varphi_t}^n \\ \nonumber
& = \int_M \dot{\varphi}_t (- R_{\varphi_t} + \tr{\omega}{\Ric(\omega_0)}   ) \omega_{\varphi_t}^n - \int_M \dot{\varphi}_t \Delta h_0 \, \omega_{\varphi_t}^n \\
& = - \int_M \dot{\varphi}_t R_{\varphi_t} \omega_{\varphi_t}^n.
\end{align}

The key fact we need is as follows:

\begin{lemma} \label{lmd}
Let $\varphi= \varphi(t)$ solve the K\"ahler-Ricci flow (\ref{pcma0}).  Then
\begin{equation} \label{mabe1}
\oddt \emph{Mab}_{\omega_0}(\varphi) = -\int_M | \partial \dot{\varphi}|^2_{\omega} \omega^n.
\end{equation}
In particular, the Mabuchi energy is decreasing along the K\"ahler-Ricci flow.  Moreover, there exists a uniform constant $C$ such that
\begin{equation} \label{mode}
\oddt \int_M | \partial \dot{\varphi}|^2_{\omega} \omega^n \le C \int_M | \partial \dot{\varphi}|^2_{\omega} \omega^n.
\end{equation}
\end{lemma}
\begin{proof}  Observe that from the K\"ahler-Ricci flow equation we have $\ddbar \dot{\varphi} = - \Ric(\omega)$ and taking the trace of this gives $\Delta \dot{\varphi} = -R$.  Then
\begin{equation}
\oddt \Mab_{\omega_0}(\varphi) = - \int_M \dot{\varphi} R  \, \omega^n = \int_M \dot{\varphi} \Delta \dot{\varphi} \, \omega^n = -\int_M | \partial \dot{\varphi}|^2_{\omega} \omega^n,
\end{equation}
giving (\ref{mabe1}).  For (\ref{mode}), compute
\begin{align} \nonumber
\oddt \int_M | \partial \dot{\varphi}|^2_{\omega} \omega^n & = \int_M ( \ddt{} g^{\ov{j}{i}})  \partial_i \dot{\varphi} \partial_{\ov{j}} \dot{\varphi}\, \omega^n + 2 \textrm{Re} \left( \int_M g^{\ov{j}i} \partial_i (\Delta\dot{\varphi})  \partial_{\ov{j}} \dot{\varphi}\, \omega^n \right) + \int_M | \partial \dot{\varphi}|^2 \Delta \dot{\varphi}\,  \omega^n \\ \nonumber
& = \int_M R^{\ov{j}i}  \partial_i \dot{\varphi} \partial_{\ov{j}} \dot{\varphi}\, \omega^n - 2 \int_M (\Delta \dot{\varphi})^2 \omega^n - \int_M | \partial \dot{\varphi}|^2 R\, \omega^n \\
& \le  C \int_M | \partial \dot{\varphi}|^2_{\omega} \omega^n,
\end{align}
using (\ref{ddtdotphicM0}), an integration by parts and the fact that, since we have $C^{\infty}$ estimates for $\omega$, we have uniform bounds of the Ricci and scalar curvatures of $\omega$.
  \qed
\end{proof}

It is now straightforward to complete the proof of the convergence of the K\"ahler-Ricci flow.   Since we have uniform estimates for $\omega(t)$ along the flow, we see from the formula (\ref{Mab2}) that the Mabuchi energy is uniformly bounded.  From (\ref{mabe1})  there is a sequence of times $t_i \in [i, i+1]$ for which
\begin{equation} \label{goestozero}
\left( \int_M \left| \partial \log \frac{\omega^n}{\Omega} \right|^2_{\omega} \, \omega^n \right) (t_i) =
\left( \int_M | \partial \dot{\varphi}|^2_{\omega} \, \omega^n \right) (t_i)  \rightarrow 0, \quad \textrm{as } i \rightarrow \infty.
\end{equation}
By the differential inequality (\ref{mode}),
\begin{equation} \label{goestozero2}
\left( \int_M \left| \partial \log \frac{\omega^n}{\Omega} \right|^2_{\omega} \, \omega^n \right) (t)   \rightarrow 0, \quad \textrm{as } t \rightarrow \infty.
\end{equation}
But since we have $C^{\infty}$ estimates for $\varphi(t)$ we can apply the Arzela-Ascoli Theorem to obtain a sequence of times $t_{j}$ such that $\varphi(t_{j})$ converges in $C^{\infty}$ to $\varphi_{\infty}$, say.  Writing $\omega_{\infty} = \omega_0 + \ddbar \varphi_{\infty}>0$, we have from (\ref{goestozero2}),
\begin{equation}
\left( \int_M \left| \partial \log \frac{\omega_{\infty}^n}{\Omega} \right|^2_{\omega_{\infty}} \, \omega_{\infty}^n \right)  = 0,
\end{equation}
and hence
\begin{equation} \label{infty}
\log \frac{\omega_{\infty}^n}{\Omega} = C,
\end{equation}
for some constant $C$.  Taking $\ddbar$ of (\ref{infty}) gives $\Ric(\omega_{\infty})=0$.  Hence for a sequence of times $t_{j} \rightarrow \infty$ the K\"ahler-Ricci flow converges to  $\omega_{\infty}$,  the unique K\"ahler-Einstein metric in the cohomology class $[\omega_0]$.

To see that the convergence of the metrics $\omega(t)$ is in $C^{\infty}$ without passing to a subsequence, we argue as follows.  If not, then by  the same argument as in the proof of Theorem \ref{thmc1n} we can find a sequence of times $t_k \rightarrow \infty$ such that $\omega(t_k)$ converges in $C^{\infty}$ to $\omega_{\infty}' \neq \omega_{\infty}$.  But by (\ref{goestozero2}), $\omega_{\infty}'$ is K\"ahler-Einstein, contradicting the uniqueness of K\"ahler-Einstein metrics in $[\omega_0]$.  This completes the proof of Theorem \ref{thmc0}.

\begin{remark} \emph{It was pointed out to the authors by Zhenlei Zhang that one can equivalently consider the functional $\int_M h\, \omega^n$, where $h$ is the Ricci potential of the evolving metric.}
\end{remark}

\pagebreak

\section{The case when $K_M$ is big and nef} \label{secttsuji}

In the previous section we considered the K\"ahler-Ricci flow on manifolds with $c_1(M)<0$, which is equivalent to the condition that the canonical line bundle $K_M$ is ample.  In this section we consider the case where the line bundle $K_M$ is not necessarily ample, but nef and big.  Such a manifold is known as a smooth minimal model of general type.

\subsection{Smooth minimal models of general type}

As in the case of $c_1(M)<0$ we consider the normalized K\"ahler-Ricci flow
\begin{equation} \label{nkrf2}
\ddt{} \omega  = - \Ric(\omega) - \omega, \qquad \omega|_{t=0} = \omega_0,
\end{equation}
but we impose no restrictions on the K\"ahler class of $\omega_0$.  We will prove:

\begin{theorem} \label{tsuji}
Let $M$ be a smooth minimal model of general type (that is, $K_M$ is  nef and big).  Then
\begin{enumerate}
\item[(i)] The solution $\omega=\omega(t)$ of the normalized K\"ahler-Ricci flow (\ref{nkrf2}) starting at any K\"ahler metric $\omega_0$ on $M$ exists for all time.
\item[(ii)] There exists a codimension 1  analytic subvariety $S$ of $M$ such that $\omega(t)$ converges in $C^{\infty}_{\emph{loc}}(M \setminus S)$ to a K\"ahler metric $\omega_{\emph{KE}}$ defined on $M\setminus S$ which satisfies the K\"ahler-Einstein equation
\begin{equation}
\emph{Ric}(\omega_{\emph{KE}}) = - \omega_{\emph{KE}}, \quad \textrm{on} \  M \setminus S.
\end{equation}
\end{enumerate}
\end{theorem}

We will see later in Section \ref{convunique} that $\omega_{\textrm{KE}}$ is unique under some suitable conditions.
Note that if $K_M$ is not ample, then $\omega_{\textrm{KE}}$ cannot extend to be a smooth K\"ahler metric on $M$, and we call $\omega_{\textrm{KE}}$ a \emph{singular K\"ahler-Einstein metric}.
The first proof of Theorem \ref{tsuji} appeared in the work of Tsuji \cite{Ts1}.  Later, Tian-Zhang \cite{TZha} extended this result (see Section \ref{sectionpluri} below) and clarified some parts of Tsuji's proof.   Our exposition will for the most part follow \cite{TZha}.

From part (i) of Theorem \ref{algebraic} we see that under the assumptions of Theorem \ref{tsuji}, the cohomology class $[\omega_0] - t c_1(M)$ is K\"ahler for all $t\ge 0$ and hence by Theorem \ref{longtime}, the (unnormalized) K\"ahler-Ricci flow has a smooth solution $\omega(t)$ for all time $t$.  Rescaling as in
Section \ref{sectnkrf} we obtain a solution of the normalized K\"ahler-Ricci flow (\ref{nkrf2}) for all time.  This establishes part (i) of Theorem \ref{tsuji}.  Observe that in fact we only need $K_M$ to be nef to obtain a solution to the K\"ahler-Ricci flow for all time.

It is straightforward to calculate the  K\"ahler class of the evolving metric along the flow.  Indeed, $[\omega(t)]$ evolves according to the ordinary differential equation
\begin{equation} \label{ode1}
\frac{d}{dt} [\omega(t)] = - c_1(M) - [\omega], \quad [\omega(0)] = [\omega_0],
\end{equation}
and this has a solution
\begin{equation} \label{ode2}
[\omega(t)]  = - (1- e^{-t}) c_1(M) + e^{-t} [\omega_0].
\end{equation}
This shows that, in particular, $[\omega(t)] \rightarrow -c_1(M)$ as $t \rightarrow \infty$.

We now rewrite (\ref{nkrf2}) as a parabolic complex Monge-Amp\`ere equation.  First, from the Base Point Free Theorem (part (ii) of Theorem \ref{algebraic}), $K_M$ is semi-ample.  Hence there exists a smooth closed nonnegative $(1,1)$-form $\hat{\omega}_{\infty}$ on $M$ with $[\hat{\omega}_{\infty}] = -c_1(M)$.  Indeed, we may take $\hat{\omega}_{\infty} = \frac{1}{m}\Phi^* \omega_{\textrm{FS}}$ where $\Phi: M \rightarrow \mathbb{P}^N$ is a holomorphic map defined by holomorphic sections of $K_M^m$ for $m$ large and $\omega_{\textrm{FS}}$ is the Fubini-Study metric (see Section \ref{sectnotion}).

  Define reference metrics in $[\omega(t)]$ by
\begin{equation}
\hat{\omega}_t= e^{-t} \omega_0 + (1-e^{-t}) \hat{\omega}_{\infty}, \qquad \textrm{for } t\in [0, \infty).
\end{equation}
Let $\Omega$ be the smooth volume form on $M$ satisfying
\begin{equation} \label{vfcond1}
\ddbar \log \Omega = \hat{\omega}_{\infty} \in - c_1(M), \qquad \int_M \Omega = \int_M \omega_0^n.
\end{equation}
We then consider the parabolic complex Monge-Amp\`ere equation
\begin{equation} \label{npcma2}
\ddt{ \varphi} = \log \frac{ (\hat{\omega}_t + \ddbar \varphi)^n}{\Omega} - \varphi, \qquad \hat{\omega}_t+ \ddbar \varphi>0 , \qquad \varphi|_{t=0} =0,
\end{equation}
which is equivalent to (\ref{nkrf2}).    Hence a solution to (\ref{npcma2}) exists for all time.

\pagebreak[3]
\subsection{Estimates}

In this section we prove the estimates needed for  the second part of Theorem \ref{tsuji}.  Assume that $\varphi= \varphi(t)$ solves (\ref{npcma2}).  We have:

\pagebreak[3]
\begin{lemma} \label{lemmat1}
There exists a uniform constants $C$ and $t'>0$ such that on $M$,
\begin{enumerate}
\item[(i)] $\displaystyle{\varphi(t) \le C}$ for $t \ge 0$.
\item[(ii)] $\displaystyle{\dot{\varphi}(t) \le C t e^{-t}}$ for $t \ge t'$.  In particular, $\dot{\varphi}(t) \le C$ for $t \ge 0$.
\item[(iii)] $\displaystyle{\omega^n(t) \le C \Omega}$ for $t \ge 0$.
\end{enumerate}

\end{lemma}
\begin{proof}
Part (i) follows immediately from the maximum principle.  Indeed if $\varphi$ achieves a maximum at a point $(x_0, t_0)$ with $t_0>0$ then, directly from (\ref{npcma2}),
\begin{equation}
0 \le \ddt{\varphi} \le \log \frac{\hat{\omega}_t^n}{\Omega} - \varphi \quad \textrm{at } (x_0,t_0),
\end{equation}
and hence $\varphi \le  \log ( \hat{\omega}_t^n/\Omega) \le C$.

Part (ii) is a result of \cite{TZha}.  Compute
\begin{align}
\left(\ddt{} - \Delta \right) \varphi  & = \dot{\varphi} - n + \tr{\omega}{\hat{\omega}_t}\\ \label{evolvephidot}
\left(\ddt{} - \Delta \right)\dot{\varphi} & = -e^{-t} \tr{\omega}{(\omega_0 - \hat{\omega}_{\infty})} - \dot{\varphi},
\end{align}
using the fact that $\ddt{} \hat{\omega}_t = -e^{-t} (\omega_0 - \hat{\omega}_{\infty})$.  Hence
\begin{align} \label{etphi}
\left(\ddt{} - \Delta \right)\left( e^t \dot{\varphi} \right) &= - \tr{\omega}{(\omega_0 - \hat{\omega}_{\infty})} \\ \label{dotphiphi}
\left(\ddt{} - \Delta \right) \left( \dot{\varphi} + \varphi +nt \right) & =  \tr{\omega}{\hat{\omega}_{\infty}}.
\end{align}
Subtracting (\ref{dotphiphi}) from (\ref{etphi}) gives
\begin{equation} \label{crucial}
\left(\ddt{} - \Delta \right) \left( (e^t-1) \dot{\varphi} - \varphi -nt \right) = - \tr{\omega}{\omega_0}  <0,
\end{equation}
which implies that the maximum of $(e^t-1) \dot{\varphi} - \varphi -nt$ is decreasing in time, giving
\begin{equation}
(e^t-1) \dot{\varphi} - \varphi -nt \le 0.
\end{equation}
This establishes (ii).  Part (iii) follows from Corollary \ref{volform} (or using (i) and (ii) and the fact that $\omega^n/\Omega = e^{\dot{\varphi} +\varphi}$). \qed

\end{proof}

We now prove lower bounds for $\varphi$ and $\dot{\varphi}$ away from a subvariety.   To do this we need to use \emph{Tsuji's trick} of applying Kodaira's Lemma (part (iii) of Theorem \ref{algebraic}).

Since $K_M$ is big and nef, there exists an effective divisor $E$ on $M$ with $K_M - \delta [E]>0$  for some sufficiently small $\delta>0$.    Since $\hat{\omega}_{\infty}$ lies in the cohomology class $c_1(K_M)$ it follows that for any Hermitian metric $h$ of $[E]$ the cohomology class of $\hat{\omega}_{\infty} - \delta R_h$ is K\"ahler.  Then by the $\partial \ov{\partial}$-Lemma we may pick a Hermitian metric $h$ on $[E]$ such that
\begin{equation}
\hat{\omega}_{\infty} - \delta R_h \ge c \omega_0,
\end{equation}
for some constant $c>0$.  Moreover, if we pick any $\ve \in (0, \delta]$ we have
\begin{equation}
\hat{\omega}_{\infty} - \ve R_h \ge c_{\ve} \omega_0,
\end{equation}
for $c_{\ve} = c\ve/\delta >0$.  Indeed, since $\hat{\omega}_{\infty}$ is semi-positive,
\begin{equation}
\hat{\omega}_{\infty} - \ve R_h = \frac{\ve}{\delta} (\hat{\omega}_{\infty} - \delta R_h) + \left( 1 - \frac{\ve}{\delta} \right) \hat{\omega}_{\infty} \ge \frac{\ve}{\delta} (\hat{\omega}_{\infty} - \delta R_h) \ge \frac{c\ve}{\delta}  \omega_0.
\end{equation}

Now fix a holomorphic section $\sigma$ of $[E]$ which vanishes to order 1 along the divisor $E$.  It follows that
\begin{equation} \label{comega0}
\hat{\omega}_{\infty} + \ve \ddbar \log |\sigma|^2_h \ge c_{\ve} \omega_0, \qquad \textrm{on } M \setminus E,
\end{equation}
since $\partial \ov{\partial} \log |\sigma|^2_h = \partial \ov{\partial} \log h$ away from $E$.
Note that here (and henceforth) we are writing $E$ for the support of the divisor $E$.

We can then prove:

\begin{lemma} \label{phiE} With the notation above, for every $\ve \in (0, \delta]$ there exists a constant $C_{\ve}>0$ such that on $(M \setminus E) \times [0, \infty)$,
\begin{enumerate}
\item[(i)] $\displaystyle{\varphi \ge \ve \log |\sigma|^2_h -C_{\ve}}$.
\item[(ii)] $\displaystyle{\dot{\varphi} \ge \ve \log |\sigma|^2_h -C_{\ve}}$.
\item[(iii)] $\displaystyle{ \omega^n \ge \frac{1}{C_{\ve}} | \sigma|^{2\ve}_h \Omega}$.
\end{enumerate}
\end{lemma}
\begin{proof}
It suffices to prove the estimate
\begin{equation} \label{suffice}
\varphi + \dot{\varphi} \ge \ve \log |\sigma|^2_h -C_{\ve}, \quad \textrm{on } M \setminus E,
\end{equation}
where we write $C_{\ve}$ for a constant that depends only on $\ve$ and the fixed data.
Indeed this inequality immediately implies (iii).  The estimates (i) and (ii) follow from (\ref{suffice}) together with the upper bounds of $\dot{\varphi}$ and $\varphi$ given by Lemma \ref{lemmat1}.

To establish (\ref{suffice}), we will bound from below the quantity $Q$ defined by
\begin{equation} \label{Q}
Q = \dot{\varphi} + \varphi - \ve \log |\sigma|^2_h = \log \frac{\omega^n}{|\sigma|^{2\ve}_h \Omega}, \quad \textrm{on } M \setminus E.
\end{equation}
Observe that for any fixed time $t$, $Q(x,t) \rightarrow \infty$ as $x$ approaches  $E$.  Hence for each time $t$, $Q$ attains a minimum (in space) in the interior of the set $M \setminus E$.  Now
from (\ref{dotphiphi}) we have
\begin{equation}
\left(\ddt{} - \Delta \right) \left( \dot{\varphi} + \varphi \right)  =  \tr{\omega}{\hat{\omega}_{\infty}} -n.
\end{equation}
Using this we compute on $M \setminus E$,
\begin{align}
\left(\ddt{} - \Delta \right) Q & = \tr{\omega}{\hat{\omega}_{\infty}} - n + \ve \tr{\omega}{\left( \ddbar \log |\sigma |^2_h \right) }\\
& = \tr{\omega}{\left( \hat{\omega}_{\infty} + \ve\ddbar \log |\sigma |^2_h \right) } -n \\
& \ge c_{\ve} \tr{\omega}{\omega_0} -n,
\end{align}
where for the last line we used (\ref{comega0}).

Then if  $Q$ achieves a minimum at $(x_0, t_0)$ with $x_0$ in $M \setminus E$ and $t_0>0$ then at $(x_0,t_0)$ we have
\begin{equation}
\tr{\omega}{\omega_0} \le \frac{n}{c_{\ve}}.
\end{equation}
By the geometric-arithmetic means inequality, at $(x_0, t_0)$,
\begin{equation}
\left( \frac{\omega_0^n}{\omega^n} \right)^{1/n} \le \frac{1}{n} \tr{\omega}{\omega_0} \le \frac{1}{c_{\ve}},
\end{equation}
which gives a uniform lower bound for the volume form $\omega^n(x_0,t_0)$.  Hence
\begin{equation}
Q(x_0,t_0) = \log \frac{\omega^n}{| \sigma |^2_h \Omega} (x_0, t_0) \ge -C_{\ve},
\end{equation}
and since $Q$ is bounded below at time $t=0$ we obtain the desired lower bound for $Q$.  \qed
\end{proof}

Next we prove estimates for $g(t)$ away from a divisor.  First, \emph{we from now on fix an $\ve$ in $(0, \delta]$ sufficiently small so that $\omega_0 + \ve \ddbar \log h$ is K\"ahler}. We will need the following lemma.

\begin{lemma} \label{lemmanewref}  For the $\ve>0$ fixed as above,
the metrics $\hat{\omega}_{t, \ve}$ defined by
\begin{equation} \label{newref}
\hat{\omega}_{t, \ve} := \hat{\omega}_t + \ve \ddbar \log h = \hat{\omega}_{\infty} + \ve \ddbar \log h + e^{-t} (\omega_0 -\hat{\omega}_{\infty}).
\end{equation}
give a smooth family of K\"ahler metrics for $t \in [0,\infty)$.  Moreover there exists a constant $C>0$ such that for all $t$,
\begin{equation}
\frac{1}{C} \omega_0 \le \hat{\omega}_{t, \ve} \le C \omega_0.
\end{equation}
\end{lemma}
\begin{proof}
From (\ref{comega0}) we see that $\hat{\omega}_{\infty} + \ve \ddbar \log h$ is K\"ahler.
Hence we may choose $T_0>0$ sufficiently large so that, for $C>0$ large enough,
\begin{equation}
\frac{1}{C} \omega_0 \le \hat{\omega}_{\infty} + \ve \ddbar \log h + e^{-t} (\omega_0 -\hat{\omega}_{\infty}) \le C \omega_0,
\end{equation}
for all $t > T_0$.  It remains to check that $\hat{\omega}_{t, \ve}$ is K\"ahler for $t \in [0,T_0]$.  But for $t \in [0,T_0]$,
\begin{align} \nonumber
\hat{\omega}_{t,\ve}
& = (1-e^{-t}) \left( \hat{\omega}_{\infty} + \ve \ddbar \log h \right) + e^{-t} \left( \omega_0 + \ve \ddbar \log h \right) \\
& > e^{-T_0} \left( \omega_0 + \ve \ddbar \log h \right) >0,
\end{align}
by definition of $\ve$. \qed
\end{proof}

We can now prove bounds for the evolving metric:

\begin{lemma} \label{trick2}
There exist uniform constants $C$ and $\alpha$ such that on $(M \setminus E) \times [0, \infty)$,
\begin{equation} \label{trt}
\emph{tr}_{\omega_0}{\, \omega} \le \frac{C}{|\sigma|_h^{2\alpha}}.
\end{equation}
Hence there exist uniform constants $C'>0$ and $\alpha'$ such that on $(M \setminus E) \times [0, \infty)$,
\begin{equation} \label{ulbm}
\frac{|\sigma|^{2\alpha'}_h}{C'} \omega_0 \le \omega(t) \le \frac{C'}{|\sigma|^{2\alpha'}_h} \omega_0.
\end{equation}
\end{lemma}
\begin{proof}
Define a quantity $Q$ on $M \setminus E$ by
\begin{equation}
Q = \log \tr{\omega_0}{\omega} - A \left( \varphi - \ve \log |\sigma|^2_{h} \right),
\end{equation}
for $A$ a sufficiently large constant to be determined later.  For any fixed time $t$, $Q(x,t) \rightarrow -\infty$ as $x$ approaches $E$. Then compute using Proposition \ref{propChat},
\begin{align}
\left( \ddt{} - \Delta \right) Q & \le C_0 \tr{\omega}{\omega_0} - A \dot{\varphi} + A \Delta \left( \varphi - \ve \log |\sigma|^2_{h}\right).
\end{align}

  Now at any point of $M \setminus E$,
\begin{equation}
\Delta \left( \varphi - \ve \log |\sigma|^2_{h}\right) = \tr{\omega}{\left( \omega - \hat{\omega}_t - \ve \ddbar \log |\sigma|^2_h \right)} = n - \tr{\omega}{\hat{\omega}_{t, \ve}}.
\end{equation}
Applying Lemma \ref{lemmanewref}, we may choose $A$ sufficiently large so that
$A \hat{\omega}_{t, \ve} \ge (C_0 + 1) \omega_0$ and hence
\begin{align} \nonumber
\left( \ddt{} - \Delta \right) Q & \le - \tr{\omega}{\omega_0} - A\left( \log \frac{\omega^n}{\Omega} - \varphi \right) +An   \\
& \le - \tr{\omega}{\omega_0} - A \log \frac{\omega^n}{\omega_0^n} +C,
\end{align}
where we have used the upper bound on $\varphi$ from Lemma \ref{lemmat1}.

Working in a compact time interval $[0,t']$ say, suppose that $Q$ achieves a maximum at $(x_0, t_0)$ with $x_0$ in $M$ and $t_0>0$.  Then at
$(x_0, t_0)$  we have
\begin{equation}
 \tr{\omega}{\omega_0} + A \log \frac{\omega^n}{\omega_0^n} \le C.
\end{equation}
By the same argument as in the proof of Lemma \ref{lemmatr} we see that $(\tr{\omega_0} \omega) (x_0, t_0) \le C$.

Then for any $(x, t) \in (M \setminus E) \times [0, t']$ we have
\begin{align} \nonumber
Q(x,t) & = (\log \tr{\omega_0}{\omega})(x,t) - A \left( \varphi - \ve \log |\sigma|_h^2 \right)(x,t) \\ \nonumber
 & \le  Q(x_0,t_0) \\
& \le  \log C - A\left(\varphi - \ve \log |\sigma|^2_h\right) (x_0, t_0) \le C',
\end{align}
where for the last line we used part (i) of Lemma \ref{phiE}.  Since $t'$ is arbitary, we have on $(M \setminus E) \times [0, \infty)$,
\begin{equation}
\log \tr{\omega_0}{\omega} \le C+ A \left( \varphi - \ve \log |\sigma|_h^2 \right).
\end{equation}
Since $\varphi$ is bounded from above we obtain (\ref{trt}) after exponentiating.

For (\ref{ulbm}), combine (\ref{trt}) with part (iii) of Lemma \ref{phiE}.
\qed

\end{proof}

We now wish to obtain higher order estimates on compact subsets of $M \setminus E$:

\begin{lemma} \label{hoesigma}
For $m=0,1,2, \ldots$, there exist uniform constants $C_m$ and $\alpha_m$ such that on $(M \setminus E) \times [0,\infty)$,
\begin{equation}
S \le \frac{C_0}{|\sigma|_h^{2\alpha_0}}, \quad | \nabla_{\mathbb{R}}^m \emph{Rm}(g) |\le \frac{C_m}{|\sigma|_h^{2\alpha_m}},
\end{equation}
where we are using the notation of Sections \ref{sect3rd} and \ref{secthoe}.
\end{lemma}
\begin{proof}
We prove only the bound on $S$ and leave the bounds on curvature and its derivatives as an exercise to the reader.  We will follow quite closely an argument given in \cite{SW2}.
From Proposition \ref{propPSS1} and (\ref{ulbm}),
\begin{align}
\left( \ddt{} - \Delta \right) S & = - | \ov{\nabla} \Psi |^2 - | \nabla \Psi |^2 + | \Psi |^2
 - 2 \textrm{Re} \left( \, g^{\ov{j} i} g^{\ov{q} p} g_{k \ov{\ell}} \nabla^{\ov{b}} \hat{R}_{i \ov{b} p}^{\ \ \ k} \ov{\Psi_{j q}^{\ell }} \right) \\
 & \le - | \ov{\nabla} \Psi |^2 - | \nabla \Psi |^2 + S + C |\sigma|_h^{-K} \sqrt{S},
\end{align}
for a uniform constant $K$.  We have
\begin{equation}
| \partial S | \le \sqrt{S} (| \ov{\nabla} \Psi| + | \nabla \Psi |).
\end{equation}
Moreover,
\begin{equation} \label{boundsigma}
| \partial |\sigma|_h^{4K} | \le C | \sigma|_h^{3K} \quad \textrm{and} \quad | \Delta |\sigma|^{4K}_h | \le C | \sigma|_h^{3K},
\end{equation}
where we are increasing $K$ if necessary.  Then
\begin{align} \nonumber
\left( \ddt{} - \Delta \right) ( |\sigma|_h^{4K} S) & = | \sigma|_h^{4K} \left( \ddt{} - \Delta \right) S - 2 \textrm{Re} (g^{\ov{j} i} \partial_i |\sigma|^{4K}_h \partial_{\ov{j}}  S )  - (\Delta |\sigma|_h^{4K}) S \\ \nonumber
& \le - |\sigma|_h^{4K} ( | \ov{\nabla} \Psi |^2 + | \nabla \Psi |^2 ) + C |\sigma|_h^{3K} \sqrt{S} ( | \ov{\nabla} \Psi| + | \nabla \Psi|) \\ \nonumber \mbox{} & + C |\sigma|_h^{2K} S + C \\ \label{sigS}
& \le C(1+ |\sigma|_h^{2K} S).
\end{align}
But from Proposition \ref{propkeyeqn} and (\ref{ulbm})
\begin{align} \nonumber
\left( \ddt{} - \Delta \right) \tr{\omega_0}{\omega} & = - \tr{\omega_0}{\omega} - g^{\ov{\ell} k} {R(g_0)}_{k \ov{\ell}}^{\ \ \, \ov{j}i}  g_{i \ov{j}} - g_0^{\ov{j} i} g^{\ov{q}p} g^{\ov{\ell} k} \nabla^0_i g_{p \ov{\ell}} \nabla^0_{\ov{j}} g_{k \ov{q}} \\
& \le C | \sigma|_h^{-K} - \frac{1}{C} |\sigma|_h^K S - \frac{1}{2} g_0^{\ov{j} i} g^{\ov{q}p} g^{\ov{\ell} k} \nabla^0_i g_{p \ov{\ell}} \nabla^0_{\ov{j}} g_{k \ov{q}},
\end{align}
where  $\nabla^0$ denotes the covariant derivative with respect to $g_0$.  We may assume that $K$ is large enough so that $| (\Delta | \sigma|_h^K) \tr{\omega_0}{\omega}| \le C$.  Then
\begin{align} \nonumber
\left( \ddt{} - \Delta \right) ( | \sigma|_h^{K} \tr{\omega_0}{\omega}) & \le - \frac{1}{C} | \sigma|^{2K}_h S + C - 2 \textrm{Re} ( g^{\ov{j}i} \partial_i | \sigma|_h^K \partial_{\ov{j}} \tr{\omega_0}{\omega}) \\ \nonumber & \mbox{} - \frac{1}{2} |\sigma|_h^K  g_0^{\ov{j} i} g^{\ov{q}p} g^{\ov{\ell} k} \nabla^0_i g_{p \ov{\ell}} \nabla^0_{\ov{j}} g_{k \ov{q}} \\ \label{sigtr}
& \le - \frac{1}{C} | \sigma|^{2K}_h S +C,
\end{align}
where for the last line we have used:
\begin{align}
| 2 \textrm{Re} ( g^{\ov{j}i} \partial_i | \sigma|_h^K \partial_{\ov{j}} \tr{\omega_0}{\omega}) | & \le C + \frac{1}{C} | \partial |\sigma|^K_h |^2 | \partial \tr{\omega_0}{\omega}|^2 \\
& \le C + \frac{1}{2} |\sigma|^K_h g_0^{\ov{j} i} g^{\ov{q}p} g^{\ov{\ell} k} \nabla^0_i g_{p \ov{\ell}} \nabla^0_{\ov{j}} g_{k \ov{q}} ,
\end{align}
which follows from (\ref{claimcs}), increasing $K$ if necessary.

Now define $Q = | \sigma|_h^{4K}S + A |\sigma|_h^K \tr{\omega_0}{\omega}$ for a constant $A$.  Combining (\ref{sigS})  and (\ref{sigtr}) we see that for $A$ sufficiently large,
\begin{align}
\left( \ddt{} - \Delta \right) Q \le - |\sigma|^{2K}_h S + C,
\end{align}
and then $Q$ is bounded from above by the maximum principle.  The bound on $S$ then follows. \qed
\end{proof}

As a consequence:

\begin{lemma} \label{hoesigma2}
$\varphi=\varphi(t)$ and $\omega=\omega(t)$ are uniformly bounded in $C^{\infty}_{\emph{loc}}(M \setminus E)$.
\end{lemma}
\begin{proof}
Applying Theorem \ref{interior} gives the $C^{\infty}_{\textrm{loc}}(M \setminus E)$ bounds for $\omega$.  Since by Lemmas \ref{lemmat1} and \ref{phiE}, $\varphi$ is uniformly bounded (in $C^0$) on compact subsets of $M \setminus E$, the $C^{\infty}_{\textrm{loc}}(M \setminus E)$ bounds on $\varphi$ follow from those on $\omega$.  \qed
\end{proof}


\subsection{Convergence of the flow and uniqueness of the limit} \label{convunique}

We now complete the proof of Theorem \ref{tsuji}.  From part (ii) of Lemma \ref{lemmat1} we have $\dot{\varphi} \le C t e^{-t}$ for $t \ge t'$.  Hence for $t \ge t'$,
\begin{equation}
\ddt{} \left( \varphi + C e^{-t} (t+1) \right) \le 0.
\end{equation}
On the other hand, from Lemma \ref{phiE}, the quantity $\varphi + C e^{-t}(t+1)$ is uniformly bounded from below on compact subsets of $M \setminus E$.  Hence $\varphi(t)$ converges pointwise on $M \setminus E$ to a function $\varphi_{\infty}$.  Since we have $C^{\infty}_{\textrm{loc}}(M \setminus E)$ estimates for $\varphi(t)$ this implies, by a similar argument to that given in the proof of Theorem \ref{thmc1n}, that $\varphi$ converges to $\varphi_{\infty}$ in $C^{\infty}_{\textrm{loc}}(M \setminus E)$.  In particular $\varphi_{\infty}$ is smooth on $M \setminus E$.  Define $\omega_{\infty} = \hat{\omega}_{\infty} + \ddbar \varphi_{\infty}$.  Then $\omega_{\infty}$ is a smooth K\"ahler metric on $M \setminus E$.

Moreover, since $\varphi(t)$ converges to $\varphi_{\infty}$ we must have, for each $x \in M \setminus E$,  $\dot{\varphi} (x, t_i) \rightarrow 0$ for a sequence of times $t_i \rightarrow \infty$.  But since $\dot{\varphi}(t)$ converges in $C^{\infty}_{\textrm{loc}}(M \setminus E)$ as $t \rightarrow \infty$ we have by uniqueness of limits that $\dot{\varphi}(t)$ converges to zero in $C^{\infty}_{\textrm{loc}}(M \setminus E)$ as $t \rightarrow \infty$.  Taking the limit of (\ref{npcma2}) as $t \rightarrow \infty$ we obtain
\begin{equation} \label{limittsuji}
\log \frac{ \omega_{\infty}^n}{\Omega} - \varphi_{\infty} =0
\end{equation}
on $M \setminus E$ and applying $\partial \ov{\partial}$ to this equation gives  $\Ric(\omega_{\infty}) = - \omega_{\infty}$ on $M \setminus E$.  This completes the proof of Theorem \ref{tsuji}.

We have now proved the existence of a singular K\"ahler-Einstein metric on $M$.  We now prove a uniqueness result.  Let $\Omega$, $\hat{\omega}_{\infty}$, $\sigma$ and $h$ be as above.

\begin{theorem} \label{eu}  There exists a unique smooth K\"ahler metric $\omega_{\emph{KE}}$ on $M \setminus E$ satisfying
\begin{enumerate}
\item[(i)] $\displaystyle{\emph{Ric}(\omega_{\emph{KE}}) = - \omega_{\emph{KE}}}$ on  $M \setminus E$.
\item[(ii)] There exists a constant $C$ and for every $\ve>0$ a constant $C_{\ve}>0$ with
\begin{equation} \label{cond1}
 \frac{1}{C_{\ve}} | \sigma|_h^{2\ve} \Omega \le \omega_{\emph{KE}}^n \le C \Omega, \qquad \textrm{on } M \setminus E.
\end{equation}
\end{enumerate}
\end{theorem}

Note that although it may appear that condition (ii) depends on the  choices of $\Omega$,  $\sigma$ and $h$, in fact it is easy to see it does  not.

\begin{proof}[Proof of Theorem \ref{eu}]   The existence part follows immediately from Theorem \ref{tsuji}, Lemma \ref{lemmat1} and Lemma \ref{phiE}, so it remains to prove uniqueness.   Suppose $\oke$ and $\toke$ are two solutions and define functions $\psi$ and $\tilde{\psi}$ on $M \setminus E$ by
\begin{equation}
\psi = \log \frac{\oke^n}{\Omega} \quad \textrm{and} \quad \tilde{\psi} = \log \frac{\toke^n}{\Omega},
\end{equation}
with $\Omega$ as in (\ref{vfcond1}).  Then we have
\begin{equation} \label{kepsi}
\oke = - \Ric(\oke) =  \hat{\omega}_{\infty} +\ddbar \psi, \quad   \toke= - \Ric(\toke) = \hat{\omega}_{\infty} + \ddbar \tilde{\psi}.
\end{equation}
Hence it suffices to show that $\psi = \tilde{\psi}$.  By symmetry it is enough to show $\psi \ge \tilde{\psi}$.

For any $\ve>0$ and $\delta>0$ sufficiently small, define
\begin{equation}
H = \psi - (1-\delta) \tilde{\psi} - \delta \ve \log |\sigma|^2_h.
\end{equation}
From the condition (\ref{cond1}),  $\tilde{\psi}$ is bounded from above and $\psi \ge \ve' \log |\sigma|^2_h -C_{\ve'}$ for any $\ve'>0$.  Taking $\ve' = \ve \delta/2$ we see that
\begin{equation}
H \ge -\frac{\ve \delta}{2} \log |\sigma|^2_h - C_{\ve'} - C,
\end{equation}
and hence
 $H$ is bounded from below by a constant depending on $\ve$ and $\delta$ and  tends to infinity on $E$.  Hence $H$ achieves a minimum at a point $x_0 \in M \setminus E$.

On the other hand, we have
\begin{equation}
\log \frac{\oke^n}{\toke^n} = \psi - \tilde{\psi},
\end{equation}
which using (\ref{kepsi}) we can rewrite as
\begin{equation}
\log \frac{ \left(\hat{\omega}_{\infty} + (1- \delta) \ddbar \tilde{\psi} - \delta \ve R_h + \ddbar H\right)^n}{\toke^n} = \psi - \tilde{\psi}.
\end{equation}
Since $\delta \hat{\omega}_{\infty} - \delta \ve R_h$ is K\"ahler for $\ve$ sufficiently small, we obtain
\begin{align}
\psi - \tilde{\psi} \ge \log \frac{ (1-\delta)^n \left(\toke + \ddbar \left(\frac{H}{1-\delta}\right) \right)^n}{\toke^n}.
\end{align}
Hence at the point $x_0$ at which $H$ achieves a minimum we have
\begin{equation}
\psi - \tilde{\psi} \ge n \log (1-\delta),
\end{equation}
\end{proof}
and so, using the inequality $\tilde{\psi} \ge \ve \log |\sigma|^2_h - C_{\ve}$,
\begin{equation}
H(x_0) \ge \delta \tilde{\psi}(x_0) + n \log (1-\delta) - \delta \ve \log |\sigma|^2_h(x_0) \ge - \delta C_{\ve} + n\log(1-\delta).
\end{equation}
For any $\ve>0$ we may choose $\delta= \delta(\ve)$ sufficiently small so that $\delta C_{\ve} < \ve/2$ and $n \log (1-\delta) > -\ve/2$, giving $H(x_0) \ge - \ve$ and hence $H \ge -\ve$ on $M \setminus E$.  It follows that on $M \setminus E$,
\begin{equation}
\psi \ge (1-\delta) \tilde{\psi} + \delta \ve \log |\sigma|^2_h - \ve.
\end{equation}
Letting $\ve \rightarrow 0$ (so that $\delta \rightarrow 0$ too) gives $\psi \ge \tilde{\psi}$ as required.  \qed

\pagebreak[3]
\subsection{Further estimates using pluripotential theory} \label{sectionpluri}

In this section we will show how results from pluripotential theory can be used to improve on the estimates given in the proof of Theorem \ref{tsuji}.

The following \emph{a priori} estimate, extending Yau's zeroth order estimate, was proved by Eyssidieux-Guedj-Zeriahi  \cite{EGZ2}.  A slightly weaker version of this result, which would also suffice for our purposes, was proved independently by Zhang \cite{Zha1}.

\begin{theorem} \label{EGZ}
Let $M$ be a compact K\"ahler manifold and $\omega$ a closed smooth semi-positive $(1,1)$-form with $\int_M \omega^n>0$.   Let $f$ be a smooth nonnegative function.   Fix $p>1$. Then if $\varphi$ is a smooth function with $\omega + \ddbar \varphi \ge 0$ solving the complex Monge-Amp\`ere equation
\begin{equation}
(\omega+ \ddbar \varphi)^n = f \omega^n,
\end{equation}
then there exists a constant $C$ depending only on $M, \omega$ and $\| f\|_{L^p(M, \omega)}$  such that
\begin{equation}
\emph{osc}_M \varphi \le C.
\end{equation}
\end{theorem}

The differences between this result and Theorem \ref{yau} are that here $\omega$ is only required to be semi-positive   and the estimate on $\varphi$ depends only on the $L^p$ bound of the right hand side of the equation.   We remark that we have not stated the result in the sharpest possible way.  The conditions that $\varphi$ and $f$ are smooth can be relaxed to $\varphi$ being bounded  with $\omega+\ddbar \varphi \ge 0$  and $f$ being in $L^p$.   We have ignored this to avoid technicalities such as defining the Monge-Amp\`ere operator in this more general setting.
We omit the proof of this theorem since it goes beyond the scope of these notes.  The theorem is a generalization of a seminal work of Ko{\l}odziej \cite{Kol1}.  For a further generalization, see \cite{BEGZ}.

We will apply Theorem \ref{EGZ} to show that the solution $\varphi=\varphi(t)$ of the parabolic complex Monge-Amp\`ere equation (\ref{npcma2}) is uniformly bounded, a result first established by Tian-Zhang \cite{TZha}.  Moreover, we can in addition obtain a bound on $\dot{\varphi}$ \cite{Zha2}.

\begin{proposition} \label{propzhang}
There exists a uniform $C$ such that under the assumptions of Theorem \ref{tsuji}, $\varphi$ solving (\ref{npcma2}) satisfies for $t \in [0,\infty)$,
\begin{equation}
\| \varphi \|_{C^0} \le C \quad \textrm{and} \quad \| \dot{\varphi} \|_{C^0} \le C.
\end{equation}
Hence there exists a uniform constant $C'>0$ such that for $t \in [0,\infty)$,
\begin{equation}
\frac{1}{C'} \Omega \le \omega^n \le C' \Omega.
\end{equation}
\end{proposition}
\begin{proof} First observe that
\begin{equation}
(\hat{\omega}_t + \ddbar \varphi)^n = f \hat{\omega}_t^n, \quad \textrm{for} \quad f = e^{\dot{\varphi} + \varphi} \frac{\Omega}{\hat{\omega}_t^n}\ge 0.
\end{equation}
From the definition of $\hat{\omega}_t$ and Lemma \ref{lemmat1} we see that $f$ is uniformly bounded from above, and hence bounded in $L^p$ for any $p$.  Applying Theorem \ref{EGZ} we see that $\textrm{osc}_M \varphi \le C$ for some uniform constant.

For the bound on $\varphi$, it only remains to check that there exists a constant $C'$ such that for each time $t$ there exists $x \in M$ with $|\varphi(x)| \le C'$.
From Lemma \ref{lemmat1} we have an upper bound for $\varphi(x)$ for all $x \in M$.  For the lower bound, observe that
\begin{equation}
\int_M e^{\dot{\varphi} + \varphi} \Omega = \int_M (\hat{\omega}_t + \ddbar \varphi)^n = \int_M \hat{\omega}_t^n \ge c,
\end{equation}
for some uniform constant $c>0$.  It follows that at each time $t$ there exists $x \in M$ with $e^{\dot{\varphi}(x)+ \varphi(x)} \ge c/\int_M \Omega$.  Since $\dot{\varphi}$ is uniformly bounded from above by Lemma \ref{lemmat1} this gives $\varphi(x) \ge - C'$ for that $x$, as required.

For the bound on $\dot{\varphi}$ we use an argument due to Zhang \cite{Zha2}.   From (\ref{npcma2}) and Theorem \ref{scalar},
\begin{equation} \label{dotphiestimate}
\ddt{} (\dot{\varphi} + \varphi) = \ddt{} \left( \log \frac{\omega^n}{\Omega} \right)= -R - n \le C_0e^{-t}
\end{equation}
for a uniform constant $C_0$.  We may suppose that  $\| \varphi \|_{C^0} \le C_0$  for the same constant $C_0>0$.  We claim that $\dot{\varphi} > - 4C_0$.  Suppose not.  Then there exists a point  $(x_0, t_0)$ with $\dot{\varphi}(x_0,t_0) \le -4C_0$.  Using (\ref{dotphiestimate}) we have for any $t>t_0$,
\begin{equation}
(\dot{\varphi}+ \varphi)(x_0, t) - (\dot{\varphi} + \varphi)(x_0,t_0) \le  C_0 \int_{t_0}^t e^{-s}ds = C_0(e^{-t_0} - e^{-t}).
\end{equation}
Hence for $t>t_0$,
\begin{equation}
\dot{\varphi} (x_0,t)  \le (\dot{\varphi} + \varphi)(x_0,t_0) + C_0e^{-t_0} - \varphi(x_0,t) \le - C_0,
\end{equation}
using the fact that $\dot{\varphi}(x_0,t_0) \le - 4C_0$.  This is a contradiction since $\varphi(x_0,t)$ is uniformly bounded as $t \rightarrow \infty$.
 \qed
\end{proof}

An immediate consequence is:

\begin{corollary}
The singular K\"ahler-Einstein metric $\omega_{\emph{KE}}$ constructed in Theorem \ref{tsuji} satisfies
\begin{equation}
\frac{1}{C} \Omega \le \omega_{\emph{KE}}^n \le C \Omega \quad \textrm{on } M \setminus E,
\end{equation}
for some $C>0$.
\end{corollary}

As another application of Proposition \ref{propzhang}, we use the estimate on $\varphi$ together with the parabolic Schwarz lemma to obtain a lower bound on the metric $\omega$.

\begin{lemma} \label{applyschwarz} Under the assumptions  of Theorem \ref{tsuji},
there exists a uniform constant $C$ such that
\begin{equation}
\omega \ge \frac{1}{C} \hat{\omega}_{\infty}, \quad \textrm{on} \quad M \times [0, \infty).
\end{equation}
\end{lemma}
\begin{proof}
Recall  that $\hat{\omega}_{\infty} = \frac{1}{m}\Phi^* \omega_{\textrm{FS}}$ where $\Phi: M \rightarrow \mathbb{P}^N$ is a holomorphic map  and $\omega_{\textrm{FS}}$ is the Fubini-Study metric on $\mathbb{P}^N$.  We can then directly apply Theorem \ref{psl} to obtain
\begin{equation}
\left( \ddt{} - \Delta \right) \log \tr{\omega}{\hat{\omega}_{\infty}} \le C' \tr{\omega}{\hat{\omega}_{\infty}}+1,
\end{equation}
for $C'$ an upper bound for the bisectional curvature of $\omega_{\textrm{FS}}$.
Define $Q = \log \tr{\omega}{\hat{\omega}_{\infty}} - A \varphi$ for $A$ to be determined later.  Compute, using Proposition \ref{propzhang},
\begin{align} \nonumber
\left( \ddt{} - \Delta \right) Q & \le C' \tr{\omega}{\hat{\omega}_{\infty}} - A \dot{\varphi} + An - A \tr{\omega}{\hat{\omega}_t} +1 \\
& \le - \tr{\omega}{\hat{\omega}_{\infty}} + C,
\end{align}
where we have chosen $A$ to be sufficiently large so that $A \hat{\omega}_{t} \ge (C'+1) \hat{\omega}_{\infty}$.
It follows from the maximum principle that $Q$ and hence $\tr{\omega}{\hat{\omega}_{\infty}}$ is uniformly bounded from above and this completes the proof of the lemma.  \qed
\end{proof}

Observe that Lemma \ref{applyschwarz} together with the volume upper bound from Lemma \ref{lemmat1} show that the metric $\omega(t)$ is uniformly bounded above and below  on compact subsets of $M \setminus S$, for $S$ the set of points where  $\hat{\omega}_{\infty}$ is degenerate.  Thus we can obtain an alternative proof of Theorem \ref{tsuji} which avoids the use of Lemma \ref{phiE} and Lemma \ref{trick2}.

Finally we mention that Zhang \cite{Zha2} also proved a uniform bound for the scalar curvature of the evolving metric in this setting.

\pagebreak

\section{The K\"ahler-Ricci flow on a  product elliptic surface} \label{sectpes}

In this section we investigate \emph{collapsing} along the K\"ahler-Ricci flow.  We study this behavior in the simple case of a product of two Riemann surfaces.

\subsection{Elliptic surfaces and the K\"ahler-Ricci flow}

Let $M$ now have complex dimension two.  An \emph{elliptic curve} $E$ is a compact Riemann surface with $c_1(E)=0$ (by the Gauss-Bonnet formula this is equivalent to having genus equal to 1).   We say that $M$ is an \emph{elliptic surface} if there exists a surjective holomorphic map $\pi: M \rightarrow S$ onto a Riemann surface $S$ such that the fiber $\pi^{-1}(s)$ is an elliptic curve for all but finitely many $s \in S$.
In particular, the product of an elliptic curve and any Riemann surface is an elliptic surface, which we will call a \emph{product elliptic surface}.

 In \cite{SoT1}, the K\"ahler-Ricci flow was studied on a general minimal  elliptic surface (see Section \ref{sectlast} for a definition of \emph{minimal}).  In this case there are finitely many singular fibers of the map $\pi$.   It was shown that the K\"ahler-Ricci flow converges in $C^{1+\beta}$ for any $\beta \in (0,1)$ at the level of potentials away from the singular fibers, and also converges on $M$ in the sense of currents, to a \emph{generalized K\"ahler-Einstein metric} on the base $S$.   A higher dimensional analogue was given in \cite{SoT2}.

Here we study the behavior of the K\"ahler-Ricci flow in the more elementary case of a product elliptic surface
$M= E\times S$, where $E$ is an elliptic curve and $S$ is a Riemann surface with $c_1(S)<0$ (genus greater than 1).  Because of the simpler structure of the manifold, we can obtain stronger estimates than in \cite{SoT1}.

  By the uniformization theorem for Riemann surfaces (or the results of Section \ref{sectn0}), $S$ and $E$ admit K\"ahler metrics of constant curvature which are unique up to scaling.  Hence we can define K\"ahler metrics $\KES$ on $S$ and $\KEE$ on $E$ by
\begin{equation}
\Ric(\KES) = - \KES, \quad \Ric(\KEE) = 0, \quad \int_E \KEE = 1.
\end{equation}
Denote by $\pi_S$ and $\pi_E$ the projection maps $\pi_S: M \rightarrow S$ and $\pi_E: M \rightarrow E$.

As in the case of the previous section we consider the normalized K\"ahler-Ricci flow
\begin{equation} \label{nkrf3}
\ddt{} \omega  = - \Ric(\omega) - \omega, \qquad \omega|_{t=0} = \omega_0,
\end{equation}

The first Chern class of $M$ is given by $c_1(M) = - [\pi^* \KES]$, which can be seen from the equation
\begin{equation} \label{ricse}
\Ric( \pi_S^* \KES +  \pi_E^*  \KEE) =  - \pi_S^* \KES.
\end{equation}
Since $\pi^* \KES$ is a nonnegative (1,1) form on $M$, it follows from Theorem \ref{longtime} that a solution to (\ref{nkrf3}) exists for all time for any initial K\"ahler metric $\omega_0$.

As a simple example, first consider the  case when the initial metric $\omega_0$ splits as a product.  Suppose
$\omega_0 =   \pi_E^* \omega_{E}^0 + \pi_S^* \omega^0_{S}$, where $\omega_{E}^0$ and $\omega^0_{S}$ are smooth metrics on $E$ and $S$ respectively.
Then the K\"ahler-Ricci flow splits into the K\"ahler-Ricci flows on $E$ and $S$, with $\omega(t) =  \pi_E^* \omega_{E,t} + \pi_S^* \omega_{S,t}$ where $\omega_{E,t}$ and $\omega_{S,t}$ solve the  K\"ahler-Ricci flow on $E$ and $S$ respectively.  Since $c_1(E)=0$ and $c_1(S)<0$ we can apply the results of
 Section \ref{sectn0}  to see that  $\omega_{E,t}$ converges in $C^{\infty}$ to $0$ (because of the normalization) as $t\rightarrow \infty$ and $\omega_{S,t}$ converges in $C^{\infty}$ to $\KES$. Hence the solution to the original normalized K\"ahler-Ricci flow converges in $C^{\infty}$ to $\pi_S^* \KES$.


We now turn back to the general case of a non-product metric.  For convenience, here and henceforth we will  drop the $\pi^*_S$ and $\pi^*_E$ and write $\KES$ and $\KEE$ for the $(1,1)$-forms pulled back to $M$.   We prove:

\pagebreak[3]
\begin{theorem} \label{pe} Let $\omega(t)$ be the solution of the normalized Kahler-Ricci flow (\ref{nkrf3}) on $M=E \times S$ with  initial Kahler metric $\omega_0$.
Then
\begin{enumerate}
\item[(i)] For any $\beta \in (0,1)$, $\omega(t)$ converges to $\omega_S$ in $C^\beta(M, g_0)$.
\item[(ii)] The curvature tensors of $\omega(t)$ and their derivatives are uniformly bounded along the flow.
\item[(iii)] For any fixed fiber $E=\pi_S^{-1}(s)$, we have
\begin{equation}
\| e^t \omega(t)|_{E} - \omega_{\emph{flat}} \|_{C^0(E)} \rightarrow 0 \quad \textrm{as} \quad t \rightarrow \infty,
\end{equation}
where $\omega_{\emph{flat}}$ is the K\"ahler Ricci-flat metric on $E$ with $\int_E \omega_{\emph{flat}} = \int_E \omega_0$.
\end{enumerate}
\end{theorem}

\begin{remark} \emph{We conjecture that in (i), the convergence in $C^{\beta}(M)$ can be replaced by $C^{\infty}(M)$ convergence.  Such a result is contained in the work of Gross-Tosatti-Zhang \cite{GTZ} for the case of a family of Ricci-flat metrics.  It seems likely that their methods could be extended to cover this case too.  It would also be interesting to find a proof of $C^{\infty}$ convergence using only the maximum principle.} \end{remark}

Since the normalized K\"ahler-Ricci flow exists for all time we can compute, as in (\ref{ode1}) and (\ref{ode2}), the evolution of the K\"ahler class to be
\begin{equation}
[\omega(t)]  = e^{-t} [\omega_0] + (1-e^{-t}) [ \KES].
\end{equation}

Before proving Theorem \ref{pe} we will, as in the sections above, reduce (\ref{nkrf3}) to a parabolic complex Monge-Amp\`ere equation.   Define reference metrics $\hat{\omega}_t \in [\omega(t)]$ by
\begin{equation}
\hat{\omega}_t = e^{-t} \omega_0 + (1-e^{-t})   \KES, \qquad \textrm{for } t \in [0,\infty).
\end{equation}
Define a smooth volume form  $\Omega$ on $M$ by
\begin{equation}
\ddbar \log \Omega =  \KES \in - c_1(M), \quad \int_M \Omega = 2 \int_M \omega_0 \wedge \omega_S.
\end{equation}
In fact, from (\ref{ricse}) one can see that $\Omega$ is a constant multiple of $\omega_S \wedge \omega_E$.
We consider the parabolic complex Monge-Amp\`ere equation
\begin{equation} \label{npcma4}
\ddt{ \varphi} = \log \frac{ e^t (\hat{\omega}_t + \ddbar \varphi)^2}{\Omega} - \varphi, \qquad \hat{\omega}_t+ \ddbar \varphi>0 , \qquad \varphi|_{t=0} =0.
\end{equation}
As in earlier sections, a solution $\varphi = \varphi(t)$ of (\ref{npcma4}) exists for all time and  $\omega = \hat{\omega}_t + \ddbar \varphi$ solves the normalized K\"ahler-Ricci flow.  Note that we insert the factor of $e^t$ in the equation to ensure that $\varphi$ is uniformly bounded (see Lemma \ref{2phi} below) but of course it does not change the evolution of the metric along the flow.

\subsection{Estimates}

In this section we establish uniform estimates for the solution $\varphi= \varphi(t)$ of (\ref{npcma4}), which we know exists for all time.

\begin{lemma} \label{2phi} There exists $C>0$ such that on $M \times [0,\infty)$,
\begin{enumerate}
\item[(i)] $\displaystyle{| \varphi| \le C.}$
\item[(ii)] $\displaystyle{ | \dot{\varphi}| \le C.}$
\item[(iii)] $\displaystyle{\frac{1}{C} \hat{\omega}_t^2 \le \omega^2 \le  C \hat{\omega}_t^2.}$
\end{enumerate}
\end{lemma}

\begin{proof}  For (i), first note that since $e^t \hat{\omega}_t^2 = e^{-t} \omega_0^2 + 2(1-e^{-t}) \omega_0 \wedge   \KES$ we have
\begin{equation} \label{comparevf}
\frac{1}{C} \Omega \le e^t \hat{\omega}_t^2 \le C \Omega.
\end{equation}
Hence if $\varphi$ achieves a maximum at $(x_0, t_0)$ with $t_0>0$ then at that point,
\begin{equation}
0 \le \ddt{}  \varphi   \le \log \frac{e^t \hat{\omega}_t^2}{\Omega} - \varphi  \le \log C - \varphi,
\end{equation}
giving $\varphi \le \log C$.  The lower bound of $\varphi$ follows similarly.

For (ii) observe that $\ddt{} \hat{\omega}_t =    \KES - \hat{\omega}_t$ and hence
\begin{equation}
\left( \ddt{} - \Delta \right) \dot{\varphi} = \tr{\omega}{(   \KES - \hat{\omega}_t)} + 1 - \dot{\varphi}.
\end{equation}
By definition of $\hat{\omega}_t$ there exists a uniform constant $C_0>0$ such that $C_0 \hat{\omega}_t \ge   \KES$.  For the upper bound of $\dot{\varphi}$, we apply the maximum principle to $Q_1 = \dot{\varphi} - (C_0-1) \varphi$. Compute
\begin{align} \nonumber
\left( \ddt{} - \Delta \right) Q_1 & = \tr{\omega} (   \KES - \hat{\omega}_t) + 1 - C_0 \dot{\varphi} + (C_0-1)\tr{\omega}{(\omega - \hat{\omega}_t)} \\
& \le 1 - C_0 \dot{\varphi} + 2(C_0-1),
\end{align}
and we see that $\dot{\varphi}$ is uniformly bounded from above at a point where $Q_1$ achieves a maximum.  Since $\varphi$ is bounded by (i) we obtain the required  upper bound of $\dot{\varphi}$.

For the lower bound of $\dot{\varphi}$, let $Q_2 = \dot{\varphi} + 2 \varphi$ and compute
\begin{align} \nonumber
\left( \ddt{} - \Delta \right) Q_2 & = \tr{\omega} (   \KES - \hat{\omega}_t) + 1 + \dot{\varphi} - 2 \tr{\omega}{(\omega - \hat{\omega}_t)} \\
& \ge \tr{\omega}{\hat{\omega}_t}  +  \dot{\varphi} -3.
\end{align}
Using (\ref{npcma4}), (\ref{comparevf}) and the arithmetic-geometric means inequality, we have at a point $(x_0, t_0)$ where $Q_2$ achieves a minimum,
\begin{equation}
e^{-(\dot{\varphi}+\varphi)/2} = \left( \frac{ \Omega}{e^t \omega^2} \right)^{1/2} \le C \left( \frac{\hat{\omega}_t^2}{\omega^2} \right)^{1/2} \le \frac{C}{2} \tr{\omega}{\hat{\omega}_t} \le C' - \dot{\varphi}.
\end{equation}
Hence $\dot{\varphi}$ is uniformly bounded from below at $(x_0, t_0)$, giving (ii).  Part (iii) follows from (i) and (ii).  \qed \end{proof}

Next we   estimate $\omega$ in terms of $\hat{\omega}_t$.    It is convenient to define another family of reference metrics $\tilde{\omega}_t$ whose curvature we can control more precisely.  Define $\tilde{\omega}_0 =   \KEE +   \KES$ and
\begin{equation} \label{tildeomega}
\tilde{\omega}_t = e^{-t} \tilde{\omega}_0 + (1-e^{-t})  \KES = \KES + e^{-t} \KEE.
\end{equation}
Observe that $\tilde{\omega}_t$ and $\hat{\omega}_t$ are uniformly equivalent.

\begin{lemma} \label{mbell} There exists $C>0$ such that on $M \times [0, \infty)$,
\begin{equation}
\frac{1}{C} \tilde{\omega}_t \le \omega \le C \tilde{\omega}_t
\end{equation}
\end{lemma}
\begin{proof}
From part (iii) of Lemma \ref{2phi} it suffices to obtain an upper bound of the quantity  $\tr{\tilde{\omega}_t}{\omega}$ from above.
Compute using the argument of Proposition \ref{propkeyeqn},
\begin{align}
\left( \ddt{} - \Delta \right)  \tr{\tilde{\omega}_t}{\omega} & \le - \tr{\tilde{\omega}_t}{\omega} - g^{ \ov{j}i} \tilde{R}_{i \ov{j}}^{\ \ \,  \ov{\ell}k}  g_{k \ov{\ell}}
 \mbox{} - \tilde{g}_t^{\ov{j} i} g^{\ov{q} p} g^{\ov{\ell} k} \tilde{\nabla}_i g_{p \ov{\ell}} \tilde{\nabla}_{\ov{j}} g_{k \ov{q}} + \left( \ddt{} \tilde{g}_t^{ \ov{j}i} \right) g_{i \ov{j}}, \label{mpropkeyeqn}
\end{align}
where we are using $\tilde{R}_{i \ov{j}}^{\ \ \,  \ov{\ell}k}$ and $\tilde{\nabla} = \nabla_{\tilde{g}_t}$ to denote the curvature and covariant derivative  with respect to $\tilde{g}_t$.
Since $\ddt{} \tilde{\omega}_t = - \tilde{\omega}_t +  \KES \ge - \tilde{\omega}_t$, we have
\begin{equation} \label{that}
 \left( \ddt{} \tilde{g}_t^{ \ov{j}i} \right) g_{i \ov{j}} \le \tr{\tilde{\omega}_t}{\omega}.
\end{equation}
Hence, from the argument of Proposition \ref{propChat},
\begin{equation} \label{thisthat}
\left( \ddt{} - \Delta \right) \log \tr{\tilde{\omega}_t}{\omega} \le - \frac{1}{\tr{\tilde{\omega}_t}{\omega}} g^{\ov{j}i} \tilde{R}_{i \ov{j}}^{\ \ \,  \ov{\ell}k}  g_{k \ov{\ell}}.
\end{equation}
Next we claim that
\begin{align} \nonumber
- g^{ \ov{j}i}  \tilde{R}_{i \ov{j}}^{\ \ \,  \ov{\ell}k} g_{k \ov{\ell}} & = (\tr{\omega}{ \KES})  \frac{ 2\, \KEE \wedge \omega}{\tilde{\omega}_0^2} \\ & \le (\tr{\omega}{ \KES})  (\tr{\tilde{\omega}_0}{\omega}) \le (\tr{\omega}{ \KES} ) (\tr{\tilde{\omega}_t}{\omega}).  \label{this}
\end{align}
To see (\ref{this}), compute in a local holomorphic product coordinate system $(z^1, z^2)$ with $z^1$ a normal coordinate for $\KES|_S$ in the base $S$ direction and $z^2$ a normal coordinate for $\KEE|_E$ in the fiber $E$ direction.  In these coordinates $\tilde{g}_t$ is diagonal and $(\tilde{g}_t)_{1 \ov{1}} = (g_S)_{1 \ov{1}}$.  Since the curvature of $\KEE$ vanishes, we have from (\ref{ricse})
\begin{equation} \label{R1111}
\tilde{R}_{1 \ov{1} 1 \ov{1}} = - ( g_S)_{1 \ov{1}} ( g_S)_{1 \ov{1}},
\end{equation}
and $\tilde{R}_{i \ov{j} k \ov{\ell}}=0$ if $i,j, k$ and  $\ell$ are not all equal to 1.
Hence the only non-zero component of the curvature of $\tilde{\omega}_t$ appearing in (\ref{this}) is $\tilde{R}_{1 \ov{1}}^{\ \ \, 1 \ov{1}}=-1$.  This gives the first equality of (\ref{this}), and the next two inequalities follow from the definition of $\tilde{\omega}_0$ and $\tilde{\omega}_t$.

Combining (\ref{thisthat}), (\ref{this}) we have
\begin{equation}
\left( \ddt{} - \Delta \right) \log \tr{\tilde{\omega}_t}{\omega} \le \tr{\omega}{ \KES}.
\end{equation}
Now define
\begin{equation}
Q_3 = \log \tr{\tilde{\omega}_t}{\omega} - A \varphi,
\end{equation}
for $A=C_0+1$  where $C_0$ is the positive constant with $C_0 \hat{\omega}_t \ge \KES$
and compute
\begin{align} \nonumber
\left( \ddt{} - \Delta \right) Q_3 & \le \tr{\omega}{ \KES} - A \dot{\varphi} + A \tr{\omega}{(\omega - \hat{\omega}_t)} \\ \nonumber
& \le C - \tr{\omega}{\hat{\omega}_t} \\
& \le C - \frac{1}{C'} \tr{\tilde{\omega}_t}{\omega},
\end{align}
for some $C'>0$.  For the last line we have used the estimate (iii) of Lemma \ref{2phi} and the fact that $\tilde{\omega}_t$ and $\hat{\omega}_t$ are uniformly equivalent.  Since $\varphi$ is uniformly bounded from part (i) of Lemma \ref{2phi} we see that $Q_3$ is bounded from above by the maximum principle, completing the proof of the lemma. \qed
\end{proof}

Next we prove an estimate on the derivative of $\omega$ using an argument similar to that of Theorem \ref{theoremSbound}.



\begin{lemma}  \label{lemmaSe} There exists a uniform constant $C$ such that on $M \times [0,\infty)$,
\begin{equation}
S := | \nabla_{\tilde{g}_0} g|^2 \le C \quad \textrm{and} \quad | \nabla_{\tilde{g}_0} g|^2_{\tilde{g}_0} \le C,
\end{equation}
where $| \cdot |$, $| \cdot |_{\tilde{g}_0}$ denote the norms with respect to the metrics $g=g(t)$ and $\tilde{g}_0$ respectively.  Moreover, we have
\begin{align} \label{S2}
\left( \ddt{} - \Delta \right) S & \le - \frac{1}{2}| \emph{Rm}(g) |^2 + C'
\end{align}
\end{lemma}
for a uniform constant $C'$.
\begin{proof}
First we show that
\begin{equation} \label{goodtr}
\left( \ddt{} - \Delta \right)  \tr{\tilde{\omega}_t}{\omega} \le C - \frac{1}{C'}| \nabla_{\tilde{g}_0} g|^2,
\end{equation}
for uniform constants $C, C'$.
From (\ref{mpropkeyeqn}), (\ref{that}), (\ref{this}) and part (iii) of Lemma \ref{2phi},
\begin{align}\label{goodtr2}
\left( \ddt{} - \Delta \right)   \tr{\tilde{\omega}_t}{\omega} & \le C -  \tilde{g}_t^{\ov{j} i} g^{\ov{q} p} g^{\ov{\ell} k} \tilde{\nabla}_i g_{p \ov{\ell}} \tilde{\nabla}_{\ov{j}} g_{p \ov{q}}  \le C - \frac{1}{C'} S.
\end{align}
For the last inequality we are using the fact that $\nabla_{\tilde{g}_t} = \nabla_{\tilde{g}_0}$ which can be seen by choosing a coordinate system at a point in which  $\partial_i \tilde{g}_t=0$ for all $i$ and any $t \ge 0$.  This establishes (\ref{goodtr}).

Using the notation of Proposition \ref{propPSS1}, write $\Psi^k_{ij} = \Gamma^k_{ij} - \Gamma(\tilde{g}_0)^k_{ij}$ so that
$S=  |\Psi|^2$.  Then
\begin{align} \nonumber
\left( \ddt{} - \Delta \right) S & = - | \ov{\nabla} \Psi |^2 - | \nabla \Psi |^2 +  | \Psi |^2
 - 2 \textrm{Re} \left( \, g^{\ov{j} i} g^{\ov{q} p} g_{k \ov{\ell}} \nabla^{\ov{b}} R(\tilde{g}_0)_{i \ov{b} p}^{\ \ \ k} \ov{\Psi_{j q}^{\ell }} \right).
\end{align}
We have
\begin{equation}  \label{nablarm}
\nabla^{\ov{b}} R(\tilde{g}_0)_{i \ov{b} p}^{\ \ \ k} = - g^{\ov{b}a} \Psi^m_{ia} R(\tilde{g}_0)_{m \ov{b} p}^{\ \ \ \ \, k} - g^{\ov{b} a} \Psi^m_{pa} R(\tilde{g}_0)_{i \ov{b} m}^{\ \ \ \ k} + g^{\ov{b}a} \Psi^k_{ma} R(\tilde{g}_0)_{i \ov{b}p}^{ \ \ \ m}.
\end{equation}
Indeed, as in the proof of Lemma \ref{mbell}, this can be seen by choosing a local holomorphic product coordinate system $(z^1, z^2)$ centered at a point $x$ with $z^1$ normal for $\omega_S$ and $z^2$ normal for $\omega_E$.
Using the argument of (\ref{this}) and the result of Lemma \ref{mbell} we have
\begin{align} \nonumber
| \textrm{Rm}(\tilde{g}_0)|_g^2 & := g^{\ov{j} i} g^{\ov{\ell}k} g^{\ov{q}p} g_{a\ov{b}} R(\tilde{g}_0)_{i\ov{\ell} p}^{\ \ \ a} \ov{ R(\tilde{g}_0)_{j \ov{k} q}^{ \ \ \ \, \, b}} \\ \nonumber
& = (\tr{\omega}{\omega_S})^3 \frac{2\, \omega_E \wedge \omega}{\tilde{\omega}_0^2} \\ \label{nablarm2}
& \le (\tr{\omega}{\tilde{\omega}_t})^3 \tr{\tilde{\omega}_t}{\omega} \le C.
\end{align}
Combining (\ref{nablarm}) and (\ref{nablarm2}),
\begin{equation}
\left| 2 \textrm{Re} \left( \, g^{\ov{j} i} g^{\ov{q} p} g_{k \ov{\ell}} \nabla^{\ov{b}} R(\tilde{g}_0)_{i \ov{b} p}^{\ \ \ k} \ov{\Psi_{j q}^{\ell }} \right) \right| \le C S.
\end{equation}
Since $| \ov{\nabla} \Psi |^2 = | \textrm{Rm}(\tilde{g}_0) - \textrm{Rm}(g)|_g^2$, we compute
\begin{align} \nonumber
\left( \ddt{} - \Delta \right) S & \le  - | \ov{\nabla} \Psi |^2 - | \nabla \Psi |^2 +  CS \\ \label{Spre2}
& \le - \frac{1}{2}| \textrm{Rm}(g) |^2 + CS + C'
\end{align}
Then the upper bound on $S$ follows from (\ref{goodtr2}) and (\ref{Spre2})  by applying the maximum principle to $S + A \tr{\tilde{\omega}_t}{\omega}$ for sufficiently large $A$.
The inequality $| \nabla_{\tilde{g}_0} g|^2_{g_0} \le C$ follows from the fact that the metric $g(t)$ is bounded from above by $g_0$ (Lemma \ref{mbell}).
  The inequality (\ref{S2}) follows from (\ref{Spre2}). \qed
\end{proof}

We then easily obtain estimates for curvature and all covariant derivatives of curvature, establishing part (ii) of Theorem \ref{pe}.

\begin{lemma} \label{curvesti}
There exist uniform constants $C_m$ for $m=0, 1, 2, \ldots$ such that on $M \times [0,\infty)$,
\begin{equation}
 | \nabla^m_{\mathbb{R}} \emph{Rm}(g) |^2 \le C_m.
\end{equation}
\end{lemma}
\begin{proof}
This follows from Lemma \ref{lemmaSe} and the arguments of Theorem \ref{theoremcurv1} and \ref{theoremcurv2}. \qed
\end{proof}

\subsection{Fiber collapsing and convergence}

In this subsection, we complete the proof of Theorem \ref{pe}.

First we define a closed $(1,1)$ form $\oflat$ on $M$ with the properties that $[\oflat]=[\omega_0]$ and for each $s \in S$,  $\oflat$ restricted to the fiber $\pi_S^{-1}(s)$ is a K\"ahler-Ricci flat metric.  To do this, fix $s \in S$ and define a function $\rho_s$ on $\pi_S^{-1}(s)$ by
\begin{equation}
\omega_0|_{\pi_S^{-1}(s)} + \ddbar \rho_s>0, \quad \textrm{Ric}\left( \omega_0|_{\pi_S^{-1}(s)} + \ddbar \rho_s\right) =0, \quad \int_{\pi_S^{-1}(s)} \rho_s \, \omega_0 =0.
\end{equation}
Since $\rho_s$ satisfies a partial differential equation with parameters depending smoothly on $s \in S$, it follows that $\rho_s$ varies smoothly with $s$ and hence defines a  smooth function on  $M$, which we will call $\rho$.  Now set $\oflat := \omega_0 + \ddbar \rho$.  This is a closed $(1,1)$ form with the desired properties.  Note that for each $s$ in $S$, $\oflat|_{\pi_S^{-1}(s)}$ is a metric, but $\oflat$  may not be positive definite as a $(1,1)$ form on $M$.

We make use of $\oflat$ to prove the following estimate on $\varphi$.

\begin{lemma}\label{c0conv} There exists $C>0$ such that on $M\times [0, \infty)$,
\begin{equation}
|\varphi |\leq C(1+t)e^{-t}.
\end{equation}
\end{lemma}
\begin{proof}  Since $\oflat$ is a constant multiple of $\omega_E$ when restricted to each fiber, we see from the definition of $\Omega$ that
\begin{equation}
\Omega =  2\omega_S \wedge \oflat.
\end{equation}
Let $Q= \varphi  - e^{-t} \rho.$ Then
\begin{equation}
 \ddt Q = \log \frac{ e^t  (e^{-t} \oflat+ (1-e^{-t} ) \omega_S + \ddbar Q )^2}{2\omega_S \wedge \oflat} - Q .
 \end{equation}
For a positive constant $A$, consider the quantity $Q_1 = e^t Q - At$.  At a point $(x_0,t_0)$ with $t_0>0$ where $Q_1$ achieves a maximum, we have
\begin{align} \nonumber
0 \le \ddt{} Q_1 & \le e^t \log \frac{ e^t  (e^{-t} \oflat+ (1-e^{-t} ) \omega_S)^2}{2\omega_S \wedge \oflat}  - A \\
& \le e^t \log (1+ C e^{-t}) -A \le C' - A,
\end{align}
 for uniform constants $C, C'$.  Choosing $A>C'$ gives a contradiction.  Hence $Q_1$ is bounded from above.  It follows that
 $\varphi \le C (1+t)e^{-t}$ for a uniform constant $C$.   The lower bound for $\varphi$ is similar.
 \qed
\end{proof}

\begin{lemma} \label{pesmooth} Fix $\beta \in (0,1)$.  We have
\begin{enumerate}
\item[(i)] $\varphi(t) \rightarrow 0$ in $C^{2+\beta}(M)$ as $t \rightarrow \infty$.
\item[(ii)] $\omega(t) \rightarrow \omega_S$ in $C^{\beta}(M)$ as $t\rightarrow \infty$.
\item[(iii)] $\displaystyle{\ddt\varphi \rightarrow 0}$ in $C^0(M)$ as $t \rightarrow \infty$.
\end{enumerate}
\end{lemma}

\begin{proof}  From Lemma \ref{lemmaSe} the tensor $\nabla_{\tilde{g}_0} g$ is bounded with respect to the fixed metric $\tilde{g}_0$.  Moreover, $g \le C \tilde{g}_0$ for some   uniform $C$. It follows that $\Delta_{\tilde{g}_0} \varphi$ is bounded in $C^1(M, \tilde{g}_0)$.  Since $\varphi$ is bounded in $C^0$, we can apply the standard Schauder estimates for Poisson's equation \cite{GT}, to see that $\varphi$ is bounded in $C^{2+\alpha}$ for any $\alpha \in (0,1)$.  Choosing $\alpha>\beta$,
part (i) follows from this together with Lemma \ref{c0conv}.
Part (ii) follows from part (i) and the fact that $\hat{\omega}_t$ converges in $C^{\infty}$ to $\omega_S$ as $t \rightarrow \infty$.

For part (iii), suppose for a contradiction that there exist $\ve>0$ and a sequence $\{ (x_i, t_i)\}_{i\in \mathbb{N}} \subset M \times [0, \infty)$ with $t_i \rightarrow \infty$ and
\begin{equation}
\left| \dot{\varphi} \right| (x_i, t_i) > \ve.
\end{equation}
From Lemma \ref{2phi} and Lemma \ref{curvesti}, the quantity
\begin{equation}
\ddt{} \dot{\varphi}  = -R(\omega) - 1 - \dot\varphi
\end{equation}
is uniformly bounded in $C^0(M\times [0, \infty))$.   Hence there exists a uniform constant $\delta>0$ such that for each $i$,
\begin{equation}
\left| \dot{\varphi}\right|(x_i, t) \ge \frac{\ve}{2} \quad \textrm{for all } t \in [t_i, t_i + \delta].
\end{equation}
 Hence
 \begin{align} \nonumber
 \frac{\ve  \delta}{2} \leq  \int_{t_i}^{t_i+\delta} | \dot{\varphi}|(t, x_i) dt
 & = \left|\int_{t_i}^{t_i +\delta} \dot{\varphi}(x_i,t) dt \right| \\  \nonumber
 & = |\varphi(x_i, t_i + \delta) - \varphi(x_i, t_i)|  \\ & \le \sup_{x \in M} | \varphi(x, t_i+ \delta) - \varphi (x, t_i) |,
 \end{align}
a contradiction since $\varphi(t)$ converges uniformly to $0$
  in $C^0(M)$ as $t \rightarrow \infty$.
\qed
\end{proof}

Finally, we prove part (iii) of Theorem \ref{pe}.

\begin{lemma} \label{fiberflat}  Fix $s \in S$ and write $E = \pi^{-1}_S(s)$ for the fiber over $s$.  Write $\omega_{\emph{flat}} = \omega_{\emph{flat}}|_{E}$.  Then on $E$,
\begin{equation}
e^t \omega(t)|_{E} \rightarrow \omega_{\emph{flat}} \quad \textrm{as }  t \rightarrow \infty,
\end{equation}
where the convergence is uniform on $C^0(E)$.  Moreover, the convergence is uniform in $s \in S$.
\end{lemma}
\begin{proof}  We use here an argument similar to one found in \cite{To2}.  Applying Lemma \ref{lemmaSe} we have
\begin{equation}
| \nabla_{g_E}(g|_E) |^2_{g|_E}  \le | \nabla_{\tilde{g}_0} g |^2 \le C.
\end{equation}
From Lemma \ref{mbell}, we see that $g|_E$ is uniformly equivalent to $e^{-t} g_E$.  It follows that
\begin{equation}
| \nabla_{g_E} (e^t g|_E)|^2_{g_E} = e^{-t} | \nabla_{g_E} (g|_E) |^2_{e^{-t} g_E} \le C e^{-t} | \nabla_{g_E} (g|_E) |^2_{g|_E} \le  C' e^{-t}.
\end{equation}
Since $\gflat$ is a constant multiple of $g_E$ we see that
\begin{equation}
| \nabla_{g_E} (e^t g|_E - g_{\textrm{flat}})|^2_{g_E} \le C' e^{-t}.
\end{equation}
Moreover, $[e^t \omega|_E] = [\oflat]$.  It is now straightforward to complete the proof of the lemma.  Indeed, any two K\"ahler metrics on the Riemann surface $E$ are conformally equivalent and hence we can write $e^t\omega|_E = e^\sigma \oflat$ for a smooth function $\sigma = \sigma(x,t)$ on $E \times [0, \infty)$.  We have
\begin{equation}
| d (e^\sigma -1) |^2_{g_E} \rightarrow 0 \quad \textrm{as } t \rightarrow \infty, \quad \textrm{and} \quad \int_E (e^\sigma-1) \omega_E =0.
\end{equation}
From the second condition, for each time $t$ there exists $y(t) \in E$ with $\sigma(y(t),t)=0$ and hence by the Mean Value Theorem for manifolds,
\begin{equation}
|e^{\sigma(x,t)}-1| = |(e^{\sigma(x,t)} -1) - (e^{\sigma(y(t),t)} -1)| \rightarrow 0 \quad \textrm{as } t \rightarrow \infty,
\end{equation}
uniformly in $x \in E$. This says precisely that $e^t \omega(t)|_E \rightarrow \oflat$ uniformly as $t\rightarrow \infty$.  Moreover, none of our constants depend on the choice of $s \in S$.  This completes the proof of the lemma.
\qed
\end{proof}

Combining Lemma \ref{curvesti} with Lemmas \ref{pesmooth} and \ref{fiberflat}
 completes the proof of Theorem \ref{pe}.

\pagebreak
\section{Finite time singularities} \label{sectfinite}

In this section, we describe some behaviors of the K\"ahler-Ricci flow in the case of a finite time singularity.  The complete behavior of the flow is far from understood, and is the subject of current research.  In Section \ref{basicestimates}, we prove some basic estimates, most of which hold under fairly weak hypotheses.  Next, in Section \ref{sectscalarcurvature}, we describe a result of \cite{Zha3} on the behavior of the scalar curvature and discuss some speculations.  In Sections \ref{sectexamples} and \ref{sectfinitecollapsing} we describe, without proof, some recent results \cite{SSW, SW2} and illustrate with an example.

\subsection{Basic estimates} \label{basicestimates}

We now consider the K\"ahler-Ricci flow
 \begin{equation} \label{krffinite}
\ddt{} \omega  = - \Ric(\omega), \qquad \omega|_{t=0} = \omega_0,
\end{equation}
in the case when $T<\infty$.  The cohomology class $[\omega_0] - T c_1(M)$ is a limit of K\"ahler classes but is itself no longer K\"ahler.  The behavior of the K\"ahler-Ricci flow as $t$ tends towards the singular time $T$ will depend crucially on properties of this cohomology class.

We first observe  that since $T<\infty$ we immediately have from Corollary \ref{volform}  the estimate
\begin{equation} \label{vfb2}
\omega^n \le C \Omega,
\end{equation}
for a uniform constant $C$.

As in Section \ref{maximal} we reduce (\ref{krf}) to a parabolic complex Monge-Amp\`ere equation.
Choose a closed (1,1) form $\hat{\omega}_T$ in the cohomology  class $[\omega_0] - T c_1(M)$.
Given this we can define a family of reference forms $\hat{\omega}_t$ by
\begin{equation}
\hat{\omega}_t = \frac{1}{T} ((T-t) \omega_0 + t \hat{\omega}_T) \in [\omega_0] - t c_1(M).
\end{equation}
Observe that $\hat{\omega}_t$ is \emph{not} necessarily a metric, since $\hat{\omega}_T$ may have negative eigenvalues.
Write $\chi = \frac{1}{T} ( \hat{\omega}_T - \omega_0) = \ddt{} \hat{\omega}_t \in - c_1(M)$ and define $\Omega$ to be the volume form with
\begin{equation}
\ddbar \log \Omega = \chi \in - c_1(M), \qquad \int_M \Omega = \int_M \omega_0^n.
\end{equation}

We then consider the parabolic complex Monge-Amp\`ere equation
\begin{equation} \label{pcma4}
\ddt{ \varphi} = \log \frac{ (\hat{\omega}_t + \ddbar \varphi)^n}{\Omega}, \qquad \hat{\omega}_t+ \ddbar \varphi>0 , \qquad \varphi|_{t=0} =0.
\end{equation}
From (\ref{vfb2}) we immediately have:

\begin{lemma} \label{dphi2} For a uniform constant $C$ we have on $M \times [0,T)$,
\begin{equation}
\dot{\varphi} \le C.
\end{equation}
\end{lemma}

If we assume that $\hat{\omega}_T \ge 0$ then the next result shows that the potential $\varphi$ is bounded  \cite{TZha} (see also \cite{SW2}).  Note that since $[\omega_0] -T c_1(M)$ is on the boundary of the K\"ahler cone, one would expect in many cases that this class contains a nonnegative representative $\hat{\omega}_T$.

\begin{proposition} Assume that $\hat{\omega}_T$ is nonnegative.  Then for a uniform constant $C$ we have on $M \times [0,T)$,
\begin{equation}
| \varphi| \le C.
\end{equation}
\end{proposition}
\begin{proof}
The upper bound of $\varphi$ follows from Lemma \ref{dphi2}.  Alternatively,  use the same argument as in the upper bound of $\varphi$ in Lemma \ref{lemmaphibound}.  For the lower bound, observe that
\begin{equation} \label{otn}
\hat{\omega}_t^n = \frac{1}{T^n} \left( (T-t) \omega_0 + t \hat{\omega}_T\right)^n \ge \frac{1}{T^n} (T-t)^n \omega_0^n \ge c_0 (T-t)^n \Omega,
\end{equation}
for some uniform constant $c_0 >0$.  Here we are using the fact that $\hat{\omega}_T$ is nonnegative. Define
\begin{equation}
\psi = \varphi + n(T-t) ( \log (T-t)-1) - (\log c_0 -1)t,
\end{equation}
and compute
\begin{equation}
\ddt{\psi} = \log \frac{ (\hat{\omega}_t + \ddbar \varphi)^n}{\Omega} - n \log (T-t) - (\log c_0 -1).
\end{equation}
At a point where $\psi$ achieves a minimum in space we have $\ddbar \psi = \ddbar \varphi \ge 0$ and hence from (\ref{otn}),
\begin{equation}
\ddt{\psi} \ge \log (c_0(T-t)^n) - n \log (T-t) - (\log c_0 -1) =1.
\end{equation}
It follows from the minimum principle that $\psi$ cannot achieve a minimum after time $t=0$, and so $\psi$ is uniformly bounded from below.  Hence $\varphi$ is bounded from below.  \qed
\end{proof}

If $\hat{\omega}_T$ is the pull-back of a K\"ahler metric from another manifold via a holomorphic map (so in particular $\hat{\omega}_T \ge 0$), we have by the parabolic Schwarz lemma (Theorem \ref{psl}) a lower bound for $\omega(t)$:

\begin{lemma} \label{lemmapsa} Suppose there exists a holomorphic map $f: M \rightarrow N$ to a compact K\"ahler manifold $N$ and let $\omega_N$ be a K\"ahler metric on $N$.  We assume that
\begin{equation}
[\omega_0] - T c_1(M) = [f^* \omega_N].
\end{equation}
 Then on $M \times [0,T)$,
\begin{equation}
\omega \ge \frac{1}{C} f^* \omega_N,
\end{equation}
for a uniform constant $C$.
\end{lemma}
\begin{proof}  The method is similar to that of Lemma \ref{applyschwarz}.
We take $\hat{\omega}_T = f^* \omega_N \ge 0$.  Define $u = \tr{\omega}{f^* \omega_N}$.  We apply the maximum principle to the quantity
\begin{equation}
Q = \log u - A\varphi - An (T-t) (\log(T-t)-1),
\end{equation}
for $A$ to be determined later, and where we assume without loss of generality that $u>0$.
Compute using (\ref{psi})
\begin{align} \nonumber
\left( \ddt{} - \Delta \right) Q & \le C_0 u - A\dot{\varphi} + An \log(T-t) + A \tr{\omega}{(\omega- \hat{\omega}_t)} \\
& =  \tr{\omega}{(C_0 f^* \omega_N - (A-1) \hat{\omega}_t)} - A \log \frac{\omega^n}{\Omega (T-t)^n} - \tr{\omega}{\hat{\omega}_t} + An.
\end{align}
Now choose $A$ sufficiently large so that $(A-1) \hat{\omega}_t - C_0 f^*\omega_N \ge f^* \omega_N$ for all $t \in [0,T]$.  By the geometric-arithmetic means inequality, there exists a constant $c>0$ such that
\begin{equation}
\tr{\omega}{\hat{\omega}_t} \ge \frac{(T-t)}{T} \tr{\omega}{\omega_0} \ge c \left( \frac{(T-t)^n \Omega}{\omega^n} \right)^{1/n}.
\end{equation}
Then, arguing as in the proof of Lemma \ref{lemmatr},
\begin{align} \nonumber
\left( \ddt{} - \Delta \right) Q & \le
 - u + A \log \frac{(T-t)^n\Omega }{\omega^n} - c\left( \frac{(T-t)^n \Omega}{\omega^n} \right)^{1/n}  + An \le - u + C,
\end{align}
for a uniform constant $C$, since the map $\mu \mapsto A \log \mu - c \mu^{1/n}$ is uniformly bounded from above for $\mu>0$.
Hence at a maximum point of $Q$ we see that $u$ is bounded from above by $C$.  Since $\varphi$ and $(T-t) \log (T-t)$ are uniformly bounded this shows that $Q$ is uniformly bounded from above.  Hence $u$ is uniformly bounded from above.  \qed
\end{proof}

A natural question is:  when is the limiting class $[\omega_0] -T c_1(M)$ represented by the pull-back of a K\"ahler metric from another manifold via a holomorphic map?
 It turns out that this always occurs if the initial data is appropriately `algebraic'.

\begin{proposition}  Assume there exists a line bundle $L$ on $M$ such that $k [\omega_0] = c_1(L)$ for some positive integer $k$.  Then there exists a holomorphic map $f: M \rightarrow \mathbb{P}^N$ to some projective space $\mathbb{P}^N$ and
\begin{equation}
[\omega_0 ] - T c_1(M) = [f^* \omega],
\end{equation}
for some K\"ahler metric $\omega$ on $\mathbb{P}^N$.
\end{proposition}
\begin{proof}
We give a sketch of the proof.
Note that by the assumption on $L$, the manifold $M$ is a smooth projective variety.  From the Rationality Theorem of Kawamata and Shokurov \cite{KMM, KMori}, $T$ is rational.
The class $[\omega_0]-T c_1(M)$ is nef since it is the limit of K\"ahler classes.  From the Base Point Free Theorem (part (ii) of Theorem \ref{algebraic}), $[\omega_0] - Tc_1(M)$ is semi-ample, and the result follows. \qed
\end{proof}

If we make a further assumption on the map $f$ then we can  obtain $C^{\infty}$ estimates for the evolving metric away from a subvariety.

\begin{theorem} \label{awayfromsubvariety}
Suppose there exists a holomorphic map $f: M \rightarrow N$ to a compact K\"ahler manifold $N$ which is a biholomorphism outside a  subvariety $E \subset M$.   Let $\omega_N$ be a K\"ahler metric on $N$.  We assume that
\begin{equation}
[\omega_0] - T c_1(M) = [f^* \omega_N].
\end{equation}
 Then on any compact subset $K$ of $M \setminus E$ there exists a constant $c_K>0$ such that
 \begin{equation} \label{CK}
 \omega  \ge c_K \omega_0, \quad \textrm{on} \quad K \times [0,T).
 \end{equation}
Moreover we have uniform $C^{\infty}_{\emph{loc}}$ estimates for $\omega(t)$ on $M \setminus E$.
\end{theorem}
\begin{proof}  The inequality (\ref{CK}) is immediate from Lemma \ref{lemmapsa} and the fact that $f^*\omega_N$ is a K\"ahler metric on $M \setminus E$.  From the volume form bound (\ref{vfb2}), we immediately obtain uniform upper and lower bounds for $\omega$ on compact subsets of $M \setminus E$.  The higher order estimates follow from the same arguments as in Lemmas \ref{hoesigma} and \ref{hoesigma2}. \qed
\end{proof}

We will see in Section \ref{sectexamples} that the situation of Theorem \ref{awayfromsubvariety} arises in the case of blowing down an exceptional divisor.

\subsection{Behavior of the scalar curvature} \label{sectscalarcurvature}

In this section we give prove the following result of Zhang \cite{Zha3} on the behavior of the scalar curvature.    Given the estimates we have developed so far, we can give quite a short proof.  Recall that we have a lower bound of the scalar curvature from Theorem \ref{scalar}.

\begin{theorem} \label{zz}
Let $\omega=\omega(t)$ be a solution of the K\"ahler-Ricci flow (\ref{krffinite}) on the maximal time interval $[0,T)$.  If $T< \infty$ then
\begin{equation}
\limsup_{t \rightarrow T} \left(\sup_M R(g(t)) \right) = \infty. \label{Rinfty}
\end{equation}
\end{theorem}

In the case of the general Ricci flow with a singularity at time $T<\infty$ it is known that $\sup_M |\textrm{Ric}(g(t))| \rightarrow \infty$ as $t \rightarrow T$  \cite{Se}.

\begin{proof}[Proof of Theorem \ref{zz}]  We will assume that (\ref{Rinfty}) does not hold and obtain a contradiction.
Since we know from Theorem \ref{scalar} that the scalar curvature has a uniform lower bound, we may assume that $\| R(t) \|_{C^0(M)}$ is uniformly bounded for $t \in [0,T)$.  Let $\varphi$ solve the parabolic complex Monge-Amp\`ere equation (\ref{pcma4}). First note that
\begin{equation}
\left| \ddt{} \log \left( \frac{\omega^n}{\Omega} \right) \right| = | R| \le C.
\end{equation}
Integrating in time we see that  $|\dot{\varphi}| = |\log \frac{\omega^n}{\Omega} |$ is uniformly bounded.  Integrating in time again, we obtain a uniform bound for $\varphi$.  Define $H = t \dot{\varphi} - \varphi -nt$, which is a bounded quantity.  Then using (\ref{evolvephit}) we obtain (cf. (\ref{crucial})),
\begin{align}  \label{trick}
\left( \ddt{} - \Delta \right) H & = t\, \tr{\omega}{\chi} -n + \tr{\omega}{(\omega- \hat{\omega}_t)}  = \tr{\omega}{(t \chi - \hat{\omega}_t)}
= - \tr{\omega}{\omega_0}.
\end{align}
Apply Proposition \ref{propChat} to see that
\begin{align} \label{psapp}
\left( \ddt{} - \Delta \right) \tr{\omega_0}{\omega} \le C_0 \tr{\omega}{\omega_0},
\end{align}
for a uniform constant $C_0$ depending only on  $\omega_0$.  Define
$Q = \log \tr{\omega_0}{\omega} + A H$ for $A=C_0+1$.  Combining (\ref{trick}) and (\ref{psapp}), compute
\begin{align}
\left( \ddt{} - \Delta \right) Q \le - \tr{\omega}{\omega_0}<0,
\end{align}
and hence by the maximum principle $Q$ is bounded from above by its value at time $t=0$.  It follows that $\tr{\omega_0}{\omega}$ is uniformly bounded from above.  Since we have a lower bound for $\dot{\varphi} = \log \frac{\omega^n}{\Omega}$, we see that for a uniform constant $C$,
\begin{equation}
\frac{1}{C} \omega_0 \le \omega \le C \omega_0, \quad \textrm{on} \quad M \times [0,T).
\end{equation}
Applying Corollary \ref{choe}, we obtain uniform estimates for $\omega(t)$ and all of its derivatives.  Hence $\omega(t)$ converges to a smooth K\"ahler metric $\omega(T)$ which is contained in $[\omega_0]-T c_1(M)$.  Thus $[\omega_0] - T c_1(M)$ is a K\"ahler class,  contradicting the definition of $T$.  \qed
\end{proof}

We remark that Theorem \ref{zz} can  be proved just as easily using the parabolic Schwarz lemma instead of Proposition \ref{propChat}.  Indeed one can replace $Q$ with $Q=\log \tr{\omega}{\omega_0} + AH$ and apply the Schwarz lemma with the holomorphic map $f$ being the identity map and $\omega_N = \omega_0$.  This was the method in \cite{Zha3}.  Also, one can find in \cite{Zha3} a different way of obtaining a contradiction, one which avoids the higher order estimates.

We finish this section by mentioning a couple of `folklore conjectures':

\begin{conjecture} \label{conjweak}
Let $\omega=\omega(t)$ be a solution of the K\"ahler-Ricci flow (\ref{krf}) on the maximal time interval $[0,T)$.  If $T< \infty$ then
\begin{equation}
R \le \frac{C}{T-t},
\end{equation}
for some uniform constant $C$.
\end{conjecture}

This conjecture has been established in dimension 1 by Hamilton and Chow \cite{Ch0, Hs} and by Perelman in higher dimensions if $[\omega_0] =  c_1(M)>0$  \cite{P4} (see also \cite{SeT}).  Perelman's  result makes use of the functionals he introduced in \cite{P1}.  In \cite{Zha3}, it was shown in a quite general setting, that $R \le C/(T-t)^2$.

A stronger version of Conjecture \ref{conjweak} is:

\begin{conjecture} \label{conjstrong}
Let $\omega=\omega(t)$ be a solution of the K\"ahler-Ricci flow (\ref{krf}) on the maximal time interval $[0,T)$.  If $T< \infty$ then
\begin{equation}
|\emph{Rm}| \le \frac{C}{T-t},
\end{equation}
for some uniform constant $C$.
\end{conjecture}

 Another way of saying this is that all finite time singularities along  the K\"ahler-Ricci flow are of \emph{Type I}.  This is related to a conjecture of Hamilton and Tian that the (appropriately normalized) K\"ahler-Ricci flow on a manifold with positive first Chern class converges to a K\"ahler-Ricci soliton, with a possibly different complex structure in the limit.

\subsection{Contracting exceptional curves} \label{sectexamples}

In this section we briefly describe, without proof, the example of \emph{blowing-down} \emph{exceptional curves} on a K\"ahler surface in finite time.  We begin by defining what is meant by a \emph{blowing-down} and \emph{blowing-up} (see for example \cite{GH}).

First, we define the \emph{blow-up} of the origin in $\mathbb{C}^2$.  Let $z^1, z^2$ be coordinates on $\mathbb{C}^2$, and let $U$ be a open neighborhood of the origin.  Define
\begin{equation}
\tilde{U} = \{ (z, \ell ) \in U \times \mathbb{P}^1 \ | \ z \in \ell \},
\end{equation}
where we are considering $\ell$ as a line in $\mathbb{C}^2$ through the origin.  One can check that $\tilde{U}$ is a 2-dimensional complex submanifold of $U \times \mathbb{P}^1$.   There is a holomorphic map $\pi: \tilde{U} \rightarrow U$ given by $(z, \ell) \mapsto z$ which maps $\tilde{U} \setminus \pi^{-1}(0)$ biholomorphically onto $U \setminus \{0 \}$.     The set
 $\pi^{-1}(0)$ is a $1$-dimensional submanifold of $\tilde{U}$, isomorphic to $\mathbb{P}^1$.

 Given a point $p$ in a K\"ahler surface $N$ we can use local coordinates to construct the \emph{blow up} $\pi: M \rightarrow N$ of $p$, by replacing a neighborhood $U$ of $p$ with the blow up $\tilde{U}$ as above.  Thus $M$ is a K\"ahler surface  and $\pi$ a holomorphic map extending the local map given above.  Up to isomorphism, this construction is independent of choice of coordinates.  The curve $E=\pi^{-1}(p)$ is called the \emph{exceptional  curve}.  Since $\pi(E)=p$, the map $\pi$ contracts or \emph{blows down} the curve $E$.   Moreover, $\pi$ is an isomorphism from $M \setminus E$ to $N \setminus \{p \}$.  From the above we see that $E$ is a smooth curve which is isomorphic to $\mathbb{P}^1$.  Moreover, the reader can check that it satisfies $E \cdot E =-1$.

 Conversely, given a curve $E$ on a surface $M$ with these properties we can define a map blowing down $E$.  More precisely, we define an irreducible curve $E$ in $M$ to be a \emph{$(-1)$-curve} if it is smooth, isomorphic to $\mathbb{P}^1$ and has $E \cdot E = -1$.  If $M$ admits a $(-1)$-curve $E$ then there exists a holomorphic map $\pi: M \rightarrow N$ to a smooth K\"ahler surface $N$ and a point $p \in Y$ such that $\pi$ is precisely the blow down of $E$ to $p$, as constructed above.  Note that if $E$ is a $(-1)$ curve then by the Adjunction Formula, $K_E \cdot E =-1$.

The main result of \cite{SW2} says that, under appropriate hypotheses on the initial K\"ahler class, the K\"ahler-Ricci flow will blow down $(-1)$-curves on $M$ and then continue on the new manifold.  To make this more precise, we need a definition.

\begin{definition} \label{defcan}
We say that the solution $g(t)$ of the K\"ahler-Ricci flow (\ref{krffinite}) on a compact K\"ahler surface $M$ performs a  {\bf  canonical surgical contraction} if the following holds.  There exist distinct $(-1)$ curves $E_1, \ldots, E_k$ of $M$, a compact K\"ahler surface $N$ and a blow-down map $\pi: M \rightarrow N$ with  $\pi(E_i) = y_i \in N$ and
$\pi|_{M \setminus \bigcup_{i=1}^k E_i}$  a biholomorphism onto $N \setminus \{ y_1, \ldots, y_k \}$ such that:
\begin{enumerate}
\item[(i)]  As $t \rightarrow T^-$,   the metrics $g(t)$ converge to a smooth K\"ahler metric $g_T$ on $M \setminus \bigcup_{i=1}^k E_i$ smoothly on compact subsets of $M \setminus \bigcup_{i=1}^k E_i.$

\item[(ii)] $(M, g(t))$ converges to a unique compact metric space $(\hat N, d_T)$ in the Gromov-Hausdorff sense as $t\rightarrow T^-$. In particular, $(\hat N, d_T)$ is homeomorphic to the K\"ahler surface $N$.

\item[(iii)] There exists a unique maximal smooth solution $g(t)$ of the K\"ahler-Ricci flow on $N$ for $t\in (T, T_N)$, with $T< T_N \le \infty$, such that $g(t)$ converges to $(\pi^{-1})^*g_N$ as $t  \rightarrow T^+$ smoothly on compact subsets of $N \setminus \{ y_1, \ldots, y_k\}$.

\item[(iv)] $(N, g(t))$ converges to $(N, d_T)$ in the Gromov-Hausdorff sense as $t\rightarrow T^+$.

\end{enumerate}
\end{definition}

The following theorem is proved in \cite{SW2}.  It essentially says that whenever the evolution of the K\"ahler classes along the K\"ahler-Ricci flow indicate that a blow down should occur at the singular time $T< \infty$, then the K\"ahler-Ricci flow carries out  a canonical surgical contraction at time $T$.

\begin{theorem} \label{thmblowup} Let $g(t)$ be a smooth solution of the K\"ahler-Ricci flow (\ref{krffinite}) on a K\"ahler surface $M$ for $t$ in $[0,T)$ and  assume $T<\infty$. Suppose there exists a blow-down map $\pi:  M \rightarrow N$ contracting disjoint $(-1)$ curves $E_1, \ldots, E_k$ on $M$ with $\pi(E_i) = y_i \in N$, for a smooth compact K\"ahler surface $(N, \omega_N)$   such that the limiting K\"ahler class satisfies
\begin{equation} \label{assumption1}
[\omega_0] - T c_1(M) = [\pi^*\omega_N].
\end{equation}
Then the K\"ahler-Ricci flow $g(t)$ performs a canonical surgical contraction with respect to the data $E_1, \ldots, E_k$, $N$ and $\pi$.
\end{theorem}

Note that from Theorem \ref{awayfromsubvariety}, we have $C^{\infty}_{\textrm{loc}}$ estimates for $g(t)$ on $M \setminus \bigcup_{i=1}^k E_i$, and thus part (i) in the definition of canonical surgical contraction follows immediately.  For the other parts, estimates are needed for $g(t)$ near the subvariety $E$.  To continue the flow on the new manifold, some techniques are adapted from \cite{SoT3}.  We refer the reader to \cite{SW2} for the details.  In fact, the same result is shown to hold in \cite{SW2} for blowing up points in higher dimensions, and in \cite{SW3} the results are extended to the case of an exceptional divisor $E$ with normal bundle $\mathcal{O}(-k)$, which blows down to an orbifold point. See also \cite{LT} for a different approach to the study of blow-downs.

In Section \ref{sectkahlersurfaces}, we will show how Theorem \ref{thmblowup} can be applied quite generally for the K\"ahler-Ricci flow on a K\"ahler surface.

\subsection{Collapsing in finite time} \label{sectfinitecollapsing}

In this section, we briefly describe, again without proof, another example of a finite time singularity.

Let $M$ be a projective bundle over a smooth projective variety $B$.  That is, $M = \mathbb{P}(E)$, where $\pi: E \rightarrow B$ is a holomorphic vector bundle which we can take to have rank $r$.  Write $\pi$ also for the map $\pi: M \rightarrow B$.  Of course, the simplest example of this would be a product $B \times \mathbb{P}^{r-1}$.  We consider the K\"ahler-Ricci flow (\ref{krffinite}) on $M$.  The flow will always develop a singularity in finite time.  This is because
\begin{equation}
\int_F (c_1(M))^{r-1} >0,
\end{equation}
for any fiber $F$, whereas if $T=\infty$ then $\frac{1}{t}[\omega_0] - c_1(M)>0$ for all $t>0$.  The point is that the fibers $F \cong \mathbb{P}^{r-1}$ must shrink to zero in finite time along the K\"ahler-Ricci flow.

In \cite{SSW}, it is shown that:

\begin{theorem} \label{theoremSSW}
Assume that
\begin{equation}
[\omega_0] - T c_1(M) = [\pi^* \omega_B],
\end{equation}
for some K\"ahler metric $\omega_B$ on $B$.  Then there exists a sequence of times $t_i \rightarrow T$ and a distance function $d_B$ on $B$, which is uniformly equivalent to the distance function induced by $\omega_B$, such that $(M, \omega(t_i))$ converges to $(B, d_B)$ in the Gromov-Hausdorff sense.
\end{theorem}

Note that from Lemma \ref{lemmapsa} we immediately have $\omega(t) \ge \frac{1}{C} \pi^* \omega_B$ for some uniform $C>0$.  The key estimates proved in \cite{SSW} are:
\begin{enumerate}
\item[(i)] $\displaystyle{\omega(t) \le C \omega_0}$.
\item[(ii)] $\displaystyle{ \textrm{diam}_{\omega(t)} F \le C(T-t)^{1/3}}$, for every fiber $F$.
\end{enumerate}

Thus we see that the metrics are uniformly bounded from above along the flow and the fibers collapse.  Given (i) and (ii) it is fairly straightforward to establish Theorem \ref{theoremSSW}.  We refer the reader to \cite{SSW} for the details.

The following conjectures are made in \cite{SSW}:

\begin{conjecture}
With the assumptions above:
\begin{enumerate}
\item[(a)] There exists unique distance function $d_B$ on $B$ such that $(M, \omega(t))$ converges in the Gromov-Hausdorff sense to $(B, d_B)$, without taking subsequences.
\item[(b)] The estimate (ii) above can be strengthened to $\displaystyle{ \textrm{diam}_{\omega(t)} F \le C(T-t)^{1/2}}$, for every fiber $F$.
\item[(c)]  Theorem \ref{theoremSSW} (and parts (a) and (b) of this conjecture) should hold more generally for a bundle $\pi: M \rightarrow B$ over a K\"ahler base $B$ with  fibers $\pi^{-1}(b)$ being Fano manifolds admitting metrics of nonnegative bisectional curvature.
\end{enumerate}
\end{conjecture}

We end this section by describing an example which illustrates both the case of contracting an exceptional curve and the case of collapsing the fibers of a projective bundle.  Let $M$ be the blow up of $\mathbb{P}^2$ at one point $p\in \mathbb{P}^2$.  Let $f: M \rightarrow \mathbb{P}^2$ be the map blowing down the exceptional curve $E$ to the point $p$.  To see the bundle structure on $M$, note that the blow-up of $\mathbb{C}^2$ at the origin  can be identified with $M \setminus f^{-1}(H)$ for $H$ a hyperplane in $\mathbb{P}^2$.   We have a map $\pi$ from the blow up of $\mathbb{C}^2$, which is $ \{ (z, \ell) \in \mathbb{C}^2 \times \mathbb{P}^1 \ | \ z \in \ell \}$,  to $\mathbb{P}^1$ given by projection onto the second factor.  This extends to a holomorphic bundle map $\pi: M \rightarrow \mathbb{P}^1$ which has $\mathbb{P}^1$ fibers.  We refer the reader to \cite{C2, SW1} for more details.

Writing $\omega_1$ and $\omega_2$ for the Fubini-Study metrics on $\mathbb{P}^1$ and $\mathbb{P}^2$ respectively, we see that
every K\"ahler class $\alpha$ on $M$ can be written as a linear combination $\alpha = \beta [\pi^* \omega_1] + \gamma [f^* \omega_2]$ for $\beta, \gamma>0$.  The boundary of the K\"ahler cone is spanned by the two rays $\mathbb{R}^{\ge 0} [\pi^* \omega_1]$ and $\mathbb{R}^{\ge 0} [f^* \omega_2]$.  The first Chern class of $M$  is given by
\begin{equation}
c_1(M) =  [\pi^* \omega_1] + 2 [f^* \omega_2]>0.
\end{equation}
Hence if the initial K\"ahler metric $\omega_0$ is in the cohomology class $\alpha_0 = \beta_0 [\pi^* \omega_1] + \gamma_0 [f^* \omega_2]$ then the solution $\omega(t)$ of the K\"ahler-Ricci flow (\ref{krffinite}) has cohomology class
\begin{equation}
[\omega(t)] = \beta(t) [\pi^* \omega_1] + \gamma(t) [f^* \omega_2], \quad \textrm{with} \quad \beta (t) = \beta_0 - t, \ \gamma(t) = \gamma_0 -2t.
\end{equation}
There are three different behaviors of the K\"ahler-Ricci flow according to whether the cohomology class $[\omega(t)]$ hits the boundary of the K\"ahler cone at a point on $\mathbb{R}^{>0}[\pi^* \omega_1]$, at a point on $\mathbb{R}^{>0}[f^* \omega_2]$ or at zero.  Namely:
\begin{enumerate}
\item[(i)]  If $\gamma_0 > 2 \beta_0$ then a singularity occurs at time $T = \beta_0$ and
\begin{equation}
[\omega_0] - T c_1(M) = \gamma(T) [f^* \omega_2], \quad \textrm{with } \gamma(T) = \gamma_0-2\beta_0>0.
\end{equation}
Thus we are in the case of Theorem \ref{thmblowup} and the K\"ahler-Ricci flow performs a canonical surgical contraction at time $T$.
\item[(ii)]  If $\gamma_0< 2\beta_0$ then a singularity occurs at time $T = \gamma_0/2$ and
\begin{equation}
[\omega_0]- Tc_1(M) = \beta(T) [\pi^* \omega_1], \quad \textrm{with } \beta(T) = \beta_0 - \gamma_0/2 >0.
\end{equation}
Thus we are in the case of Theorem \ref{theoremSSW} and the K\"ahler-Ricci flow will collapse the $\mathbb{P}^1$ fibers and converge in the Gromov-Hausdroff sense, after passing to a subsequence, to a metric on the base $\mathbb{P}^1$.
\item[(iii)]  If $\gamma_0=2\beta_0$ then the cohomology class changes by a rescaling.  It was shown by Perelman \cite{P4, SeT} that $(M, \omega(t))$ converges in the Gromov-Hausdorff sense to a point.
\end{enumerate}

The behavior of the K\"ahler-Ricci flow on this manifold $M$, and higher dimensional analogues, was analyzed in detail by Feldman-Ilmanen-Knopf \cite{FIK}.  They constructed self-similar solutions of the K\"ahler-Ricci flow through such singularities (see also \cite{Cao3}) and carried out a careful study of their properties.  Moreover, they posed a number of conjectures, some of which were established in \cite{SW1}.

Indeed if we make the assumption that the initial metric $\omega_0$ is invariant under a maximal compact subgroup of the automorphism group of $M$, then stronger results than those given in Theorems \ref{thmblowup} and \ref{theoremSSW} were obtained in \cite{SW1}.  In particular,  in the situation of case (ii), it was shown in \cite{SW1} that $(M, \omega(t))$ converges in the Gromov-Hausdorff sense (without taking subsequences) to a multiple of the Fubini-Study metric on $\mathbb{P}^1$ (see also \cite{Fo}).

One can see from the above some general  principles for what we expect with the K\"ahler-Ricci flow.  Namely, the behavior of the flow ought to be able to be read from the behavior of the cohomology classes $[\omega(t)]$ as $t$ tends to the singular time $T$.  If the limiting class $[\omega_0] - T c_1(M) = [\pi^* \omega_N]$ for some $\pi: M \rightarrow N$ with $\omega_N$ K\"ahler on $N$, then we expect geometric convergence of $(M, \omega(t))$ to $(N, \omega_N)$ in some appropriate sense.  This philosophy was discussed by Feldman-Ilmanen-Knopf \cite{FIK}.

\pagebreak

\section{The K\"ahler-Ricci flow and the analytic minimal model program} \label{sectlast}

In this section, we begin by discussing, rather informally, some of the basic ideas behind the minimal model program (MMP) with scaling.  Next we discuss the program of Song-Tian relating this to the  K\"ahler-Ricci flow.  Finally, we describe the case of K\"ahler surfaces.

\subsection{Introduction to the minimal model program with scaling}

In this section, we give a brief introduction of Mori's minimal model program (MMP) in birational geometry.  For more extensive references on this subject, see \cite{CKL, Deb, KMM, KMori}, for example.  We also refer the reader to  \cite{Si2} for a different analytic approach to some of these questions.

We begin with a definition.  Let $X$ and $Y$ be projective varieties.  A \emph{rational map} from $X$ to $Y$ is given by a holomorphic map $f: X \setminus V \rightarrow Y$, where $V$ is a subvariety of $X$.  We identify two such maps if they agree on $X - W$ for some subvariety $W$.  Thus a rational map is really an equivalence class of pairs $(f_U, U)$ where $U$ is the complement of a variety in $X$ (i.e. a Zariski open subset of $X$) and $f_U: U \rightarrow Y$ is a holomorphic map.

We say that a rational map $f$ from $X$ to $Y$ is \emph{birational} if there exists a rational map from $Y$ to $X$ such that $f \circ g$ is the identity as  a rational map.  If a birational map from $X$ to $Y$ exists then we say that $X$ and $Y$ are \emph{birationally equivalent} (or \emph{birational} or \emph{in the same birational class}).

Although birational varieties agree only on a dense open subset, they share many properties (see e.g. \cite{GH, Ha}).  The minimal model program is concerned with finding a `good' representative of a variety within its birational class.  A `good' variety $X$ is one satisfying either:
\begin{enumerate}
\item[(i)]  $K_X$ is nef; \ or
\item[(ii)]  There exists a holomorphic map $\pi: X \rightarrow Y$ to a lower dimensional variety $Y$ such that the generic fiber $X_y= \pi^{-1}(y)$ is a manifold  with $K_{X_y}<0$.\end{enumerate}

In the first case, we say that $X$ is a \emph{minimal model} and in the second case we say that $X$ is a \emph{Mori fiber space} (or \emph{Fano fiber space}).  Roughly speaking,  since $K_X$ nef can be thought of as a  `nonpositivity' condition on $c_1(X) = [\Ric(\omega)]$, (i) implies that $X$ is `nonpositively curved' in some weak sense.  Condition (ii) says rather that $X$ has a `large part' which is `positively curved'.  The two cases (i) and (ii) are mutually exclusive.

The basic idea of the MMP is to find a finite sequence of birational maps $f_1, \ldots, f_k$ and varieties $X_1, \ldots, X_k$,
\begin{equation}
\begin{diagram}\label{diag}
\node{X=X_0} \arrow[2]{e,t,..}{f_1} \node[2]{X_1} \arrow[2]{e,t,..}{f_2} \node[2]{X_2} \arrow[2]{e,t,..}{f_2}  \node[2]{\ldots} \arrow[2]{e,t,..}{f_k} \node[2]{X_k}
\end{diagram}
\end{equation}
so that $X_k$ is our `good' variety: either of type (i) or type (ii).   Recall that $K_X$ nef means that $K_X \cdot C \ge 0$ for all curves $C$.    Thus we want to find maps $f_i$ which `remove' curves $C$ with $K_X \cdot C <0$, in order to make the canonical bundle `closer' to being nef.

If the complex dimension is 1 or 2, then we can carry this out in the category of smooth varieties.
In the case of complex dimension 1, no birational maps are needed and case (i) corresponds to $c_1(X) <0$ or $c_1(X)=0$ while case (ii) corresponds to $X= \mathbb{P}^1$.  Note that in case (i), $X$ admits a metric of negative or zero curvature, while in case (ii) $X$ has a metric of positive curvature.

In complex dimension two, by the Enriques-Kodaira classification  (see \cite{BHPV}), we can obtain our `good' variety $X$  via a finite sequence of blow downs (see Section \ref{sectkahlersurfaces}).

Unfortunately, in dimensions three and higher, it is not possible to find such a sequence of birational maps if we wish to stay within the category of smooth varieties.  Thus to carry out the minimal model program, it is necessary to consider varieties with  singularities.  This leads to all kind of complications, which go well beyond the scope of these notes.
  For the purposes of this discussion, we will restrict ourselves to smooth varieties except where it is absolutely impossible to avoid mentioning singularities.

We need some further definitions.  Let $X$ be a smooth projective variety.
As we have discussed in Section \ref{sectnotion}, there is a natural pairing  between divisors and curves.  A \emph{1-cycle} $C$ on $X$ is a formal finite sum $C = \sum_i a_i C_i$, for $a_i \in \mathbb{Z}$ and $C_i$ irreducible curves.  We say that 1-cycles $C$ and $C'$ are \emph{numerically equivalent} if $ D \cdot C = D \cdot C'$ for all divisors $D$, and in this case we write $C\sim C'$.  We denote by $N_1(X)_{\mathbb{Z}}$ the space of $1$-cycles modulo numerical equivalence.  Write
\begin{equation}
N_1(X)_{\mathbb{Q}} = N_1(X)_{\mathbb{Z}} \otimes_{\mathbb{Z}} \mathbb{Q} \quad \textrm{and} \quad N_1(X)_{\mathbb{R}} = N_1(X)_{\mathbb{Z}} \otimes_{\mathbb{Z}} \mathbb{R}.
\end{equation}
Similarly, we say that divisors $D$ and $D'$ are \emph{numerically equivalent} if $D \cdot C= D' \cdot C$ for all curves $C$.  Write $N^1(X)_{\mathbb{Z}}$ for the set of divisors modulo numerical equivalence.  Define $N^1(X)_{\mathbb{Q}}$, $N^1(X)_{\mathbb{R}}$ similarly.  One can check that $N_1(X)_{\mathbb{R}}$ and $N^1(X)_{\mathbb{R}}$ are vector spaces of the same (finite) dimension.  In the obvious way, we can talk about $1$-cycles with coefficients in $\mathbb{Q}$ or $\mathbb{R}$ (and correspondingly, $\mathbb{Q}$- or $\mathbb{R}$-divisors) and we can talk about numerical equivalence of such objects.

Within the vector space $N_1(X)_{\mathbb{R}}$ is a cone $NE(X)$ which we will now describe.
We say that an element of $N_1(X)_{\mathbb{R}}$ is \emph{effective} if it is numerically equivalent to a $1$-cycle of the form
 $C = \sum_i a_i C_i$ with $a_i \in \mathbb{R}^{\ge 0}$ and $C_i$ irreducible curves.  Write $NE(X)$ for the cone of effective elements of $N_1(X)_{\mathbb{R}}$, and write $\overline{NE(X)}$ for its closure in the vector space $N_1(X)_{\mathbb{R}}$.  The importance of $\ov{NE(X)}$ can be seen immediately from the following theorem, known as \emph{Kleiman's criterion}:

 \begin{theorem} \label{kleiman}
 A divisor $D$ is ample if and only if $D \cdot w >0$ for all nonzero $w \in \overline{NE(X)}$.
 \end{theorem}

 We can now begin to describe the \emph{MMP with scaling} of \cite{BCHM}.   This is an algorithm for finding a specific sequence of birational maps $f_1, \ldots f_k$.
   First, choose an ample divisor $H$ on $X$.  Then define
 \begin{equation} \label{Tdefn2}
 T = \sup \{ t >0 \ | \ H +tK_X >0 \}.
 \end{equation}
 If $T= \infty$, then we have nothing to show since $K_X$ is already nef and we are in case (i).  Indeed, if $C$ is any curve in $X$ then
 \begin{equation}
 K_X \cdot C = \frac{1}{t} (H+tK_X)\cdot C  -\frac{1}{t} H\cdot C \ge - \frac{1}{t} H \cdot C \rightarrow 0 \quad \textrm{as} \quad t \rightarrow \infty.
 \end{equation}
 We can assume then that $T<\infty$.  We can apply the Rationality Theorem of Kawamata and Shokurov \cite{KMM, KMori} to see that $T$ is rational, and hence $H+TK_X$ defines a $\mathbb{Q}$-line bundle.

 Next we apply the Base Point Free Theorem (part (ii) of Theorem \ref{algebraic}) to $L=H+T K_X$ to see that for sufficiently large $m\in \mathbb{Z}^{\geq 0}$,  $L^m$ is globally generated and $H^0(X, L^m)$ defines a holomorphic map $\pi: X \rightarrow \mathbb{P}^N$ such that $L^m = \pi^* \mathcal{O}(1)$. We write $Y$ for the image of $\pi$.  This variety $Y$ is uniquely determined for $m$ sufficiently large.  The next step is to establish properties of this map $\pi$.

 Define a subcone $NE(\pi)$ of $\overline{NE(X)}$ by
\begin{equation}
NE(\pi) = \{ w \in \ov{NE(X)} \ | \ L \cdot w =0 \},
\end{equation}
which is nonempty by  Theorem \ref{kleiman}.  We now make the following:

\bigskip
\noindent
{\bf Simplifying assumption:} \ $NE(\pi)$ is an \emph{extremal ray} of $\overline{NE(X)}$.
\bigskip

A \emph{ray} $R$  of $\overline{NE(X)}$ is a subcone of the form $R=\{ \lambda w \ | \ \lambda \in [0, \infty) \}$ for some $w \in \overline{NE(X)}$.  We say that a subcone $C$ in $\ov{NE(X)}$ is \emph{extremal} if $a, b \in \overline{NE(X)}$, $a+b \in C$ implies that $a, b \in C$.  In general,  $NE(\pi)$ is an extremal subcone but not necessarily a ray.  However, it is expected that it will be an extremal ray for generic choice of initial ample divisor $H$ (see the discussion in \cite{SoT3}).

\begin{remark} \emph{
In the case that $NE(\pi)$ is not an extremal ray, one can still continue the MMP with scaling by applying Mori's Cone Theorem \cite{KM} to find such an extremal ray contained in $NE(\pi)$. }
 \end{remark}

The extremal ray $R=NE(\pi)$ has the additional property of being \emph{$K_X$-negative}.  We say that a ray is \emph{$K_X$-negative} if $K_X \cdot w <0$ for all nonzero $w$ in the ray.   Clearly this is true in this case since $0=L \cdot w = H \cdot w + T K_X \cdot w$ and therefore $K_X \cdot w = - T^{-1} H \cdot w <0$ if $w$ is a nonzero element of $R$.  Thus from the point of view of the minimal model program, $R$ contains  `bad' curves (those with negative intersection with $K_X$) which  we want to remove.

Moreover,  the map $\pi$ contracts all curves whose class lies in the extremal ray $R = NE(\pi)$.  The union of these curves is called the \emph{locus} of $R$.
  In fact, the locus of $R= NE(\pi)$ is exactly the set of points where the map $\pi: X \rightarrow Y$ is not an isomorphism.  Moreover, $R$  is a subvariety of $X$ \cite{Deb, KM}.
 There are  three cases:

 \bigskip
 \noindent
 {\bf Case 1.}  \  The locus of $R$ is equal to $X$.  In this case $\pi$ is a fiber contraction and $X$ is a Mori fiber space.

 \bigskip
 \noindent
 {\bf Case 2.} \ The locus of $R$ is an irreducible divisor $D$.  In this case $\pi$ is called a \emph{divisorial contraction}.

 \bigskip
 \noindent
 {\bf Case 3.} \ The locus of $R$ has codimension at least 2.  In this case, $\pi$ is called a \emph{small contraction}.

\bigskip

The process of the MMP with scaling is then as follows:
if we are in case 1, we stop, since $X$ is already of type (ii).  In case 2 we have a map $\pi: X \rightarrow Y \subset \mathbb{P}^N$ to a subvariety $Y$.  Let $H_Y$ on $Y$ be restriction of $\mathcal{O}(1)|_{Y}$.  We can then repeat the process of the minimal model program with scaling with $(Y, H_Y)$ instead of $(X, H)$.

The serious difficulties occur in case 3.  Here the image $Y$ of $\pi$ will have very bad singularities and it will not be possible to continue this process on $Y$.  Instead we have to work on a new space given by a procedure known as a \emph{flip}.
Let $\pi: X \rightarrow Y$ be a small contraction as in case 3.  The \emph{flip} of $\pi: X \rightarrow Y$ is a variety $X^+$ together with a  holomorphic birational map $\pi^+: X^+ \rightarrow Y$ satisfying the following conditions:
\begin{enumerate}
\item[(a)]   The \emph{exceptional locus} of $\pi^+$  (that is, the set of points in $X^+$ on which $\pi^+$ is not an isomorphism)  has codimension strictly larger than 1.
\item[(b)]  If $C$ is a curve contracted by $\pi^+$ then $K_{X^+} \cdot C >0$.
\end{enumerate}
Thus we have a diagram
 \begin{equation}
\begin{diagram}\label{diag15}
\node{X} \arrow{se,b,}{\pi}  \arrow[2]{e,t,..}{(\pi^+)^{-1}\circ\pi }     \node[2]{X^+} \arrow{sw,r}{\pi^+} \\
\node[2]{Y}
\end{diagram}
\end{equation}
The composition $(\pi^+)^{-1} \circ \pi$ is a birational map from $X$ to $X^+$, and is also sometimes called a \emph{flip}.  In going from $X$ to $Y$ we have contracted curves $C$ with $K_X \cdot C <0$.  The point of (b) in the definition above is that in going from $Y$ to $X^+$ we do not wish to `gain' any curves $C$ of negative intersection with the canonical bundle.  The process of the flip replaces curves $C$ on $X$ with $K_X \cdot C<0$ with curves $C'$ on $X^+$ with $K_{X^+} \cdot C' >0$.  This fits into the strategy of trying to make the canonical bundle `more nef'.

Given a small contraction $\pi:X \rightarrow Y$, the question of whether there actually exists a flip $\pi^+: X^+ \rightarrow Y$ is a difficult one.  It has  been established for the MMP with scaling \cite{BCHM, HM}.  Returning now to the MMP with scaling:  if we are in case 3 we replace $X$ by its flip $X^+$ and we denote by $L^+$ the strict transform of $\mathcal{O}(1)|_Y$ via  $\pi^+$ (see for example \cite{Ha}).  We can now repeat the process with $(X^+, H^+)$ instead of $(X, H)$.

We have described now the basic process of the MMP with scaling.  Start with $(X,H)$ and find $\pi: X \rightarrow Y$ contracting the extremal ray $R$ on which $H + T K_X$ is zero.  In case 1, we stop.  In case 2 we carry out a divisorial contraction and restart the process.  In case 3, we replace $X$ by its flip $X^+$ and again restart the process.  A question is now:  does this process terminate in finitely many steps?  It was proved in  \cite{BCHM, HM} that the answer to this is yes, at least in the case of varieties of general type.  If we have not already obtained a Mori fiber space, then the final variety $X_k$ contains no curves $C$ with $K_X \cdot C<0$, and we are done.

We conclude this section with an example of a   flip (see  \cite{Deb, SY}). Let $X_{m, n} = \mathbb{P} ( \mathcal{O}_{\mathbb{P}^n}  \oplus \mathcal{O}_{\mathbb{P}^n} (-1)^{\oplus (m+1)})$ be the $\mathbb{P}^{m+1}$ bundle over $\mathbb{P}^n$. Let $Y_{m,n}$ be the projective cone over $\mathbb{P}^m\times \mathbb{P}^n $ in $\mathbb{P}^{(m+1)(n+1)}$ by the Segre embedding $$[Z_0, ..., Z_m]\times[W_0, ..., W_n]\rightarrow [Z_0W_0, ..., Z_iW_j, ..., Z_mW_n]\in \mathbb{P}^{(m+1)(n+1)-1} .$$ Note that $Y_{m,n}=Y_{n,m}$. Then there exists a holomorphic map $\Phi_{m,n}: X_{m,n} \rightarrow Y_{m,n}$ for $m\geq 1$ contracting the zero section of $X_{m,n}$ of codimension $m+1$ to the cone singularity of $Y_{m,n}$. The following diagram gives a flip from $X_{m, n}$ to $X_{n, m}$ for $1\leq m<n$,
\begin{equation}
\begin{diagram}\label{diagflip}
\node{X_{m,n}} \arrow{se,b,}{\Phi_{m,n}}  \arrow[2]{e,t,..}{ \tilde\Phi }     \node[2]{X_{n,m}} \arrow{sw,r}{\Phi_{n,m}} \\
\node[2]{Y_{m,n}}
\end{diagram}.
\end{equation}














\subsection{The K\"ahler-Ricci flow and the MMP with scaling}

Let $X$ be a smooth projective variety with an ample divisor $H$.
We now relate the MMP with scaling to the (unnormalized) K\"ahler-Ricci flow
\begin{equation}\label{krfmmp}
\ddt{} \omega = - \Ric(\omega), \qquad \omega|_{t=0} = \omega_0,
\end{equation}

We assume that the initial metric $\omega_0$ lies in the cohomology class $c_1([H])$ associated to the divisor $H$.
As we have seen from Section \ref{sectpma}, a smooth solution $\omega(t)$ of the K\"ahler-Ricci flow exists precisely on the time interval $[0,T)$, with $T$ defined by
(\ref{Tdefn2}).  In general, we expect that as $t \rightarrow T$, the K\"ahler-Ricci flow carries out a `surgery', which is equivalent to the algebraic procedure of contracting an extremal ray, as discussed above.

  The following is a (rather sketchy) conjectural picture for the behavior of the K\"ahler-Ricci flow, as proposed by Song and Tian in \cite{SoT1, SoT3, T2}.

\bigskip
\noindent
\emph{Step 1.} \  We start with a metric $\omega_0$ in the class of a divisor $H$ on a variety $X$.
We then consider the solution $\omega(t)$ of the K\"ahler-Ricci flow (\ref{krfmmp}) on $X$ starting at $\omega_0$.  The flow exists on $[0,T)$ with $T = \sup \{ t >0 \ | \ H+ tK_X >0 \}$.

\bigskip
\noindent
\emph{Step 2.} \  If $T= \infty$, then $K_X$ is nef and the K\"ahler-Ricci flow exists for all time. The flow $\omega(t)$ should converge, after an appropriate normalization, to a canonical `generalized K\"ahler-Einstein metric' on $X$ as $t \rightarrow \infty$.

\bigskip
\noindent
\emph{Step 3.} \  If $T<\infty$, the K\"ahler-Ricci flow deforms $X$ to $(Y, g_Y)$ with a possibly singular metric $g_Y$ as $t\rightarrow T$.

\begin{enumerate}

\item[(a)] If $\dim X = \dim Y$ and $Y$ differs from $X$ by a subvariety of codimension $1$, then we return to Step 1, replacing $(X, g_0)$ by  $(Y, g_Y)$.

\item[(b)] If $\dim X = \dim Y$ and $Y$ differs from $X$ by a subvariety of codimension greater than $1$, we are in the case of a small contraction.  $Y$ will be singular.  By considering an appropriate notion of weak K\"ahler-Ricci flow on $Y$, starting at $g_Y$, the flow should immediately resolve  the singularities of $Y$ and replace $Y$ by its flip  $X^+$ (see \cite{SY}). Then we return to Step 1 with $X^+$.

\item[(c)] If $0< \dim Y < \dim X$, then we return to Step 1 with $(Y, g_Y$).

\item[(d)] If $\dim Y=0$, $X$ should have $c_1(X)>0$.  Moreover, after appropriate normalization, the solution $(X,\omega(t))$ of the K\"ahler-Ricci flow should deform to $(X', \omega')$ where $X'$ is possibly a different manifold and  $\omega'$ is either a K\"ahler-Einstein metric  or a K\"ahler-Ricci soliton (i.e. $\Ric(\omega') = \omega' + \mathcal{L}_V(\omega')$ for a holomorphic vector field $V$).  See the discussion after Conjecture \ref{conjstrong}.

\end{enumerate}

Namely, the K\"ahler-Ricci flow should construct the sequence of manifolds $X_1, \ldots, X_k$ of the MMP with scaling, with $X_k$ either nef (as in Step 2) or a Mori fiber space (as in Step 3, part (c) or (d)).  If we have a Mori fiber space, then we can continue the flow on the lower dimensional manifold $Y$, which would correspond to a lower dimensional MMP with scaling.  At the very last step, we expect the K\"ahler-Ricci flow to converge, after an appropriate normalization, to a canonical metric.


In \cite{SoT3}, Song-Tian constructed weak solutions for the K\"ahler-Ricci flow through the finite time singularities if the flips exist a priori. Such a weak solution is smooth outside the singularities of $X$ and the exceptional locus of the contractions and flips, and it is a nonnegative closed $(1,1)$-current with locally bounded potentials.  Furthermore, the weak solution of the K\"ahler-Ricci flow is unique.

In Step 2, when $T=\infty$, one can say more about the limiting behavior of the K\"ahler-Ricci flow. The abundance conjecture in birational geometry predicts that $K_X$ is semi-ample whenever it is nef. Assuming this holds,  the pluricanonical system $H^0(X, K_X^m)$ for sufficiently large $m$  induces a holomorphic map $\phi: X \rightarrow \Xcan$. $\Xcan$ is called the canonical model of $X$ and it is uniquely determined by the canonical ring of $X$.
 If we assume that $X$ is nonsingular and $K_X$ is semi-ample, then normalized solution $g(t)/t$ always converges weakly in the sense of distributions.  Moreover:

\begin{itemize}
\item If $\kod(X)=\dim X$, then $\Xcan$ is birationally equivalent to $X$ and the limit of $g(t)/t$ is the unique singular K\"ahler-Einstein metric on $\Xcan$ \cite{TZha, Ts1}.   If $X$ is a singular minimal model, we expect the K\"ahler-Ricci flow to converge to the singular K\"ahler-Einstein metric of Guedj-Eyssidieux-Zeriahi \cite{EGZ2}.
\item If $0<\kod(X)<\dim X$, then $X$ admits a Calabi-Yau fibration over $\Xcan$, and the limit of $g(t)/t$ is the unique generalized K\"ahler-Einstein metric (possibly singular) $g_{\textrm{can}}$ on $\Xcan$ defined by $\textrm{Ric}(g_{\textrm{can}}) = -g_{\textrm{can}} + g_{\textrm{WP}}$ away from a subvariety of $X_{\textrm{can}}$, where $g_{\textrm{WP}}$ is the Weil-Petersson metric  induced from the Calabi-Yau fibration of $X$ over $X_{\textrm{can}}$ \cite{SoT1, SoT2}.
\item If $\kod(X)=0$, then $X$ itself is a Calabi-Yau manifold and so the limit of $g(t)$ is the unique Ricci flat K\"ahler metric in its initial K\"ahler class  \cite{Cao, Y2} .
\end{itemize}

A deeper question to ask is whether such a weak solution is indeed a geometric solution of the K\"ahler-Ricci flow in the Gromov-Hausdorff topology.  One would like to show that the K\"ahler-Ricci flow performs geometric surgeries in Gromov-Hausdorff topology at each singular time and replaces the previous projective variety by a `better' model. Such a model is again a projective variety and the geometric surgeries coincide with the algebraic surgeries such as contractions and flips. If this picture holds, the K\"ahler-Ricci flow gives a continuous path from $X$ to its canonical model $\Xcan$ coupled with a canonical metric in the moduli space of Gromov-Hausdorff.
We can further ask: how does the curvature behave near the (finite) singular time?  Is  the singularity is always of Type I (see Conjecture \ref{conjstrong})?  Will the flow give a complete or compact shrinking soliton  after rescaling (cf. \cite{Cao3, FIK})?

\subsection{The K\"ahler-Ricci flow on K\"ahler surfaces} \label{sectkahlersurfaces}

In this section, we describe the behavior of the K\"ahler-Ricci flow on K\"ahler surfaces, and how it relates to the MMP.  For the purpose of this section, $X$ will be a K\"ahler surface.

We begin by discussing the minimal model program for surfaces.  It turns out to be relatively straightforward.  We obtain a sequence of smooth manifolds $X_1, \ldots, X_k$ and holomorphic  maps $f_1, \ldots, f_k$,
\begin{equation}
\begin{diagram}\label{diag2}
\node{X=X_0} \arrow[2]{e,t,..}{f_1} \node[2]{X_1} \arrow[2]{e,t,..}{f_2} \node[2]{X_2} \arrow[2]{e,t,..}{f_3}  \node[2]{\ldots} \arrow[2]{e,t,..}{f_k} \node[2]{X_k}
\end{diagram}
\end{equation}
with $X_k$ `minimal' in the sense described below.  Moreover, each of the maps $f_i$ is a blow down of a curve to a point.

 We say that a K\"ahler surface $X$ is a \emph{minimal surface} if it contains no $(-1)$-curve.  By the Adjunction Formula, a surface with $K_X$ nef is minimal.   On the other hand, a minimal surface may not have $K_X$ nef (an example is $\mathbb{P}^2$) and hence this definition of \emph{minimal surface} differs from the notion of `minimal model' discussed above.

The minimal model program for surfaces is simply as follows:  given a surface $X$, contract all the $(-1)$-curves to arrive at a minimal surface.  The Kodaira-Enriques classification can then be used to deduce that one either obtains a minimal surface with $K_X$ nef, or a minimal Mori fiber space.  A minimal Mori fiber space is either $\mathbb{P}^2$ or a \emph{ruled surface}, i.e. a $\mathbb{P}^1$ bundle over a Riemann surface (in the literature, sometimes a broader definition for ruled surface is used).    Dropping the minimality condition, Mori fiber spaces in dimension two
  are precisely those surfaces  birational to a ruled surface. Note that since $\mathbb{P}^2$ is birational to $\mathbb{P}^1 \times \mathbb{P}^1$, every  surface birational to $\mathbb{P}^2$ is birational to a ruled surface.

We wish to see whether the K\"ahler-Ricci flow on a K\"ahler surface will carry out this `minimal model program'.  The K\"ahler-Ricci flow should carry out the algebraic procedure of contracting $(-1)$-curves.    Recall that in Section \ref{sectexamples} we defined the notion of \emph{canonical surgical contraction}
for the K\"ahler-Ricci flow.

Starting at any K\"ahler surface $X$, we will use Theorem \ref{thmblowup}  to show that the K\"ahler-Ricci flow will always carry out a finite sequence of canonical surgical contractions until it either arrives at a minimal surface or the flow collapses the manifold.

\begin{theorem} \label{kahsurcon}
Let $(X, \omega_0)$ be a K\"ahler  surface with a smooth K\"ahler metric  $\omega_0$.  Then there exists a unique maximal K\"ahler-Ricci flow $\omega(t)$ on $X_0, X_1, \ldots, X_k$ with canonical surgical contractions  starting at $(X, \omega_0)$. Moreover, each canonical surgical contraction  corresponds to a blow-down $\pi: X_i \rightarrow X_{i+1}$ of a finite number of disjoint $(-1)$ curves on $X_i$.  In addition we have:
\begin{enumerate}
\item[(a)]  Either $T_k < \infty$ and the flow $\omega(t)$ collapses $X_k$, in the sense that
$$\emph{Vol}_{\omega(t)} X_k \rightarrow 0, \quad \textrm{as} \ \ t \rightarrow T_k^-.$$
Then $X_k$ is birational to  a ruled surface.

\item[(b)]  Or $T_k = \infty$ and $X_k$ has no $(-1)$ curves.
\end{enumerate}
\end{theorem}

\begin{proof}
Let $T_1$ be the first singular time.  If $T_1 = \infty$ then  $K_X$ is nef and hence $X$ has no $(-1)$-curves, giving case (b).

Assume then that $T_1<\infty$.   The limiting class at time $T_1$ is given by $\alpha = [\omega_0] - T_1 c_1(X)$.
Suppose that
\begin{equation}
\alpha^2 = \lim_{t \rightarrow T_1}( [\omega_0] - t c_1(X))^2 =  \lim_{t\rightarrow T_1} \textrm{Vol}_{g(t)} X >0,
 \end{equation}
so that we are not in case (a).   Thus the class $\alpha$ is nef and big.  On the other hand, $\alpha$ cannot be a K\"ahler class by Theorem \ref{longtime}.

We further notice that $\alpha+ \ve [\omega_0]$ is K\"ahler for all $\ve>0$ by Theorem \ref{longtime}. Then
\begin{equation}
\alpha\cdot (\alpha + \ve [\omega_0] ) = \alpha^2 + \ve \alpha \cdot [\omega_0] >0
\end{equation}
 if we choose $\ve>0$ sufficiently small.

We now apply the Nakai-Moishezon criterion for K\"ahler surfaces (Theorem \ref{nakai}) to see that there must exist an irreducible curve $C$ on $X$ such that $\alpha \cdot C=0$.
Let $\mathcal{E}$ be the space of all irreducible curves $E$ on $X$ with $\alpha \cdot E=0$.  Then $\mathcal{E}$ is non-empty and every  $E$ in  $\mathcal{E}$ has $E^2<0$ by the Hodge Index Theorem (Theorem \ref{index}).  Moreover, if $E \in \mathcal{E}$,
$$E\cdot K_X = \frac{1}{T_1} E\cdot (\alpha - [\omega_0]) = - \frac{1}{T_1} E\cdot [\omega_0]<0$$ since $[\omega_0]$ is K\"ahler. It then follows from the Adjunction formula (Theorem \ref{thmadj}) that $E$ must be a $(-1)$ curve.

We claim that if $E_1$ and $E_2$ are distinct elements of $\mathcal{E}$ then they must be disjoint.  Indeed, since $E_1$, $E_2$ are irreducible and distinct we have $E_1\cdot E_2 \ge 0$.
Moreover,
 $(E_1  + E_2 ) \cdot \alpha =0$ and applying the Hodge Index Theorem again, we see that $0> (E_1 + E_2)^2 = -2 + 2 E_1 \cdot E_2$, so that the only possibility is $E_1 \cdot E_2=0$, proving the claim.  It follows that
 $\mathcal{E}$ consists of finitely many disjoint $(-1)$ curves $ E_1, ..., E_k$.

Let $\pi: X\rightarrow Y$ be the blow-down map contracting $E_1, ..., E_k$ on $X$. Then $Y$ is again a smooth K\"ahler surface. Since $H^{1,1}(X, \mathbb{R})$ is generated by $H^{1,1}(Y, \mathbb{R})$ and the $c_1([E_i])$ for $i=1, ..., k$ (see for example Theorem I.9.1 in \cite{BHPV}), there exists $\beta\in H^{1,1}(X, \mathbb{R})$ and $a_i \in \mathbb{R}$ such that
\begin{equation}
\alpha = \pi^* \beta + \sum_{i=1}^k a_i c_1([E_i]).
\end{equation}
Since $\pi^* \beta\cdot E_i=0$ for each $i=1, ..., k$, we have  $\alpha \cdot E_i= a_i =0$ for all $i$ and hence $\alpha = \pi^*\beta. $

We claim that $\beta$ is a K\"ahler class on $Y$.  First,  for any curve $C$ on $Y$, we have
$ \beta \cdot C = \alpha \cdot \pi^* C>0$.  Moreover,  $\beta^2 = \alpha^2>0.$

By the Nakai-Moishezon criterion, it remains to show that $\beta \cdot \gamma>0$ for $\gamma$ some fixed K\"ahler class on $Y$.  Now $\beta \cdot \gamma = \alpha \cdot \pi^* \gamma = \lim_{t \rightarrow T^-} [\omega(t)] \cdot \pi^* \gamma \ge 0$.  Then put $\tilde{\gamma} = \gamma + c\beta$ for $c>0$.  If $c$ is sufficiently small then $\tilde{\gamma}$ is K\"ahler and since $\beta^2>0$ we have $\beta \cdot \tilde{\gamma} >0$, as required.

We now apply Theorem \ref{thmblowup} to see that the K\"ahler-Ricci flow performs a canonical surgical contraction.
We  repeat the above procedure until either the volume tends to $0$ or the flow exists for all time. This proves that either $T_k<\infty$ and $\textrm{Vol}_{g(t)} X_k \rightarrow 0$ as $t\rightarrow T_k^-$ or $T_k=\infty$ and $X_k$ has no $(-1)$ curves.

Finally, in the case (a), the theorem follows from Proposition \ref{volumecollapse} below.
\qed
\end{proof}

We make use of Enriques-Kodaira classification for complex surfaces (see \cite{BHPV}) to prove:

\begin{proposition}  \label{volumecollapse} Let $(X, \omega_0)$ be a K\"ahler  surface with a smooth K\"ahler metric  $\omega_0$.  Let $T$ be the first singular time of the K\"ahler-Ricci flow (\ref{krfmmp}). If $T<\infty$ and $\emph{Vol}_{g(t)} X \rightarrow 0$, as $t\rightarrow T$. Then $X$ is birational to  a ruled surface.  Moreover:

\begin{enumerate}

\item[(a)] Either there exists $C>0$ such that
\begin{equation}
C^{-1} \leq \frac{ \emph{Vol}_{g(t)} X}{ (T-t)^2 } \leq C,
\end{equation}
 and $X$ is a Fano surface (in particular, is birational to $\mathbb{P}^2$) and $\omega_0 \in T c_1(X)$.

\item[(b)] Or there exists $C>0$ such that
\begin{equation}
C^{-1} \leq \frac{ \emph{Vol}_{g(t)} X}{ T-t } \leq C.
\end{equation}
 If $X$ is Fano then $\omega_0$ is not in a multiple of $c_1(X)$.

\end{enumerate}

\end{proposition}

\begin{proof} We first show that $X$ is birational to a ruled surface. Suppose for a contradiction that $\kod(X)\geq 0$.  Then some multiple of $K_X$ has a global holomorphic section and hence is effective. 
 In particular, $([\omega_0] + T K_X) \cdot K_X \ge 0$, since $[\omega_0] + T K_X$ is a limit of K\"ahler classes.   Then
\begin{align} \nonumber
0 = ([\omega_0]+ TK_X)^2 & =  T ([\omega_0] + TK_X) \cdot K_X + ([\omega_0] + TK_X) \cdot [\omega_0] \\ &\ge  ([\omega_0] + TK_X) \cdot [\omega_0]  \ge 0,
\end{align}
which implies that $([\omega_0] + T K_X) \cdot [\omega_0] =0$.  Using the Index Theorem and the fact that $[\omega_0]^2>0$ and $([\omega_0]+T K_X)^2 = 0$ we have $[\omega_0] + TK_X =0$.  But this implies that $X$ is Fano, contradicting the assumption $\kod(X) \ge 0$.
Thus we have shown that $X$ must have $\kod(X)=-\infty$.   By the Enriques-Kodaira classification for complex surfaces which are K\"ahler (see chapter VI of \cite{BHPV}),  $X$ is birational to a ruled surface.

Since $\textnormal{Vol}_{g(t)} X = ([\omega_0] + t K_X)^2$ is a quadratic polynomial in $t$ which is positive for $t \in [0,T)$ and tends to zero as $t$ tends to $T$, we have
\begin{equation} \label{quadratic}
\textnormal{Vol}_{g(t)} X = [\omega_0]^2 + 2t [\omega_0] \cdot K_X + t^2 K_X^2 = C_1 (T-t) + C_2(T-t)^2,
\end{equation}
for constants $C_1 \ge 0$ and $C_2$.  First assume $C_1=0$.   Then $C_2 >0$ and we are in case (a).   From (\ref{quadratic}) we obtain
\begin{equation} \label{quadcons}
K_X^2 = C_2>0, \quad [\omega_0]^2 = K_X^2 T^2, \quad [\omega_0] \cdot K_X = - K_X^2 T <0.
\end{equation}
In particular,  $([\omega_0] + T K_X) \cdot [\omega_0]=0$ and hence by the Index Theorem, $[\omega_0] + T K_X =0$.  Thus  $X$ is Fano and $\omega_0 \in T c_1(X)$.  Note that by the classification of surfaces,
$X$ is either $\mathbb{P}^2$, $\mathbb{P}^1 \times \mathbb{P}^1$ or $\mathbb{P}^2$ blown-up at $k$ points for $1\leq k\leq 8$.

Finally, if $C_1 >0$ then we are in case (b).   If $[\omega_0]$ is a multiple of $c_1(X)$  then the volume $\textrm{Vol}_{g(t)} X$ tends to zero of order $(T-t)^2$, a contradiction.\qed
\end{proof}


We now discuss  the long time behavior of the K\"ahler-Ricci flow when we are in case (b) of Theorem \ref{kahsurcon}.  There are  three different behaviors of the K\"ahler-Ricci flow as $t \rightarrow \infty$ depending on whether
$X$ has Kodaira dimension equal to 0, 1 or 2:
\begin{itemize}
\item if $\kod(X)=0$, then the minimal model of $X$ is a Calabi-Yau surface with $c_1(X)=0$.  The flow $g(t)$ converges smoothly to a Ricci-flat K\"ahler metric as $t\rightarrow \infty$, as shown in Section \ref{sectpes}.
\item If $\kod(X)=1$, then
 $\frac{1}{t} g(t)$ converges in the sense of currents   to the pullback of the unique generalized K\"ahler-Einstein metric on the canonical model     of $X$ as $t\rightarrow \infty$ \cite{SoT1}.  A simple example of this is given in Section \ref{sectpes} in the case of a product elliptic surface.
\item If $\kod(X)=2$, $\frac{1}{t} g(t)$ converges in the sense of currents (and smoothly outside a subvariety) to the pullback of the unique smooth orbifold K\"ahler-Einstein metric on the canonical model of $X$ as $t\rightarrow \infty$ \cite{Kob,TZha,Ts2}.   In the case that $c_1(X)<0$, we showed in Section \ref{sectpes} that $\frac{1}{t} g(t)$ converges smoothly to a smooth K\"ahler-Einstein metric.
\end{itemize}

In fact, in the case when $T_k=\infty$, the scalar curvature of $\frac{1}{t} g(t)$ is uniformly bounded as $t\rightarrow \infty$ \cite{SoT4,Zha2}. Furthermore, if we assume that $X_k$ is a minimal surface of general type, and it admits only irreducible $(-2)$-curves, then $(X_k, \frac{1}{t} g(t))$ converges in Gromov-Hausdorff sense to its canonical model with the unique smooth orbifold K\"ahler-Einstein metric \cite{SW3}.


\bigskip
\bigskip

$^{*}$ Department of Mathematics \\
Rutgers University, Piscataway, NJ 08854\\

$^{\dagger}$ Department of Mathematics \\
Northwestern University, Evanston, IL 60208 \\

\end{document}